\documentclass[11pt]{article}
\pdfoutput=1
\usepackage[T1]{fontenc}
\usepackage[USenglish]{babel}
\usepackage{libertine,libertinust1math} 
\usepackage{microtype}

\usepackage{titling}
\usepackage[all]{nowidow}
\newcommand*{\ab}{\allowbreak}
\usepackage{amsmath}
\usepackage{amsthm}
\usepackage{mathtools}
\usepackage{mathicpatch}
\usepackage{bm} 
\usepackage{pifont}

\usepackage[inline]{enumitem}
\newlist{enumenun}{enumerate}{1}
\newlist{eninline}{enumerate*}{1}

\setlist[enumerate,itemize,enumenun]{
	itemsep=4pt plus 2pt minus 2pt,
	topsep=\itemsep,
	listparindent=\parindent, parsep=0pt,
	labelindent=\parindent,
	align=left,
	leftmargin=*,
	}
	
\setlist[enumenun]{
	nosep,
	label=\normalfont(\alph*),
	ref=(\alph*),
	}
	
\newcommand*{\lsf}{1.7} 
\newcommand*{\jf}{1.7} 
\newcommand*{\lef}{1.83} 

\newlength{\liststartpad}
\newlength{\itemjoin}
\newlength{\listendpad}

\setlist[eninline]{
	before={\unskip\hspace{\liststartpad}},
	itemjoin={\hspace{\itemjoin}},
	after={\hspace{\listendpad}\ignorespacesafterend},
	}

\usepackage[x11names, dvipsnames]{xcolor}
\usepackage{graphicx}
\usepackage{float}
\usepackage[font=small]{caption}
\usepackage{tikz}
\usetikzlibrary{calc}
\usetikzlibrary{arrows.meta}
\usetikzlibrary{backgrounds}

\newlength{\inxsep}
\newlength{\inysep}

\newlength{\et}
\newlength{\vxr}
\newlength{\vxrs}

\newcommand*{\co}{coordinate}

\tikzset{every picture/.style={line width=\et}}

\tikzset{vxe/.style={
		minimum size=2*\vxr, 
		circle, 
		line width=0pt,
		inner sep=0pt,
		outer sep=0pt,
}}
\tikzset{vx/.style={vxe, fill, draw}} 
\tikzset{vxs/.style={vx, minimum size=2*\vxrs}} 

\tikzset{sp/.style={
		inner xsep = \inxsep,
		inner ysep = \inysep,
}}
\tikzset{spv/.style={
			inner xsep = \inxsep + \vxr,
		 	inner ysep = \inysep + \vxr,
		 	}}
\newcommand{\vx}{node[spv]{} node[vx]{}} 
\newcommand{\vxs}{node[spv]{} node[vxs]{}} 

\newcommand*{\lrg}[1]{\scalebox{1.2}{#1}}

\usepackage{array}   
\usepackage{tabularx}
\usepackage{booktabs}

\newcommand*\Tstrut{\rule{0pt}{2.4ex}}  
\newcommand*\Tstrutt{\rule{0pt}{2.8ex}}  

\newcolumntype{L}{>{$}l<{$}}
\newcolumntype{R}{>{$}r<{$}}
\newcolumntype{C}{>{$}c<{$}}
\newcolumntype{\LX}{>{$}X<{$}}
\newcolumntype{\Lr}{>{$}l<{$\hspace*{-\tabcolsep}}} 
\newcolumntype{\Rr}{>{$}r<{$\hspace*{-\tabcolsep}}} 
\newcolumntype{\Rlr}{>{\hspace*{-\tabcolsep}$}r<{$\hspace*{-\tabcolsep}}} 
\newcolumntype{\Ll}{>{\hspace*{-\tabcolsep}$}l<{$}} 
\newcolumntype{\LlX}{>{\hspace*{-\tabcolsep}$}X<{$}}

\newcommand*{\G}{\mathcal{G}}
\newcommand*{\D}{\mathcal{D}}
\newcommand*{\B}{\mathcal{B}}

\newcommand*{\rv}{\mathrm{v}} 
\newcommand*{\n}{\mathrm{n}}
\newcommand*{\e}{\mathrm{e}}
\newcommand*{\piv}{\pi^{\mathrm{v}}}
\newcommand*{\pie}{\pi^{\mathrm{e}}}
\newcommand*{\dv}{d^{\kern0.1em\mathrm{v}}}

\newcommand*{\ph}{{\mkern2mu \cdot \mkern2mu}}  
\renewcommand{\dag}{\dagger\kern-0.1em} 
\newcommand*{\twodots}{\mathrel{{.}\,{.}}\nobreak} 
\newcommand*{\divides}{\mathrel{\mid}\nobreak}
\DeclarePairedDelimiter\ceil{\lceil}{\rceil}
\DeclarePairedDelimiter{\abs}{|}{|}
\DeclarePairedDelimiter{\mset}{[}{]} 
\newcommand*{\cmark}{\textcolor{Green4}{\ding{51}}}%
\newcommand*{\xmark}{\textcolor{Red4}{\ding{55}}}

\newcommand*{\2}{\kern0.111em}

\usepackage[square]{natbib}
\setcitestyle{numbers}
\usepackage{url}
\mathchardef\UrlBreakPenalty=100 
\mathchardef\UrlBigBreakPenalty=150 

\usepackage[colorlinks]{hyperref} 
\definecolor{DarkNavy}{rgb}{0,0.34,0.54}
\hypersetup{
	linkcolor=DarkNavy,
	citecolor=DarkNavy,
	urlcolor=NavyBlue,
	}
\usepackage[ocgcolorlinks]{ocgx2} 

\usepackage[all]{hypcap} 

\usepackage[left=3.5cm, right=3.5cm, top=2.5cm, bottom=2.5cm]{geometry} 

\theoremstyle{plain}
\newtheorem{thm}{Theorem}[section]
\newtheorem{prop}[thm]{Proposition}
\newtheorem{lem}[thm]{Lemma}
\newtheorem{cor}[thm]{Corollary}
\newtheorem{falsethm}[thm]{``Theorem''}
\newtheorem{conj}{Conjecture}

\theoremstyle{remark}
\newtheorem*{rmk}{Remark}
\newtheorem*{rmks}{Remarks}

\theoremstyle{definition}
\newtheorem{dfn}{Definition}
\newtheorem{exmp}{Example}[section] 
\newtheorem*{prb}{Problem}
\newtheorem*{qst}{Question}

\title{A new framework for identifying most reliable graphs\\
	and a correction to the \texorpdfstring{$K_{3,3}$}{K33}-theorem}
\author{Lorents F.\ Landgren\thanks{Chalmers University of Technology and University of Gothenburg, Sweden. Email: lorents@chalmers.se}
	\and Jeffrey E.\ Steif\thanks{Chalmers University of Technology and University of Gothenburg, Sweden. Email:
		steif@chalmers.se}}
	
\begin{document}
	
\setlength{\liststartpad}{\lsf\fontdimen2\font plus \lsf\fontdimen3\font minus \lsf\fontdimen4\font}
\setlength{\itemjoin}{\jf\fontdimen2\font plus \jf\fontdimen3\font minus \jf\fontdimen4\font}
\setlength{\listendpad}{\lef\fontdimen2\font plus \lef\fontdimen3\font minus \lef\fontdimen4\font}

\maketitle

\vspace{4ex}

\begin{abstract}
Given a multigraph~$G$, the all-terminal reliability~$R(G,p)$ is the probability that $G$ remains connected under percolation with parameter~$p$. Fixing the number of vertices~$n$ and edges~$m$, we investigate which graphs maximize~$R(G,p)$---such graphs are called \emph{optimal}---paying particular attention to uniqueness and to whether the answer depends upon~$p$. We generalize the concept of a \emph{distillation} and build a framework with which we identify all optimal graphs where $m-n\in\{1,2,3\}$. These graphs are uniformly optimal in~$p$. Most have been previously identified, but with serious problems, especially when $m-n=3$. We obtain partial results for $m-n\in\{4,5\}$.

For $m-n=3$, the optimal graphs were incorrectly identified by Wang in 1994, in the infinite number of cases where $m\equiv5\pmod{9}$ and $m\geq14$. This erroneous result concerns subdivisions of~$K_{3,3}$ and has been cited extensively, without any mistake being detected. While optimal graphs were correctly described for other~$m$, the proof is fundamentally flawed. Our proof of the rectified statement is self-contained.

For $m-n=4$, the optimal graphs were recently shown to depend upon~$p$ for infinitely many~$m$. We find a new such set of $m$‑values, which gives a different perspective on why this phenomenon occurs and leads us to conjecture that uniformly optimal graphs exist only for finitely many~$m$. However, for $m-n=5$, we conjecture that there are again infinitely many uniformly optimal graphs.
\end{abstract}

\newpage

\null\vspace{3ex}

\tableofcontents

\newpage

\section{Introduction}

\subsection{The problem of reliability}
\label{ss.intro}
Given a fixed number of vertices and edges, which graphs are the most likely to remain connected after edge-percolation with parameter $p\in (0,1)$? This problem was studied by Kelmans~\cite{Kelmans} and later independently formulated in~\cite{Bauer}, where it was described as the design or synthesis of reliable networks.

Throughout this paper, we allow a~graph to have multiple edges and loops. An $(n,m)$-graph is a graph with $n$~vertices and $m$~edges, counting multiplicity.
We will focus on graphs such that $m-n$ is small; we call this quantity \emph{exceedance} and denote it by~$k$. (For connected graphs, the exceedance equals what is called the \textit{corank} plus one.) We let $\G_{n,m}$ (or $\G_{n,n+k}$) denote the set of connected $(n,m)$-graphs.

\begin{dfn}[Reliability]
	\label{d.reliability}
	The reliability function of~$G$, denoted by $R(G,p)$, is the probability that the graph~$G$ remains connected under percolation with parameter~$p$ (the probability for each edge to remain). If $R(H,p)>R(G,p)$, where $H,G\in \G_{n,m}$, we say that $H$ is strictly more reliable than $G$ with respect to~$p$. If this holds for all $p\in(0,1)$, then $H$~is strictly more reliable than~$G$.
\end{dfn}

\begin{prb}\hypertarget{prb}{}
	Given $n$, $m$ and $p$, find the \emph{$p$‑optimal graphs}, defined as the graphs in~$\G_{n,m}$ which maximize~$R(G,p)$.
\end{prb}

\begin{figure}[htb]
	\vspace{-\parsep}
	\centering
	\begin{tikzpicture}[x=0.12cm,y=0.12cm]
		\def\r{4.5} \def\s{\r*3}
		\foreach \v in {0,72,...,288}{
			\draw ({\r*cos(\v)},{\r*sin(\v)})  \vxs -- ({\r*cos(\v+72)},{\r*sin(\v+72)});
		}
		\foreach \v in {0,72,...,288}{
			\draw ({-\r*cos(\v)+\s},{-\r*sin(\v)}) \vxs -- ({-\r*cos(\v+72)+\s},{-\r*sin(\v+72)});
		}
		\draw (\r,0) -- (\s-\r,0);
	\end{tikzpicture}
	\hspace{36pt}
	\begin{tikzpicture}[x=0.12cm,y=0.12cm]
		\def\r{4.3}
		\foreach \v in {0,120,240}{
			\draw ({\r*cos(\v)},{\r*sin(\v)}) \vxs -- ({\r*cos(\v+120)},{\r*sin(\v+120)});
		}
		\foreach \v in {0,45,...,315}{
			\draw ({-\r*cos(\v)+2*\r},{-\r*sin(\v)}) \vxs -- ({-\r*cos(\v+45)+2*\r},{-\r*sin(\v+45)});
		}
	\end{tikzpicture}
	\caption{For sufficiently small (large)~$p$, the left (right) graph is strictly more reliable.}
	\label{f.handcuffs}
\end{figure}

Boesch~\cite{Boesch} defined a Uniformly Most Reliable Graph (UMRG) as an $(n,m)$-graph which is $p$‑optimal for every~$p$. A~recent survey of what is known about UMRGs is due to Romero~\cite{Rsurvey}. Noting that the definition of a~UMRG involves a non-strict inequality, we propose the following stronger notion:

\begin{dfn}[Unique optimality]
	$G \in \G_{n,m}$ is a uniquely optimal graph if $G$~is strictly more reliable than all other graphs in~$\G_{n,m}$.
\end{dfn}

\begin{figure}[hb!]
	\centering
	\vspace{-\parsep}
	\begin{tikzpicture}[x=0.12cm,y=0.10cm]
		\draw (0,0) \co(a) \vxs -- (3,3) \vxs -- (7,4) \vxs --(11,3) \vxs --(14,0) \co(b) \vxs (a) -- ($(a)!0.33!(b)$) \vxs -- ($(a)!0.66!(b)$) \vxs -- (b) (a) --(3,-3) \vxs -- (7,-4) \vxs --(11,-3) \vxs -- (b);
	\end{tikzpicture}
	\caption{The uniquely optimal $(10,11)$-graph has parallel chains of lengths 4, 3 and~4.}
	\label{f.3pc}
\end{figure}

Unique optimality is equivalent to unique $p$-optimality for all values of~$p$. This is logically a stronger property than for a UMRG to be unique, which is of~course stronger than simply being UMRG. The only known case where UMRGs are not unique is the case of trees (as noted below, every tree is a UMRG). Whether unique UMRGs are necessarily uniquely optimal seems to be a nontrivial problem.

It was conjectured in~\cite{Boesch} that a UMRG always exists, but an infinite family of almost complete graphs disproving the conjecture had already been described by~\cite{Kelmans}. Similar counterexamples were independently and elegantly demonstrated in~\cite{MyrvoldShort}. Although~\cite{MyrvoldShort}, like \cite{Boesch}, were working in the context of simple graphs, an additional argument can be made to show that there are no UMRGs with multiple edges in the relevant families $\G_{n,m}$. For the particular case of $(6,11)$-graphs, see Proposition~\ref{p.myrvold}. For an example of simply how the relative reliability of two graphs can depend upon~$p$, the reader might ponder the $(10,11)$-graphs in Figure~\ref{f.handcuffs}. 

The present paper will focus on families of graphs of low exceedance. The solution to our \hyperlink{prb}{Problem} is easy for the smallest possible values of $m-n = k$. A~tree on $n$~vertices has exceedance~$-1$ and reliability $R(G,p)=p^{n-1}$; thus, every tree is trivially a~UMRG, which for $n\geq 4$ is not unique. Proceeding to consider $k=0$, the reader may convince themself that the set of cycles constitute a family of uniquely optimal graphs for the $\G_{n,n}$‑sets, where $n\geq 1$.

For $k\in\{1,2,3\}$, it turns out that there is always a uniquely optimal graph. For $k=1$, it is well known that a construction given in~\cite{Bauer} gives a~UMRG for each size~$m$. These graphs consist of three parallel paths with as equal length as possible, as in Figure~\ref{f.3pc}. Finding the uniquely optimal graphs when $k=2$ and $k=3$ takes more work. Most of the graphs have previously been identified as UMRGs in the literature, but there are serious problems. In the case of $k=3$, an infinite number of graphs were erroneously identified as UMRGs in~\cite{Wang94} through ``Theorem''~\ref{falsethm}. The smallest such mistaken UMRG is the $(11,14)$-graph $G$, shown in Figure~\ref{f.misidentified} together with $H$, the true UMRG. For both $k=2$ and $k=3$, we will point out several flaws in previous proofs and statements and provide independent and self-contained characterizations of the uniquely optimal graphs.

\begin{figure}[tb!]
	\centering
	\includegraphics{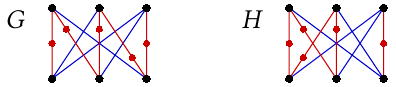}
	\caption{$G$ is the smallest in an infinite sequence of graphs which have wrongly been considered optimal since 1994, when Wang ``proved'' a 1991 conjecture by Boesch, Li and Suffel. The true uniquely optimal $(11,14)$-graph is shown to be~$H$.}
	\label{f.misidentified}
\end{figure}

With increasing $k$, the problem becomes exceedingly complex. For $k=4$, there are infinitely many~$m$ where no UMRG exists. For $k=5$, little is known.

In a recent paper by Kahl and Luttrell~\cite{Tutte}, the uniquely optimal graphs for $k\in\{0,1,2\}$ were shown to in a certain sense maximize the \emph{Tutte polynomial}, and hence also maximize many graph parameters obtainable from the Tutte polynomial (the reliability function being one). The $(k=3)$-graphs of Wang~\cite{Wang94} were then conjectured to similarly maximize the Tutte polynomial. This conjecture was amended in the update~\cite{KahlUpdate} to apply to the true uniquely optimal $(k=3)$-graphs identified by Theorem~\ref{t.k3} below.

\subsection{Summary of the paper}
In Sections~\ref{ss.graphs} and~\ref{ss.perc}, we provide basic background concerning graph theory and the reliability function. In Section~\ref{ss.bridges}, we show that except for trees, an optimal graph can never contain a~bridge. In Section~\ref{s.framework}, we introduce the proper distillation of a~graph and then the more general notion of a weak distillation. This gives us graphs with a simple structure from which more general graphs can be constructed. In Section~\ref{s.k=1}, we demonstrate how the uniquely optimal graphs can easily be obtained when $k=1$ (i.e.~$m-n=1$).

In Section~\ref{s.maindist}, we introduce an equivalence relation on graphs with the following properties.
\begin{eninline}[label={(\arabic*)}]
	\item Equivalent graphs have identical reliability functions.
	\item Every bridgeless graph has an equivalent graph with a weak distillation which is 3‑edge-connected and cubic.
\end{eninline}
This allows us to restrict attention to graphs which can be built from 3-edge-connected cubic graphs, which is crucial to obtain our main results.

In Section~\ref{s.minimize} we describe, for general $n$~and~$m$, the graphs which minimize the number of disconnecting edge sets of size~2, and within restricted sets of graphs those which minimize the number of disconnecting sets of size~3. In Section~\ref{s.move}, we study how moving an edge within a~graph affects the number of disconnecting sets.

In Section~\ref{s.k=2}, we study the case $k=2$ and obtain the uniquely optimal graphs. In Section~\ref{s.k=3}, we study the case $k=3$ and provide the first correct characterization of its uniquely optimal graphs. Finally, in Section~\ref{s.k=4,5}, we provide partial results and conjectures for the cases $k=4$ and $k=5$.

\section{Preliminaries}

\subsection{Graph theoretical essentials}
\label{ss.graphs}
For standard notions of graph theory we refer to Diestel~\cite{Diestel}. In particular, we will use the following notation and definitions. All graphs are connected finite multigraphs, which may contain loops. We use the word \emph{multigraph} only to contrast with simple graphs when discussing results from the literature.

A general graph with $n$~vertices and $m$~edges has exceedance~$k$, defined by $m-n$, and will be referred to as an $(n,m)$-graph, $(n,n+k)$-graph or $(m-k,m)$-graph, depending on convenience. The \emph{size} of a~graph is its number of edges. An $m$‑path has $m$~edges and the $r$‑star is the complete bipartite graph $K_{1,\mkern2mu r}$. The $m$‑dipole consists of two vertices connected by $m$~edges and the $m$‑bouquet is a single vertex with $m$~loops.

A cubic graph of exceedance $k$ has $2k$~vertices and $3k$~edges. To see this, use the degree sum formula with $\abs{E(G)} = n+k$ and simplify to obtain $n=2k$.

A \emph{spanning subgraph} of~$G$ contains all the vertices of~$G$. The number of spanning trees of~$G$ is denoted by $t(G)$. A~\emph{matching} of~$G$ is a set of disjoint edges of~$G$. A~matching is \emph{perfect} if it spans~$G$.

If $G$~contains the edge~$e$ (or edge set~$E$), then $G-e$ ($G - E$) is the graph resulting from the \emph{deletion} of this (these) edges. An (edge-)disconnecting set or \emph{disconnection} of~$G$ is a set of edges whose deletion disconnects~$G$. Every disconnection contains at least one \emph{cut}, which is the set of edges crossing a non-trivial partition of the vertices into two sets called \emph{sides}. Minimal cuts are called \emph{bonds}. An $i$‑disconnection is a disconnection of size~$i$, and we define an $(i,j)$-disconnection to be an $i$‑disconnection in which the smallest bond has size~$j$. We let $d_i(G)$ denote the number of $i$‑disconnections, $d_{i,j}(G)$ denote the number of $(i,j)$-disconnections and $b_i(G)$ denote the number of $i$‑bonds of~$G$.

The symmetric difference of two distinct cuts is a~cut \cite[Prop.\ 1.9.2]{Diestel}, and every cut is a disjoint union of bonds \cite[Lemma 1.9.3]{Diestel}.

The edge-connectivity of~$G$, denoted $\lambda(G)$, is the size of the smallest bond. A~graph with edge-connectivity at least~$\mu$ is said to be $\mu$‑edge-connected. An edge $e\in G$ is a~bridge if the deletion of~$e$ disconnects~$G$. Note that for a connected graph, \emph{bridgeless} is the same as \emph{2‑edge-connected}.

The following definition of \emph{cutvertex} includes a vertex with one or more loops (except if the entire graph is just one vertex with one loop). 
For loopless graphs, this definition is equivalent to the more common one, where the removal of the cutvertex disconnects the graph. A~\emph{block} is a maximal connected subgraph which has no cutvertex of its own.

\begin{dfn}[Cutvertex] 
	A vertex $v\in G$ is a cutvertex if the edges of~$G$ can be partitioned into two nonempty sets such that $v$~is the only vertex incident with at least one edge in each set.
\end{dfn}

\subsection{Reliability}
\label{ss.perc}
We use the standard notion of percolation, in which edges are independently retained with probability $p\in (0,1)$ and otherwise deleted.

Let $c_i(G)$ denote the number of connected spanning subgraphs of~$G\in\G_{n,m}$ with $m-i$~edges. Noting that $c_i(G) = 0 $ for $i>k+1$, we see that the reliability function can be expressed as
\begin{equation}
	\label{ccoeff}
	R(G,p)=\sum_{i=0}^{k+1} c_i(G)p^{m-i}(1-p)^i\,.
\end{equation}

While our results will be given in terms of the reliability function, it is often more practical to work with the complementary probability~$U(G,p)$, where \emph{U} is for \emph{unreliability}. Recalling that $d_i(\ph)$ denotes the number of $i$‑disconnections,
\begin{equation}
	\label{dcoeff}
	U(G,p)=\sum_{i=1}^{m}d_i(G)p^{m-i}(1-p)^i\,.
\end{equation}
Note that $d_i(G)=0$ for $i<\lambda(G)$, that $d_i(G)=\binom{m}{i}$ for $i > k+1$, and that $c_i(G)+d_i(G) = \binom{m}{i}$ for all $i\in [0 \twodots m ]$. (We use the notation $[a \twodots b]$ for the closed integer interval from $a$ to~$b$.) 

It follows from the above that a sufficient condition for a graph $H\in \G_{n,m}$ to be strictly more reliable than a graph $G \in \G_{n,m}$ is the following set of inequalities, where at least one is strict:
\begin{equation}
	\label{suff}
	d_i(H)\leq d_i(G)\,, \mathrlap{\qquad  \forall i\in [1 \twodots k+1]\,.}
\end{equation}
It was recently shown by Graves~\cite{Graves} that this sufficient condition is not necessary.

A sufficient condition for $H$ to be uniquely optimal is consequently that \eqref{suff} holds for all $G \in \G_{n,m} \!\setminus \{H\}$, with at least one strict inequality for each~$G$. Without the final requirement, we have a sufficient condition for~$H$ to be a~UMRG. Whether this condition is necessary is not known~\cite{Rsurvey}.

It is instructive to consider the graph theoretic properties which govern the reliability function for large~$p$ and for small~$p$. Note that $\lambda(G)$ can be expressed as $\min\{ i:d_i(G) >0\}$ and that the number of spanning trees of~$G$ equals $c_{k+1}(G)$. The following proposition is therefore immediate.

\begin{prop}
	\label{p.largesmallp}\ 
	\begin{enumenun}
		\item If the edge-connectivity of~\(H\) is strictly larger than that of~\(G\), then, for sufficiently large~\(p\),
		\begin{equation}
			\label{largesmallp}
			R(H,p) > R(G,p)\,.
		\end{equation}
		
		\item\label{pp.spanningtrees}
		If the number of spanning trees of~\(H \) is strictly larger than that of~\(G \), then~\eqref{largesmallp} holds for sufficiently small~\(p \).
	\end{enumenun}
\end{prop}

\subsection{Graphs with bridges are not \texorpdfstring{$p$}{p}-optimal}
\label{ss.bridges}
Graphs with bridges (except for trees) cannot be $p$‑optimal.
Variants of this result can be found e.g.\ in~\cite{Rsurg} and \cite{Wang94}. However, we have not been able to obtain the exact content we need from the literature. A~similar result involving cutvertices instead of bridges is possible, but the following will be sufficient for our purposes.

\begin{prop}
	\label{p.bridge}
	If \(G\in \G_{n,n+k}\) with \(n \ge 2\) and \(k \ge 0\) has a~bridge, then there exists a graph \(G'\in \G_{n,n+k}\) such that \(d_i(G') < d_i(G)\) for all  \(i\in [1 \twodots k+1] \). Hence, any \(p\)‑optimal \((n,n+k)\)-graph is bridgeless.
\end{prop}

\begin{proof}
	Since $k\ge0$, $G$~has a~cycle. Choose a bridge $b=vw$ that is adjacent to some edge $e=uv$ which belongs to a~cycle, and then construct~$G'$ by moving an incidence of~$e$ from~$v$ to~$w$, as shown in Figure~\ref{f.bridge}.
	
	We first show that every edge set which disconnects $G'$ also disconnects~$G$, which is to say that $d_i(G') \leq d_i(G)$ for all $i \in [1 \twodots m]$. Letting $E$ be an edge set which disconnects~$G'$, there are three cases for $b$~and~$e$. If $b\in E$, it is immediate that $E$ also disconnects $G$. If $e\in E$, then $G-E$ is disconnected since $G'-e$ and $G-e$ are isomorphic. If $e\notin E$ and $b \notin E$, then clearly $G-E$ is also disconnected.
	
	Now, let $i \in [1 \twodots k+1]$. We show that there is at least one edge set~$E_i$ of size~$i$ which disconnects $G$ but not~$G'$. Since $G'-b$ is a connected $(n,n+k-1)$-graph, we can choose an edge set~$K$ of size~$k$ contained in $G'-b$ such that $G'-b-K$ is a~tree. Let $E_i$ consist of~$b$ together with $i-1$~of the edges in~$K$. Then, $G'-E_i$ is connected, but $G - E_i$ is disconnected. We conclude that $d_i(G') < d_i(G)$ for every $i \in [1 \twodots k+1]$.
\end{proof}

	\begin{figure}[t!]
	\centering
	\includegraphics{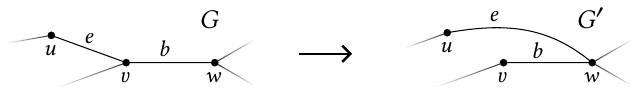}
	%
	%
	\caption{Given a graph $G$ in which $b$ is a~bridge and $e$~is a non-bridge, the surgery gives a graph~$G'$ with fewer disconnecting sets of all relevant sizes, which is therefore strictly more reliable than~$G$.}
	\label{f.bridge}
\end{figure}

\section{Distillation framework}
\label{s.framework}

\subsection{Proper distillations and chains}
\label{ss.pdist}
A critical idea introduced in~\cite{Bauer} is the conversion of a graph~$G$ into its \emph{distillation} by suppressing all vertices of degree~2. We will call this graph the \emph{proper} distillation of~$G$, since we will promptly generalize the concept to a set of \emph{weak} distillations of~$G$. We can think of the distillations as ``blueprints'' by which we can represent larger sets of graphs; in fact, it can be shown that for any~$k$, the leafless graphs of exceedance~$k$ can be represented by only finitely many proper distillations (and by the same finite set of leafless weak distillations). See Figure~\ref{f.distillations} for an example of a graph~$G$, its proper distillation~$D_1$ and some of its weak distillations. The relevant definitions are as follows.

\begin{dfn}[Vertex suppression and insertion]
	\label{d.suppins}
	A loopless 2-vertex~$v$ is \emph{suppressed} by de\-leting $v$ and joining its two neighbors by an edge. A~vertex which does not have degree two is \emph{non-suppressible}.
	The \emph{insertion} of a~vertex at the edge~$e$ is the above operation in reverse.
\end{dfn}

\begin{rmk}
	Vertex suppression and insertion are special cases of edge contraction and expansion; see Definition~\ref{d.expansion}.
\end{rmk}

\begin{dfn}[Distillation and subdivision]
	A distillation is a graph without vertices of degree~2. The \emph{proper distillation} of a non-cycle~$G$, denoted by~$D_G$, is the graph obtained by suppressing all 2‑vertices of~$G$. In turn, $G$ is a~\emph{subdivision} of~$D_G$.
	For weak distillations and weak subdivisions, see Definition~\ref{d.weaks}.
\end{dfn}

While subdivisions are usually defined for general graphs, our definition only allows for subdivisions of distillations. A~distillation and its subdivisions have the same $k$‑value, since vertex suppression and insertion preserve exceedance. Note that $G$ is bridgeless if and only if $D_G$ is bridgeless.

\begin{dfn}[Chain]
	A \emph{positive chain} of a~graph is a maximal path which contains at least one edge, does not contain its endvertices and in which every vertex has degree~2. See Figure~\ref{f.parallel}. A~\emph{chain} is either a positive chain or a \emph{zero~chain}, which will be defined in Definition~\ref{d.zerochain}. The~\emph{length} of a~chain is its number of edges.
\end{dfn}

\begin{figure}[tb!]
	\centering
	\includegraphics{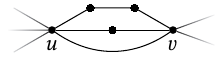}
	\caption{A part of a~graph with three parallel, positive chains of lengths 3,~2 and~1. The endvertices $u$~and~$v$ are not part of any chain; and the number of other edges incident to $u$~and~$v$ is arbitrary.}
	\label{f.parallel}
\end{figure}

Note that every edge of a graph~$G$ belongs to exactly one positive chain (an edge which is not incident to a 2‑vertex is a chain of length one). Hence, the proper distillation of~$G$ can be equivalently defined as the graph obtained by replacing each positive chain of~$G$ with an edge. Likewise, the subdivisions of~$D_G$ can be defined as all graphs obtainable by replacing the edges of~$D_G$ with positive chains.

\subsection{Weak distillations, weak subdivisions and zero chains}

We now generalize proper distillations and subdivisions to weak distillations and weak subdivisions. This will allow us to restrict our attention to cubic distillations.

The key picture is as follows: All weak subdivisions of a distillation~$D$ can be obtained by replacing the edges of~$D$ with chains, where chains can have zero length, meaning that the corresponding edge is contracted. (However, we will not allow all edges of a~cycle to be contracted.) If the chain lengths do not differ by more than one edge, the weak subdivision will be called \emph{balanced}.

\begin{figure}[tb!]
	\centering
	\includegraphics{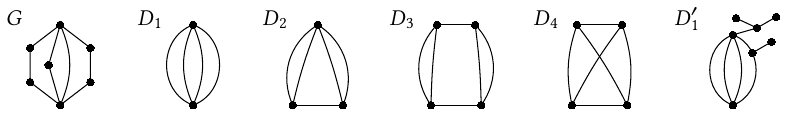}
	\caption{The $(7,9)$-graph~$G$ is a subdivision of~$D_1$, the proper distillation of $G$. The leafless weak distillations of~$G$ are $D_1$ through $D_4$, and $D'_1$ is one of infinitely many with leaves.}
	\label{f.distillations}
\end{figure}

\begin{dfn}[Edge expansion and contraction]
	\label{d.expansion}\ 
	\begin{enumenun}
		\item Expanding an edge at a vertex~$v$ means exchanging~$v$ for two adjacent vertices $v_1$ and~$v_2$, assigning each incidence with $v$ to either $v_1$ or~$v_2$. (A~loop at~$v$ is replaced either with a~loop at $v_1$~or~$v_2$ or with an additional edge between $v_1$ and~$v_2$.)
		
		\item Contracting a non-loop edge~$e$ in an $(n,m)$-graph~$G$ means removing the edge and merging its endvertices, yielding an $(n-1,m-1)$-graph denoted by~$G/e$.
		(Any edges parallel to the contracted edge become loops, which cannot be contracted.)
	\end{enumenun}
\end{dfn}

Note that edge expansion and contraction both preserve exceedance. It should also be clear that the edge-connectivity of a~graph can only decrease under edge expansion and can only increase under edge contraction.

\begin{dfn}[Weak distillation and weak subdivision]
	\label{d.weaks}\ 
	\begin{enumenun}
		\item If a distillation~$D$ can be obtained from the proper distillation of~$G$ by edge expansions, then $D$ is a \emph{weak distillation} of~$G$.
		
		\item If $G$ can be obtained from a distillation~$D$ by edge contractions and vertex insertions, then $G$ is a \emph{weak subdivision} of~$D$. 
	\end{enumenun}
\end{dfn}

It is easy to show that $G$ is a weak subdivision of~$D$ if and only if $D$ is a weak distillation of~$G$. 
Furthermore, such graphs $D$ and~$G$ have the same exceedance. See Figure~\ref{f.distillations} for examples.

The following proposition will in practice be superseded by Theorem~\ref{t.3ec}, which has an independent proof. However, we consider the below to be a natural development of ideas, and Proposition~\ref{p.cubic}\ref{pp.cubic2} will be used to easily identify the uniquely optimal $(k=1)$-graphs.
\begin{prop}
	\label{p.cubic}
	Let \(G\in \G_{n,n+k}\), where \(n\geq 1\) and \(k\geq 1\).
	\begin{enumenun}
		\item\label{pp.cubic1}
		\(G\) has a cubic weak distillation if and only if \(G\) is leafless. 
		\item\label{pp.cubic2}
		\(G\) has a bridgeless, cubic weak distillation if and only if \(G\) is bridgeless.
	\end{enumenun}
\end{prop}

\begin{proof}
	Part~\ref*{pp.cubic1}, \emph{only~if} direction: Suppose that $G$ has a leaf~$v$. Leaves are not suppressed, so $v$ remains a~leaf in~$D_G$. Expanding an edge at~$v$ creates a new leaf, so no weak distillation of~$G$ can be cubic.
	
	Part~\ref*{pp.cubic2}, \emph{only~if} direction: If $G$ has a bridgeless weak distillation, then its proper distillation is bridgeless. As previously noted, this implies that $G$ is bridgeless.
	
	Part~\ref*{pp.cubic1}, \emph{if} direction: Suppose that $G$ is leafless. Its proper distillation $D_G$ then has minimum degree~3. If $D_G$ is not cubic, pick a vertex $v\in D_G$ for which $\deg(v)\geq 4$, and expand an edge $e_v$ at~$v$ in such a~way that the two new vertices $v_1$ and~$v_2$ each receives degree at least~3. Noting that the resulting distillation remains leafless and that both $v_1$ and~$v_2$ have strictly smaller degree than~$v$, we can repeat the procedure until a cubic weak distillation of~$G$ is obtained.
	
	Part~\ref*{pp.cubic2}, \emph{if} direction: Suppose that $G$ is bridgeless, which immediately implies that $D_G$ is bridgeless. We modify the above procedure to ensure that no bridge is created by the edge expansion. Clearly, an already existing edge cannot become a~bridge. It is also easy to see that if the chosen vertex $v$ is not a cutvertex, then the expanded edge~$e_v$ cannot be a~bridge (cf.\ Figure~\ref{f.cubex2}). On the other hand, if $v$ is a cutvertex, then every block containing $v$ contributes with at least two incidences to~$v$, since $D_G$ is bridgeless. When expanding the edge~$e_v$, we let at least one incidence from each block go to each of $v_1$ and~$v_2$ (cf.\ Figure~\ref{f.cubex1}). This ensures that $e_v$~is not a~bridge. Thus, we can obtain a bridgeless cubic weak distillation of~$G$.
\end{proof}

\begin{dfn}[Zero chain]
	\label{d.zerochain}
	Let $G$ be a weak subdivision of~$D$. An edge $e\in D$ which was contracted in the construction of~$G$ corresponds to a zero-length chain in~$G$, relative to~$D$. This zero chain is located at the non-suppressible vertex of~$G$ which the endvertices of~$e$ were merged into. (See Example~\ref{ex.weak}.)
\end{dfn}

Now, given a distillation~$D$, a weak subdivision of~$D$ can be equivalently defined as a~graph which can be obtained by going through the edges of~$D$ and replacing each one with a chain of length~$\geq 0$, except for loops, including loops arising in the process, \hypertarget{zerocycle}{}which are each replaced by a chain of length~$\geq 1$ (since loops cannot be contracted).
We can therefore represent weak subdivisions of~$D$ by letting edge weights indicate chain lengths (see Corollary~\ref{c.bonds}), observing that no cycle can have total weight zero (since contracting every edge but one in a~cycle yields a~loop). In other words, weak subdivisions may have zero chains, but no ``zero cycles''.

Later on, we will have reason to consider chains which are ``adjacent'' and ``nonadjacent''. In practice, the meaning of this should be straightforward; however, there is some subtlety relating to zero chains (see Example~\ref{ex.weak}). Formally, two chains of~$G$ are adjacent relative to~$D$ if the corresponding edges of~$D$ are adjacent.

\begin{dfn}[Balanced graph/chains]
	A weak subdivision~$G$ of a distillation~$D$ is \emph{balanced} with respect to~$D$ if the chain lengths of~$G$ differ by at most one  (including zero chains). 
	If all chains have the same length, then $G$ is \emph{perfectly balanced} relative to~$D$. A~\emph{balanced set of chains} is analogously defined.
\end{dfn}

Note that if $G$ has a chain of length at least~2, then it can only be balanced with respect to its proper distillation~$D_G$ (since $G$ has a zero chain with respect to any other weak distillation). On the other hand, if $G$ has no chain of length 2~or more, then $D_G = G$ and~$G$ is balanced with respect to \emph{all} of its weak distillations. 

\begin{exmp}
	\label{ex.weak}
	In Figure~\ref{f.distillations} above, $G$~has, at its lowermost vertex, one zero chain relative to~$D_2$. Furthermore, $G$ has two zero chains relative to $D_3$ and~$D_4$, and four relative to~$D_1'$. The two chains of $G$ with lengths 1~and~2 are adjacent relative to $D_1$~and~$D_2$ but nonadjacent with respect to $D_3$ and~$D_4$ (assuming that the intended chain–edge correspondence is clear). Furthermore, $G$ is imbalanced with respect to all of its weak distillations, but inserting a~vertex at the curved edge of~$G$ would make the graph balanced with respect to its proper distillation~$D_1$.
\end{exmp}

\subsection{Bond counting}
The reader is reminded that a~bond is a minimal cut. There is a natural correspondence between the bonds of a distillation and those of its weak subdivisions, which through Corollary~\ref{c.bonds} will be useful for counting bonds. We will need the following notion of a \emph{trivial} 2‑bond. (Trivial 3- and 4‑bonds will become important later on. One may think of a trivial 4‑bond as isolating a~chain from the rest of the graph.)

\begin{dfn}[Trivial bond, assuming a leafless graph]\ 
	\begin{enumenun}
		\item A 2‑bond is trivial if its two edges belong to the same chain.
		\item A 3‑bond is trivial if its edges belong to three chains emanating from a~3‑vertex.
		\item A 4‑bond is trivial if the corresponding bond in one (or equivalently, in all) of its cubic weak distillations isolates one edge.
	\end{enumenun}
\end{dfn}

\begin{rmk}
	The following alternative definition implies the one above: \emph{A bond is trivial if at least one of its sides is a~tree.}
\end{rmk}

\begin{exmp}
	In Figure~\ref{f.distillations}, all 2‑bonds of~$G$ are trivial, while the only 2‑bond of~$D_3$ is nontrivial. Referring to Figure~\ref{f.Dgraphs}, the graphs $K_4$ and $K_{3,3}$ have only trivial 3‑bonds, while the triangular prism~$\Pi_3$ contains one nontrivial 3‑bond.
\end{exmp}

\begin{lem}
	\label{l.bonds}
	Let \(D\) be a weak distillation of a graph~\(G\), whose chain lengths (which may be zero) are denoted by~\(\ell_i\), and let \(s\geq 1\). Then each \(s\)‑bond of~\(D\) naturally corresponds to a set of \(s\)‑bonds of~\(G\) with size \(\Pi_{i=1}^s \ell_i\), and distinct \(s\)‑bonds of~\(D\) correspond to disjoint sets. For \(s\neq 2\), these sets cover the set of \(s\)‑bonds of~\(G\), and for \(s=2\) they cover the set of nontrivial \(s\)‑bonds of~\(G\).
\end{lem}
\begin{proof}
	Let $G$ be a graph with a weak distillation~$D$ and let $S$ be an $s$‑bond of~$G$. Then, either $S$ consists of two edges from the same chain, in which case $S$ is a  trivial 2‑bond, or $S$ consists of edges from $s$~different chains of~$G$, and these chains correspond to an $s$‑bond in~$D$. Thus, we have a natural mapping from the set of bonds of~$G$ minus the trivial 2‑bonds, to the bonds of~$D$. With this, the lemma is immediate. (The mapping is typically highly non-injective and is surjective if and only if $D$ is the proper distillation of~$G$.)
\end{proof}

\begin{cor}
	\label{c.bonds}
	Assign weights to a distillation~\(D\) which yield a weak subdivision~\(G\) of~\(D\). For \(s \neq 2\), the number of \(s\)‑bonds of~\(G\) equals the sum of the products of the weights of the \(s\)‑bonds of~\(D\).
\end{cor}

\section{The uniquely optimal \texorpdfstring{$(\MakeLowercase{k} = 1)$}{(k=1)}-graphs}
\label{s.k=1}

Subdivisions of the 3‑dipole (Figure~\ref{f.3pc}) have been called $\theta$‑graphs, since their shape resembles theta. However, we define $\theta$‑graphs to be \emph{weak} subdivisions of the 3‑dipole. This allows for the degenerate case were there is one zero chain. Such $\theta$‑graphs consist of two cycles connected by a cutvertex.

Balanced $\theta$‑graphs solve our \hyperlink{prb}{Problem} when $k=1$. The fact that such graphs are UMRGs was pointed out in~\cite{Li} based on~\cite{Bauer}. An independent development appeared in~\cite{WangWu} (cited in~\cite{Tseng}). We feel that it would be a natural development of ideas to provide a short and direct proof that these graphs---the first five of which are shown in Figure~\ref{f.firstk1}---are uniquely optimal.

\begin{figure}[htb!]
	\centering
	\includegraphics{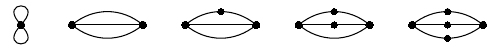}
	%
	%
	%
	%
	%
	%
	\caption{The first five uniquely optimal $(n,n+1)$-graphs. The pattern continues and cycles every three graphs.}
	\label{f.firstk1}
\end{figure}

\begin{prop}
	\label{p.k1}
	For given \(n\geq 1\) and \(m = n + 1\) there exists a uniquely optimal \((n,m)\)-graph, namely the balanced $\theta$‑graph of size~$m$.
\end{prop}

\begin{proof}
	By Proposition~\ref{p.bridge}, any $p$‑optimal $(n,m)$-graph~$G$ is bridgeless. By Proposition~\ref{p.cubic}\ref{pp.cubic2}, $G$~has a bridgeless, cubic weak distillation~$D$. Since $D$ has exceedence~1 and is cubic, $D$ has 2~vertices and 3~edges. The only possibility is the 3‑dipole, shown in Figure~\ref{f.dipole}. 
	
	Hence, $G$ is a $\theta$‑graph.
	Clearly, a percolation outcome of~$G$ is connected if and only if either at most one edge is deleted or exactly two edges from two different chains are deleted. Letting the chain lengths of~$G$ be $\ell_1$,~$\ell_2$ and~$\ell_3$, equation~\eqref{ccoeff} becomes
	\begin{equation}
		\label{k1}
		R(G,p)=p^{m}+mp^{m-1}(1-p)+\left(\ell_1 \ell_2 + \ell_1 \ell_3 + \ell_2 \ell_3 \right)p^{m-2}(1-p)^2\,.
	\end{equation}
	With $\ell_1 + \ell_2 + \ell_3 = m$, it is easy to show that the coefficient of the final term, and therefore $R(\ph,p)$, is maximized if and only if the chains are balanced, which specifies an $m$‑sized $\theta$‑graph up to isomorphism. We conclude that this is the unique $p$‑optimal $(n,m)$-graph, which is uniquely optimal since $p$~is arbitrary.
\end{proof}
\begin{figure}[h!]
	\vspace{-3mm}
	\centering
	\includegraphics{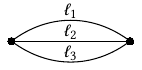}
	\caption{The 3‑dipole, or an~arbitrary $\theta$‑graph with labels representing chain lengths, at most one of which can be zero.}
	\label{f.dipole}
\end{figure}

\section{Main distillation result}
\label{s.maindist}
The main result regarding distillations is Theorem~\ref{t.3ec} in Section~\ref{ss.represent}, which essentially says that any bridgeless graph can be ``represented'' by at least one 3‑edge-connected cubic distillation. To properly formulate this result, we first need to introduce an equivalence relation on $(n,m)$-graphs (which may be considered interesting in its own right).

\subsection{Equivalently reliable graphs}
\label{ss.equiv}
We now introduce a reversible surgery which we call \emph{edge shifting} and show that it does not change the reliability function. Two edge shifting examples are shown in Figure~\ref{f.shiftex}. (It was pointed out to us by the authors of~\cite{Tutte} that an edge shift is a special case of what is known as a \emph{Whitney twist}, which does not affect the Tutte polynomial, and hence nor the reliability. However, we prefer to keep things as elementary as possible.)

\begin{dfn}[Edge shifting]
	\label{d.shifting}
	If the following surgery can be performed, it \emph{shifts} the edge~$e$ to the vertex~$v$.
	First expand an edge~$e'$ at the vertex~$v$ in such a~way that $e$ and $e'$ form a~2‑bond, and then contract~$e$.
\end{dfn}

\begin{figure}[tb!]
	\centering
	\includegraphics{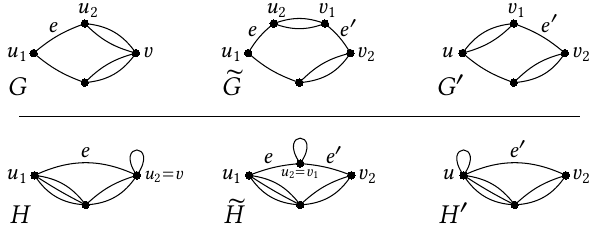}
	\caption{The two nontrivial ways to shift an edge, exemplified. In~$G$, the edge~$e$ is shifted to~$v$ by first expanding~$e'$ and then contracting~$e$, yielding~$G'$. In~$H$, the edge~$e$ is shifted to~$u_2$ so that the loop is effectively moved left. $G$~and~$G'$ ($H$~and~$H'$) are called equivalent.}
	\label{f.shiftex}
\end{figure}

\begin{rmk}
	Whether $e$ and $e'$ form a~2‑bond in the first step depends upon how $e'$ is expanded. Since there still may be a choice involved, the resulting graph is not uniquely defined.
\end{rmk}

Let $e=u_1u_2$ be an edge of some graph which can be shifted to the vertex $v$. If the edges $e$ and~$e'$ above belong to the same chain, then the resulting graph is isomorphic to the original. If the graph is 3‑vertex-connected, such trivial edge shifts will be the only ones possible. More interesting edge shifts can be performed in one of the following two circumstances.
\begin{enumerate}
	\item The vertex~$v$ is not part of nor incident with the chain containing~$e$, and every $u_1u_2$-path in $G-e$ goes through~$v$. See $G$ and the resulting~$G'$ in Figure~\ref{f.shiftex}.
	\item The vertex~$v$ is an endvertex of the chain containing $e$, and $v$ is a cutvertex. See $H$ and the resulting~$H'$ in Figure~\ref{f.shiftex}.
\end{enumerate}

By observing that any bridges remain as bridges during edge shifts, we note that every equivalence class defined below consists either of bridgeless graphs or of bridged graphs.

\begin{dfn}[Equivalent graphs]
	\label{d.equiv}
	Two graphs are equivalent if one can be obtained from the other by repeated edge shifting.
\end{dfn}

\begin{prop}
	\label{p.equiv}
	If \(G\)~and~\(H\) are equivalent graphs, then
	\begin{equation}
		\label{equivalent}
		d_i(G)= d_i(H) \mathrlap{ \qquad \forall i \in [1 \twodots m] \,,}
	\end{equation}
	so that \(G\)~and~\(H\) have the same (un)reliability function.
\end{prop}

\begin{proof}
	Let $G$ and $G'$ be two equivalent graphs, where $G'$ is obtained from~$G$ by shifting the edge $e=u_1u_2$ to $v\in G$ (which may be the same vertex as $u_1$~or~$u_2$) according to~Definition~\ref{d.shifting}. Let $\widetilde{G}$~be the intermediate graph obtained after Step~1, containing both $e$ and $e'=v_1v_2$.
	
	We claim that an arbitrary spanning subgraph~$G_i$ of~$G$, in which $i\in[1\twodots m]$ of the edges are removed, is connected if and only if the corresponding subgraph~$G_i'$ of~$G'$ is connected. This claim implies that~\eqref{equivalent} holds for $G$ and $G'$, and by induction for $G$ and any equivalent graph~$H$. It suffices to prove the \emph{only~if} direction of the claim; the \emph{if} direction then follows because of the symmetric relationship between $G$ and~$G'$.
	
	To this end, assume $G_i$~to be connected. If $e\in G_i$, then $G'_i$~is connected, since $G'_i$~is obtained from~$G_i$ by contracting~$e$ and expanding~$e'$, which preserves connectedness. 
	Suppose on the contrary that $e\notin G_i$. There is a $u_1u_2$‑path $\Pi$ in~$G_i$, since $G_i$~is connected. In $\widetilde{G}-e$, the bridge $e' = v_1v_2$ separates $u_1$ and~$u_2$, and from this we deduce that $v\in \Pi$ in~$G_i$. Now, the image in~$G'_i$ of the set of edges in~$\Pi$ is a $v_1v_2$‑path in~$G'_i$. Hence the connectivity of~$G'_i$ is the same as that of $G'_i \cup e'$. Since $G_i$ is connected, $G_i \cup e$ is connected, which implies, by the first case, that $G'_i \cup e'$ is connected, which implies that $G'_i$~is connected.
\end{proof}

\subsection{Distillations to represent bridgeless graphs}
\label{ss.represent}
Theorem~\ref{t.3ec} below---in which part~\ref{tt.3ec2} is the most important---is a continuation of Proposition~\ref{p.cubic}. The theorem will allow us to focus on 3‑edge-connected cubic distillations, for which we now introduce a special notation. (Note that all 3‑edge-connected cubic graphs are simple, except for the 3‑dipole.)

\begin{dfn}[$\D_k$]
	\label{d.Dk}
	For $k\geq1$, let $\smash{\D_k}$ denote the set of 3‑edge-connected cubic graphs (which necessarily are distillations) of exceedance~$k$.
\end{dfn}

Figure~\ref{f.Dgraphs} shows the graphs in $\D_1$ through~$\D_4$, in other words the 3‑edge-connected cubic graphs on up to 8 vertices. Theorem~\ref{t.3ec} says that, up to graph equivalence, every bridgeless graph of exceedance 1,~2,~3 or~4 is a weak subdivision of at least one of these respective graphs. When $k > 4$, the $\D_k$‑sets start to become impractically large.  
There are fourteen 3‑edge-connected cubic graphs constituting~$\D_5$. See~\cite[pp.~56–57]{cubic}.  
Of these, the most promising distillation is known as Petersen's graph, as is discussed further in Section~\ref{ss.k=5}.

\begin{figure}[b!]
	\centering
	\includegraphics{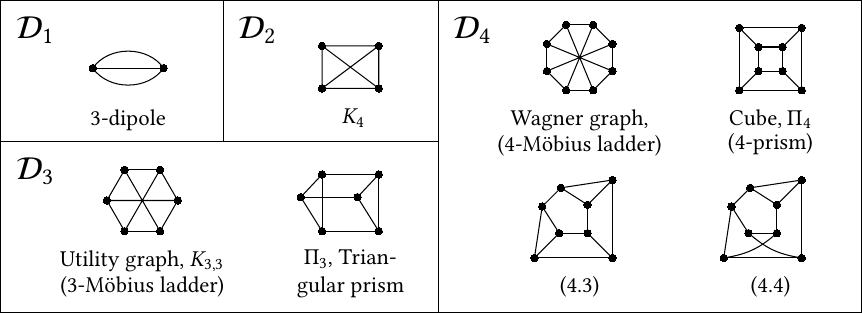}
	\caption{All 3‑edge-connected cubic graphs of exceedances 1,~2,~3 and~4.}
	\label{f.Dgraphs}
\end{figure}

\begin{thm}
	\label{t.3ec}
	Let \(G\) be a bridgeless graph in \(\G_{n,n+k}\), where \(n\geq 1\) and \(k\geq 1\).
	\begin{enumenun}
		\item\label{tt.3ec1} \(G\) has a weak distillation in~\(\D_k\) if and only if all 2‑bonds of~\(G\) are trivial.
		\item\label{tt.3ec2} \(G\) is equivalent to some graph \(G^\dag\) which has a weak distillation \(D\in\D_k\).
		\item\label{tt.3ec3} If \(G\) does not have a weak distillation in~\(\D_k\), then any \(G^\dag\) as in~\ref*{tt.3ec2} is imbalanced and has at least one zero chain with respect to its weak distillation(s) in~\(\D_k\).
	\end{enumenun}
\end{thm}

\begin{rmk}
	The weak distillation in part~\ref*{tt.3ec1} is not unique. Consider the two graphs $\Pi_3$ and~$K_{3,3}$, which constitute~$\D_3$, with the edge labeling of Figure~\ref{f.Pi3} in Section~\ref{ss.Pi3}. Contracting~$c_1$ in the two distillations yields the same graph, which therefore has two weak distillations in~$\D_3$.
\end{rmk}

\begin{proof}[Proof of part~\ref*{tt.3ec1}]
	For the \emph{only~if} direction, suppose that $G$ has a weak distillation $D\in\D_k$. Recall that edge-connectivity is nondecreasing under edge contraction. Since $D$ is 3‑edge-connected, $D_G$ has no 2‑bonds, so $G$ can only have trivial 2‑bonds.
	
	For the \emph{if} direction, given a bridgeless~$G$ with only trivial 2‑bonds, its proper distillation~$D$ is 3‑edge-connected. If $D$ is already cubic, there is nothing to prove. Suppose otherwise, and let $u$ be a vertex of~$D$ with degree at least~4. We claim that it is possible to expand an edge at~$u$ while keeping 3‑edge-connectedness. There are two cases to consider.
	
	\emph{Case 1:} $u$ is a cutvertex. Let $e_a$~and~$e_b$ be edges incident to~$u$ such that $e_a$~belongs to a~block which we call~$A$, and $e_b$~belongs to a different block~$B$, as in Figure~\ref{f.cubex1}. Since $D$ has no bridge, there is a cycle~$C_a$, contained in~$A$, starting with~$e_a$, and another cycle~$C_b$, contained in~$B$, which ends with~$e_b$. Now expand an edge $e_u=u_1u_2$ at~$u$, letting $e_a$~and~$e_b$ be incident to~$u_1$ and the other edges to~$u_2$. Call the resulting graph~$D'$.
	
	\begin{figure}[t!]
		\centering
		\includegraphics{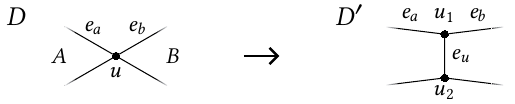}
		%
		%
		\caption{How to make a 3‑edge-connected distillation~$D$ cubic by edge expansion, case~1: If $u$ is a~cutvertex, make sure that the new edge~$e_u$ is a~chord of a~cycle.}
		\label{f.cubex1}
	\end{figure}
	
	We now show that $D'$ is 3‑edge-connected. A~bond in~$D'$ which does not involve~$e_u$ is obviously a bond in~$D$, so since $D$ is 3-edge-connected, any 1‑~or~2‑bond in~$D'$ would have to include~$e_u$. In~$D'$, $e_u$~is a chord of the cycle~$C_a \cup C_b$, so $e_u$ is not a~1‑bond and $D'$ is therefore bridgeless. Furthermore, deleting a~chord does not change the block structure, so $D'-e_u$ is also bridgeless. This implies that $D'$ is 3‑edge-connected.
	
	\emph{Case 2:} $u$ is not a cutvertex. Expanding an edge $e_u=u_1u_2$ at~$u$ so that both $u_1$~and~$u_2$ have degree at least~3 can be done in several ways. Since $u$ is not a cutvertex, this surgery cannot create a~bridge, but it might create a~2‑bond, as in Figure~\ref{f.cubex2}.
	\begin{figure}[t!]
		\centering
		\includegraphics{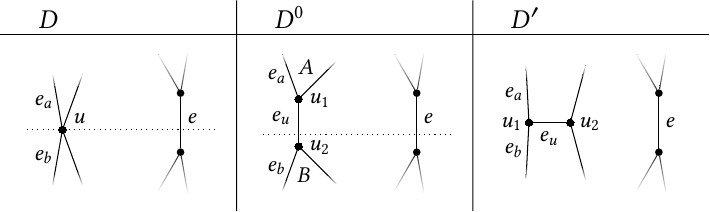}
		\caption{How to make a 3‑edge-connected distillation~$D$ cubic by edge expansion, case~2: If $u$ is not a cutvertex, a 2‑bond can possibly arise (as in $D^0$), but can also be avoided (as in~$D'$).}
		\label{f.cubex2}
	\end{figure}
	If it does not create a 2‑bond, we are done. On the other hand, suppose that expanding~$e_u$ creates the 2‑bond $\{e_u,e\}$ in the resulting graph~$D^0$.
	This implies that $u$ is a cutvertex in the graph~$D-e$. Noting that $D-e$ is bridgeless, we expand a new edge~$e_u$ in this graph, exactly as described in Case~1 (with $D-e$ instead of~$D$). We call the resulting graph $D'-e$, and then restore~$e$ to obtain~$D'$. The proof that $D'$ is 3‑edge-connected follows Case~1 verbatim.
	
	Finally, we note that in both Case~1 and Case~2, the new vertices $u_1$ and~$u_2$ have degrees between 3 and $(\deg(u)-1)$, thereby lowering the average degree of the distillation. Thus, we can repeat the above procedure until we obtain a 3‑edge-connected cubic weak distillation of~$G$.
\end{proof}

\begin{proof}[Proof of part~\ref*{tt.3ec2}]
	If~$G$ itself has no nontrivial 2‑bonds, then the statement is true by part~\ref*{tt.3ec1}. Suppose on the other hand that $G$~has at least one nontrivial 2‑bond, and let $a$ and~$b$ denote the positive chains containing the bond. Shift all the edges in chain~$b$ to chain~$a$, one by one, according to Definition~\ref{d.shifting} (one can always use an endvertex of~$a$ to expand an edge into the chain). We call the resulting graph~$G'$, which is equivalent to~$G$ by Definition~\ref{d.equiv}.
	
	We now repeat the procedure above for any remaining nontrivial 2‑bond of~$G'$. Since $G'$ has one less positive chain than~$G$, this process eventually terminates, at which point we have obtained a graph~$G^\dag$ which is equivalent to~$G$ and whose only 2‑bonds are trivial. $G^\dag$ is bridgeless since edge shifts preserve bridgelessness. Using part~\ref*{tt.3ec1}, this proves part~\ref*{tt.3ec2}.
\end{proof}

\begin{proof}[Proof of part~\ref*{tt.3ec3}]
	Given $G$ and an equivalent graph~$G^\dag$ with a weak distillation $D\in \D_k$, we show the contrapositive version of~\ref*{tt.3ec3} by demonstrating that if $G^\dag$ is balanced or has only positive chains with respect to~$D$, then $G$ has a weak distillation in~$\D_k$. There are two cases to consider.
	
	\emph{Case 1:} $G^\dag$ has only positive chains relative to~$D$. This implies that $D$ is the proper distillation of~$G^\dag$. Since $D$ is 3‑edge-connected and cubic, it is easy to see that only trivial edge shifts can be performed in~$D$ and hence also in~$G^\dag$ (which is to say that each edge can only be shifted within its chain; see the discussion after Definition~\ref{d.shifting}). Hence, $G^\dag$~is the only graph in its equivalence class and the conclusion is immediate.
	
	\emph{Case 2:} 
	$G^\dag$ is balanced and has at least one zero chain relative to~$D$. This implies that every chain of~$G^\dag$ has either length one or zero. Suppose that there is a nontrivial way to shift an edge~$e$ in~$G^\dag$, otherwise there is nothing to prove. We first expand an edge~$e'$ so that $\{e,e'\}$ is a nontrivial 2‑bond in the resulting graph~$\widetilde{\Gamma}$.
	
	We claim that $\{e,e'\}$ is the only 2‑bond of $\widetilde{\Gamma}$. To see this, first note that $G^\dag$ is 3‑edge-connected, since $D$~is, and hence any 2‑bond of~$\widetilde{\Gamma}$ necessarily involves~$e'$. Now, suppose that $e''$ is some third edge such that $\{e',e''\}$ is a 2‑bond in~$\widetilde{\Gamma}$. Then, the symmetric difference between $\{e,e'\}$ and~$\{e',e''\}$  is also a 2‑bond in~$\widetilde{\Gamma}$ (since it is a~cut and since $\widetilde{\Gamma}$ is bridgeless). This 2‑bond does not contain~$e'$, which is a contradiction.
	
	Let $\Gamma'$ be the graph obtained by contracting~$e$ in~$\widetilde{\Gamma}$. Since the only 2‑bond of~$\widetilde{\Gamma}$ is $\{e,e'\}$, it follows that $\Gamma'$ is 3‑edge-connected, and~so $\Gamma'$ has a weak distillation~$D'$ in~$\D_k$ by part~\ref{tt.3ec1}. Since $\widetilde{\Gamma}$ has no chains with more than one edge, the same holds for~$\Gamma'$, and hence $\Gamma'$ is balanced with respect to~$D'$. Furthermore, since $\Gamma'$ has a vertex of degree at least four from the contraction of~$e$, $\Gamma'$ has at least one zero chain relative to~$D'$. Since the pair $(\Gamma',D')$ remains in Case~2, it follows that every graph equivalent to~$G^\dag$, and $G$ in particular, has a weak distillation in~$\D_k$.
\end{proof}

\section{Minimizing 2- and 3-disconnections}
\label{s.minimize}

\subsection{Minimizing 2‑disconnections for \texorpdfstring{$k\geq 1$}{k≥1}}
Necessary and sufficient conditions for an $(n,m)$-graph to minimize~$d_2(\ph)$ have been known since~\cite{Bauer}, at least in the context of simple graphs. Our framework allows for a more unified formulation 
with a shorter proof. The second part of Theorem~\ref{t.d2} below describes a particular surgery which yields a lower $d_2$‑value, and will be used repeatedly in what follows.

\begin{thm}\label{t.d2}\ 
	\begin{enumenun}
			\item\label{tt.d2iff}
			For \(k\geq 1\), a graph \(G\in\G_{n,n+k}\) minimizes~\(d_2(\ph)\) if and only if \(G\) is a balanced weak subdivision of some \(D\in\D_k\) (``weak'' is not needed when \(n\geq 2k\)).
			\item\label{tt.d2decrease}
			If \(G\) is an imbalanced weak subdivision of \(D\in\D_k\), then a strictly \(d_2\)‑decreasing surgery is to choose a pair of imbalanced chains and move one edge from the longer to the shorter chain.
		\end{enumenun}
	\end{thm}

\begin{proof}
	The proof of the \emph{only~if} direction of part~\ref*{tt.d2iff} will also prove~\ref*{tt.d2decrease}. We suppose that $G\in\G_{n,n+k}$ is not a balanced weak subdivision of any distillation in~$\D_k$ and show that $G$ does not minimize~$d_2(\ph)$. There are two cases: 
	
	\emph{Case~1:} $G$ is not a weak subdivision of any graph in~$\D_k$. By Theorem~\ref{t.3ec}\ref{tt.3ec1}, $G$~either has a~bridge or is bridgeless but has nontrivial 2‑bond. If $G$ has a~bridge, then by Proposition~\ref{p.bridge} there exists a $G'\in \G_{n,n+k}$ such that $d_2(G') < d_2(G)$. If $G$ is bridgeless but has a nontrivial 2‑bond, then by Theorem~\ref{t.3ec}\ref{tt.3ec2} there is an equivalent graph~$G^\dag$ with a weak distillation in~$\D_k$, and by part~\ref{tt.3ec3}, $G^\dag$~is imbalanced with respect to this weak distillation. Since $d_2(G^\dagger) = d_2(G)$ by Proposition~\ref{p.equiv}, we can conclude through Case~2 below that $G$~is not $d_2$‑minimizing.
	
	\emph{Case~2:} $G$~is an imbalanced weak subdivision of some $D \in \D_k$. Let $a$~and~$b$ be the lengths of two imbalanced chains of~$G$ relative to~$D$ such that $a > b + 1$, and let~$G'$ be the graph obtained by moving one edge from the $a$‑chain to the $b$‑chain. By Theorem~\ref{t.3ec}\ref{tt.3ec1}, $G$~and~$G'$ have only trivial 2‑bonds, so $d_2(\ph)$ is just the number of pairs of edges which belong to the same chain. Thus, we obtain
	\vspace{-2mm}
	\begin{equation}
		\label{d2}
		d_2(G) - d_2(G') = \binom{a}{2} + \binom{b}{2} - \binom{a-1}{2} - \binom{b+1}{2}\,,
	\end{equation}
	which simplifies to $a - (b+1)$ and is positive by our assumption on $a$~and~$b$. This proves the \emph{only~if} direction of~\ref*{tt.d2iff}, as well as~\ref*{tt.d2decrease}. 
	
	For the \emph{if}~direction of~\ref*{tt.d2iff}, using the \emph{only~if} direction, it suffices to show that $d_2(\ph)$ is constant over $m$‑sized balanced weak subdivisions of distillations in~$\D_k$. Let $G\in \G_{m-k,m}$ be a balanced weak subdivision of $D\in\D_k$. Recall that $D$ has $3k$~edges, so $G$~has $3k$~chains. Since these chains are balanced, we can interpret the Euclidean division $m = 3kq + r$, where $0 \leq r < 3k$, as saying that the ``base length'' of the $3k$~chains is~$q$, while $r$~of the chains are one edge longer (cf.~\eqref{euclid} below). Since $d_2(G)$, again by Theorem~\ref{t.3ec}\ref{tt.3ec1}, equals the number of ways to choose two edges within the same chain, we have
	\vspace{-1mm}
	\[ d_2(G) = (3k-r) \binom{q}{2} + r \binom{q+1}{2} \,. \qedhere \]
\end{proof}

\subsection{Results about 3‑disconnections for \texorpdfstring{$k\geq 2$}{k≥2}}
\label{ss.3dis}
A general statement about minimizing the number of 3‑disconnections can be found in~\cite{Wang97}; however, we find multiple reasons to consider this paper unreliable. In particular, the proof of \cite[Theorem~10(c)]{Wang97} lacks a coherently structured argument and misleadingly relies upon a particular picture. (One may note that the only other known paper by its corresponding author has an erroneous main conclusion with a multiply flawed proof, as shown in Section~\ref{ss.K33}.)

The main work of this section is to prove Proposition~\ref{p.d3min}, which continues into Theorem~\ref{t.d3min2}. These results can be compared to Theorem~10(c) and 10(a)(c) in~\cite{Wang97}, respectively. First, we need some new notation.

Let $\B_m(D)$ denote the set of $m$‑sized balanced weak subdivisions of~$D$. A~\emph{balanced weighting} of~$D$ is a weighting which yields a graph $G\in\B_m(D)$, for some~$m$, when the weights are interpreted as chain lengths. The weighting of~$D$ is uniquely determined by~$G$ (up to isomorphism) except in some degenerate cases involving several zero chains, as exemplified in Example~\ref{ex.weightings}.

Consider a cubic distillation~$D$ of exceedance~$k$ and a graph $G\in \B_m(D)$.  Recalling that $D$ has size~$3k$, the $m$~edges of~$G$ are distributed as evenly as possible over $3k$~chains. We wish to define a ``standard chain length'' or ``standard weight'' $q\geq 0$, such that at least half of the edges of~$D$ have weight~$q$. Using Euclidean division with a ``centered remainder'', we define $q$~and~$r$ according to 
\begin{equation}
	\label{euclid}
	m = 3kq + r \qquad\mathrlap{(-3k/2<r\leq 3k/2)\,,}
\end{equation}
so that $D$ has $3k-\abs{r}$ edges of weight~$q$ and $\abs{r}$~edges of either weight~\(q-1\) (if \(r\)~is negative) or weight~\(q+1\) (if $r$ is positive).

\begin{dfn}[\kern1pt $b_3^\rv$ and $\piv$]
	\label{d.isolation}
	Let $D\in \D_k$ where $k\geq2$ and consider a balanced weighting of~$D$. With $q$ as above, we define:
	\begin{enumenun}
		\item The product of the weights incident to a vertex $u\in D$ is denoted by~$b_3^\rv(u)$. 
		\item The $\piv$-value of a vertex $u_i\in D$, denoted $\piv(u_i)$, is the number of edges incident to~$u_i$ which do not have weight~$q$, counted with sign in that $q+1$ is counted positively and $q-1$ negatively. The $\piv$-values of~$D$ yield a multiset denoted by $\mset{\piv_i}_{i=1}^{2k}$. The $\piv$-values are \emph{balanced} if the maximum pairwise difference of its values is one.
	\end{enumenun}
\end{dfn}

\begin{exmp}
	\label{ex.weightings}
	When some edges have zero weight, different weightings of~$D$ can yield the same weak subdivision. Let $D=K_4$ and consider the 3‑bouquet (a vertex with three loops) which can be obtained from~$K_4$ either by contracting a~3‑path 
	or by contracting a~3‑star. In both cases, $K_4$~has three edges with weight $q=0$ (contracted) and three with weight $q+1$. However, the weight arrangement of the former case gives the $\piv$-multiset $[1,1,\ab 2,2]$, while that of the latter gives the multiset $[0,2,2,2]$.
\end{exmp}

\begin{prop}
	\label{p.d3min}
	Let \(m > k \geq 2\). Suppose that \(D\in\D_k\) has only trivial 3‑bonds and consider a balanced weighting of~\(D\) together with the implied \(G\in \B_m(D)\).
	\begin{enumenun}
		\item\label{pp.d3min1} If \(D\) has balanced \(\piv\)‑values, then \(G\) minimizes \(d_3(\ph)\) in \(\B_m(D)\).
		\item\label{pp.d3min2} If \(m\geq 2k\), then \(G\) minimizes \(d_3(\ph)\) in \(\B_m(D)\) if and only if \(D\) has balanced \(\piv\)-values.
	\end{enumenun}
\end{prop}

\begin{proof}[Proof of part~\ref*{pp.d3min1}]
	We first note that there actually exists a $D\in\D_k$ with only trivial 3‑bonds; consider the $k$‑Möbius ladder of Definition~\ref{d.mobius}. Fix $m\ge k$, which determines $r$ and~$q$ according to \eqref{euclid}. If $r=0$, then there is only one graph in~$\B_m(D)$ and the $\piv$-values of~$D$ are all zeros, which makes both propositions trivial. We can therefore assume that $r\neq 0$.
	
	Recall that $D$ has $2k$~vertices, since $D$ is cubic, which we label $(u_i)_{i=1}^{2k}$. Since $D$ is 3‑edge-connected, every 3‑disconnection of~$G$ is either a~3‑bond or contains a trivial 2‑bond. For every chain of length~$\ell$, there are $\binom{\ell}{2}(m-\ell)$ disconnections of the latter kind. The chain lengths are given by $q$~and~$r$, so we have
	\begin{equation}
		\label{d3g}
		d_3(G) = K + b_3(G)\,, \qquad\mathrlap{G\in\B_m(D)}
	\end{equation}
	where $K$ does not depend upon~$G$ nor on $D\in \D_k$. Since all 3‑bonds of~$D$ are trivial, all 3‑bonds of~$G$ are. Hence,
	\begin{equation}
		\label{b3G}
		b_3(G) = \sum_{\mathclap{i\in [1\twodots 2k]}} b_3^\rv(u_i) \,.
	\end{equation}
	
	Consider the possible values of~$b_3^\rv(u)$, given in Table~\ref{tb.3bonds}. Clearly, $b_3(G)$ in~\eqref{b3G} is a cubic polynomial in~$q$ and the leading coefficient is~$2k$. We also claim that the coefficient of the quadratic term equals~$2r$. To see this, first note that $\piv(u_i)$ has the same sign as~$r$, and then consider that there are $\abs{r}$~edges in~$D$ which are potentially counted by~$\piv(u_i)$. As $u_i$ ranges over the vertices of~$D$, each of these edges is counted twice, which implies that 
	\begin{equation}
		\label{pi2r}
		\sum_{\mathclap{i\in[1 \twodots 2k]}} \piv(u_i) = 2r\,. 
	\end{equation}
	The claim now follows by noting that the quadratic coefficient of~$b_3^\rv(u)$ equals~$\piv(u)$ in each row of the table.
	
	\begin{table}[tb!]
		\renewcommand*{\arraystretch}{1}
		\caption{Functions for counting and comparing the number of trivial 3‑bonds in balanced weak subdivisions of cubic, 3‑edge-connected graphs.}
		\begin{tabularx}{\textwidth}{C L X \Lr\Ll X \LX}
			\toprule
			\quad &\bm{\piv(u)} &&& \bm{b_3^\rv(u)} && \bm{\varphi_3(\piv)}  \\
			\midrule
			&\phantom{\pm}0 && \phantom{(q\pm1)^2}q^3 &{}= q^3 && 0  \\
			&\pm1 && (q\pm1)^1q^2 &{}= q^3 \pm q^2 && 0 \\
			&\pm2 && (q\pm1)^2q^1 &{}= q^3 \pm 2q^2 + q && q \\
			&\pm3 && (q\pm1)^3 &{}= q^3 \pm 3q^2 + 3q \pm 1 && 3q \pm 1  \\
			\bottomrule
		\end{tabularx}
		\label{tb.3bonds}
	\end{table}
	
	Let $\varphi(\piv)$---also denoted by $\varphi_3(\piv)$, see Table~\ref{tb.3bonds}---be the constant and linear terms of~$b_3^\rv$ regarded as a function of~$\piv$. Combining the above observations about~$b_3(G)$ with~\eqref{d3g} yields
	\begin{equation}
		\label{phisum}
		d_3(G) = K + 2kq^3 + 2rq^2 + \sum_{\mathclap{i\in[1 \twodots 2k]}} \varphi(\piv_i)\,, \qquad\mathrlap{G\in\B_m(D)}
	\end{equation}
	where only $\piv_i = \piv(u_i)$ depends upon the choice of~$G$. Hence, $G$ minimizes $d_3(\ph)$ within our class if and only if $G$ minimizes $\sum_i \varphi(\piv_i)$.

	Using Table~\ref{tb.3bonds}, it is easy to verify the following three pairs of inequalities. (Recall that the $\pm$~signs correspond to whether $r$ is positive or negative.)
	\begin{align}
		\varphi(0) + \varphi(\pm2) &\geq 2\varphi(\pm1) \label{ineq1}\\
		\varphi(0) + \varphi(\pm3) &\geq \varphi(\pm2) + \varphi(\pm1) \label{ineq2}\\
		\varphi(\pm1) + \varphi(\pm3) &\geq 2\varphi(\pm2) \label{ineq3}
	\end{align}
	The inequalities imply that the sum $\sum_i \varphi(\piv_i)$ is bounded below by the value the sum would have if the $\piv_i$'s were replaced by a balanced multiset also satisfying~\eqref{pi2r}. This implies part~\ref*{pp.d3min1}.
\end{proof}

\begin{proof}[Proof of part~\ref*{pp.d3min2}.]
	In view of~\ref*{pp.d3min1}, we need only show the \emph{only~if} direction. This is accomplished by proving the following two statements, given $m\ge 2k$.
	\begin{eninline}[label=\textbf{\Alph*:}, ref=\Alph*]
		\item\label{apriori} Of the \emph{a~priori} possible $\piv$-multisets, which satisfy~\eqref{pi2r}, only the balanced one minimizes $\sum_i \varphi(\piv_i)$.
		\item\label{exists} There exists a weighting of~$D$ which yields a balanced $(m-k,m)$-graph and balanced $\piv$-values.
	\end{eninline}
	
	Statement~\ref*{apriori} would follow immediately if \eqref{ineq1}, \eqref{ineq2} and \eqref{ineq3} were strict inequalities. This is close to being true. We first restrict to the case where $m \geq 3k$. This implies that $q\geq2$ or that $q=1$ and $r\ge 0$, and in these cases the inequalities are strict (recalling our assumption that $r\neq 0$).
	
	The remaining case is when $2k\leq m < 3k$, which implies $q=1$ and $r\in[-k\twodots -1]$. The positive $q$‑value guarantees that \eqref{ineq1} is strict. To see that this suffices to prove~\ref*{apriori}, consider the following: The $\piv$-multiset sums to~$2r$ and~so has mean~$r/k$, which implies that a balanced $\piv$-multiset in this case only contains the numbers $-1$ and~$0$. Hence, if one takes an imbalanced multiset belonging to this case and transforms it step by step into the corresponding balanced multiset, then \eqref{ineq1} will apply at the last step.
	
	We now prove statement~\ref*{exists}. The distillation $D$ is bridgeless and cubic since $D\in\D_k$, and~so by Petersen's Theorem \cite[Cor.~2.2.2]{Diestel} there is a perfect matching in~$D$. Let $E_1$ denote the $k$~edges of a perfect matching and let $E_2$ denote the remaining $2k$~edges.
	
	With this, we can specify how to weigh the edges of~$D$ to obtain balanced $\piv$-values and a resulting weak subdivision~$G$. We start by assigning weight~$q$ to all edges; we will then change $\abs{r}$~edges to either weight $q-1$ or $q+1$. Care must be taken so that we do not create a cycle with all zero weights, since this does not yield a weak subdivision of~$D$. This concern arises only when $2k\leq m < 3k$, so that $q=1$ and $r\in[-k \twodots -1]$. Otherwise, things are analogous for positive and for negative $r$, and we therefore consider the case where $r$ is negative.
	
	If on the one hand $r\in[-k \twodots -1]$, let the $r$~edges with weight $q-1$ be arbitrarily chosen from~$E_1$. Since these edges are independent, each vertex of~$D$ is incident with at most one edge with weight $q-1$, and~so the $\piv$-multiset contains only the values $-1$~and~$0$. There is no cycle of edges weighted $q-1$, and hence we have specified a weak subdivision of~$D$.
	
	If on the other hand $r\in [\ceil{-3k/2}\twodots k-1]$, we start by assigning weight $q-1$ to the $k$~edges of the perfect matching~$E_1$, which guarantees that each $\piv$-value is at least~$-1$. We now need to assign weight $q-1$ to another $\abs{r}-k \leq \ceil{k/2}$ edges (this bound holds also for the case where $r$ is positive) in such a~way that no vertex obtains $\piv$-value $-3$. This is equivalent to choosing a matching of size~$\abs{r}-k$ in~$E_2$. The graph~$E_2$ is 2‑regular, and is hence a union of cycles; each with at least three edges since $D$ is simple, and with a total number of $2k$~edges. This makes it is immediate that $E_2$ has a matching of size~$\ceil{k/2}$, and hence a matching of size $\abs{r}-k$, as required.
\end{proof}

\begin{thm}
	\label{t.d3min2}
	Given \(k\geq 2\) and \(m\geq 3k\), consider a balanced weighting of \(D\in\D_k\) which implies a graph \(G\in\B_m(D)\). Then \(G\) minimizes \(d_3(\ph)\) within the larger set \(\B_m(\D_k)\) if and only if all 3‑bonds of \(D\) are trivial and its \(\piv\)‑values balanced. (The latter condition is vacuous when $3k\divides m$.)
\end{thm}

\begin{proof}
	A minor modification of and addition to the proof of Proposition~\ref{p.d3min} suffices. As noted, \eqref{d3g}~holds for any balanced weak $(m-k,m)$-subdivision~$G$ of any graph $D \in \D_k$, where $K$ is independent of~$G$ and~$D$. However, \eqref{b3G}, as well as \eqref{phisum}, has to be modified by adding a~term accounting for the number of nontrivial 3‑bonds of~$G$, which we may denote by $b_3^\n(G)$. By combining a~few terms in~\eqref{phisum} and adding $b_3^\n(G)$, we obtain
	\begin{equation}
		\label{phisum2}
		d_3(G) = K' + \Phi(G) + b_3^\n(G)\,, \qquad\mathrlap{G\in\B_m(\D_k)}
	\end{equation}
	where $K'$~denotes the first three terms of~\eqref{phisum}, which depend neither on~$G$ nor on~$D$, and $\Phi(G)$ denotes the ``$\varphi$‑sum'' which depends only upon the $\piv$-multiset. The previous proof goes through to show that a balanced $\piv$-multiset is sufficient and, if $m\geq2k$, necessary to minimize~$\Phi(\ph)$.
	
	Regarding $b_3^\n(G)$, suppose that $m\geq 3k$; then $G$, being balanced, has only positive chains relative to~$D$, and~so $D$ is~$D_G$, the proper distillation of~$G$. Clearly, $G$ has a nontrivial 3‑bond if and only if $D_G$ has a nontrivial 3‑bond. It follows that $d_3(G)$ is minimized, within the set under consideration, if and only if $D_G$ has balanced $\piv$-values with respect to~$G$ and in addition no nontrivial 3‑bonds.
\end{proof}

The following would be a natural extension of Theorem~\ref{t.3ec}\ref{tt.3ec1}. If true, then Theorem~\ref{t.d3min2} could be extended to hold for $m\geq 2k$. The \emph{only~if} direction is easy.

\begin{qst}
	For \(k\geq2\), let \(G\) be a bridgeless graph in \(\G_{n,n+k}\). Does \(G\) have a weak distillation in~\(\D_k\) with only trivial 3‑bonds if and only if all 2‑bonds and 3‑bonds of~\(G\) are trivial?
\end{qst}

\section{Moving edges between chains}
\label{s.move}
We will often want to move an edge from a longer chain to a shorter chain and study how the number of disconnecting sets of some size~$i$ changes. Suppose that we obtain $G'$ from~$G$ by moving an edge between chains. In principle, it would be straightforward to express and compare $c_i(G')$ and $c_i(G)$. Each term in~$c_i(\ph)$ is a product of $i$ chain lengths (cf.~\eqref{k1}) and corresponds to a connected spanning subgraph of~$D_G$ where $i$~edges are missing. However, the expressions become impractical, and the calculations do not seem very helpful to the intuition. We will prefer to count the number of edge sets which disconnect \emph{either $G$ or~$G'$}.

Necessary and sufficient conditions for an edge set to disconnect either $G$ or~$G'$, and hence be relevant for the difference $d_i(G)-d_i(G')$, are laid out in Lemma~\ref{l.move}, and the application is in Corollary~\ref{c.move}. (The proof of the lemma follows the corollary.) The reader might want to look ahead at Example~\ref{ex.li} for a demonstration.

\begin{lem} 
	\label{l.move}
	Move an edge~\(e\) within a given connected graph~\(G\), that is, contract~\(e\) and expand a new edge, identified with~\(e\), at an arbitrarily chosen vertex. Call the result \(G'\). Then an edge set~\(E\) disconnects~\(G\) but not~\(G'\) if and only if \(E\) satisfies both of the following conditions.
	\begin{enumenun}[label=(\arabic*)]
		\item\label{ll.moveG} In \(G\), the set~\(E\) contains exactly one bond~\(B\), and \(e\in B\).
		\item\label{ll.moveG'} In \(G'\), there is no bond containing~\(e\) and contained in~\(E\).
	\end{enumenun}
\end{lem}

\begin{rmk}
	If condition~\ref*{ll.moveG} of the lemma holds, then \ref*{ll.moveG'}~is equivalent to the following simpler but less useful condition:
	\begin{eninline}[label=(\arabic*')]
		\setcounter{eninlinei}{1}
		\item \(E\) does not disconnect~\(G'\).
	\end{eninline}
\end{rmk}

\begin{cor}
	\label{c.move}
	Let \(G'\) be obtained from~\(G\) as in the preceding lemma. For \(i \in [1 \twodots m]\) and \(1\leq j\leq i\), let \(x_{i,j}\) be the number of edge sets of size~\(i\) which fulfill conditions \ref*{ll.moveG} and \ref*{ll.moveG'} above and where the bond~\(B\) has size~\(j\). Let \(x'_{i,j}\) be the corresponding number of edge sets, but with the roles of \(G\) and \(G'\) interchanged in conditions \ref*{ll.moveG} and \ref*{ll.moveG'}. Then
	\begin{equation}
		\label{ddiff}
		d_i(G) - d_i(G') = \sum_{j=1}^i x_{i,j} - x'_{i,j}\,, \quad \mathrlap{\forall i \in [1 \twodots m]\,.}
	\end{equation}
\end{cor}


\begin{proof}[Proof of Lemma~\ref{l.move}]
	We consider different possibilities for the edge set~$E$, and show for each case that the two sides of the biconditional (if and only if statement) of the lemma either both hold or both fail.
	
	\emph{Case 1:} $E$ does not disconnect~$G$. It is immediate that neither side of the biconditional holds.
	
	\emph{Case 2:} $E$ contains some bond~$B$ in~$G$ such that $e\notin B$. Then $B$ is a bond also in $G/e = G'/e$, and thus $B$ is a bond in~$G'$. Again, both sides of the biconditional fail.
	
	\emph{Case 3:} $E$ contains exactly one bond~$B$ in~$G$, and $e\in B$. By using the reasoning of Case~2 with $G$ and $G'$ interchanged, we deduce that if $E$ contains a bond in~$G'$, then this bond necessarily contains~$e$. If there is such a bond in~$G'$, both sides of the biconditional fail. If there is no such bond, then both sides hold.
	
	To see that these three cases exhaust the possibilities for~$E$, suppose that $E$ contains two bonds in~$G$, both of which include~$e$. Then the symmetric difference of these bonds is a~cut which does not include~$e$, and this is covered by Case~2.
\end{proof}

\section{The uniquely optimal \texorpdfstring{$(\MakeLowercase{k} = 2)$}{(k=2)}-graphs}
\label{s.k=2}
We are now equipped to consider the sets of $(n,n+2)$-graphs. The optimal graphs of these sets have previously been described for $n\geq 4$ or $5$ in \cite{Li}, \cite{Tseng}, \cite{Wu} and \cite{Tutte}. The graphs are described in three steps, and~so it~seems that a treatment of the $(k=2)$-case needs to contain the following, in one way or another.
\begin{enumerate}[label={Step \arabic*:}, ref={Step \arabic*}]
	\item\label{i.identify} Identification of~$K_4$ as the relevant distillation.
	\item\label{i.balance} Proof that a $K_4$-subdivision has to be balanced to be possibly $p$‑optimal.
	\item\label{i.arrange} Specification of how longer or shorter chains need to be arranged when the chain lengths are not exactly the same.
\end{enumerate}

The literature contains a~few significant, previously undetected errors relating to \ref*{i.balance} and (less critically) to \ref*{i.identify}:
\begin{itemize}
	\item The middle part of the proof of Theorem~4 in~\cite{Li}, which is the standard reference for the $(k=2)$-problem, aims to show the above \ref*{i.balance}. However, in Example~\ref{ex.li} below we give counterexamples to the central claim.
	
	\item In both \cite{Tseng} and~\cite{Wu}, \ref*{i.balance} is justified by the solution to the corresponding continuous constrained optimization problem. It is not clear to us how the integer-valued solution follows from the continuous, except when they coincide.
	
	\item One distillation ($D_3$ in Figure~\ref{f.distillations}) out of four ($D_1$ through~$D_4$) is missing from the analysis of~\cite{Li}. (It was consequentially also missing from the implicit treatment via Tutte polynomials~\cite{Tutte}, but this was easily corrected by the authors, who provided an updated version~\cite{KahlUpdate}.)
\end{itemize}

Furthermore, our treatment brings the following two advantages. First, the only proper work needed within our framework is for proving \ref*{i.balance}. (\ref*{i.identify} is essentially given from Theorem~\ref{t.3ec}\ref{tt.3ec2} and \ref*{i.arrange} from Proposition~\ref{p.d3min}.) Second, graphs with multiple edges are naturally covered, needing no separate statements or proofs. (Gross and Saccoman~\cite{Gross} are usually credited with extending the main $(k=2)$-result to multigraphs by way of a separate argument.)

\begin{exmp}
	\label{ex.li}
	This example provides an infinite set of counterexamples to an erroneous deduction in~\cite{Li} which implies \ref{i.balance}. We focus on one particular counterexample, from which an infinite set is easily obtained. In the process, we derive~\eqref{K4ex}, which is used to properly prove \ref*{i.balance}.

	The mistaken idea in~\cite{Li} is that for any imbalanced pair of chains in a $K_4$-subdivision, the graph obtained by moving one edge from the longer to the shorter chain should have more spanning trees. 
	To the contrary, let $G$ be the $K_4$-subdivision shown in Figure~\ref{f.Li}, with chain lengths $(\ell_1, \ell_2, \ab\ell_3, \ab\ell_4, \ab\ell_5, \ell_6) = \ab(3,3, \ab 1, \ab 1, \ab 1, 5)$, 
	and let $G'$ be the graph obtained by moving an edge~$e$ from $\ell_1$ to~$\ell_3$. We will show that $d_3(G) - d_3(G') = -1$, which is to say that $G'$ has one less spanning tree than~$G$.
	
	\begin{figure}[hb!]
		\centering
		\includegraphics{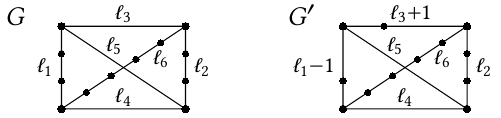}
		\renewcommand{\thefigure}{\arabic{figure}a}
		\caption{$G'$ is obtained from~$G$ by moving an edge from a longer to a shorter chain. $G'$ is in this sense ``more balanced'' than $G$, but has one less spanning tree nonetheless.}
		\label{f.Li}
	\end{figure}
	
	By Corollary~\ref{c.move} and since both graphs are bridgeless,
	\begin{equation}
		\label{K4d3bl}
		d_3(G) - d_3(G') = x_{3,2} - x'_{3,2} + x_{3,3} - x'_{3,3}\,.
	\end{equation}
	We explain step by step how to obtain~$x_{3,2}$ by its definition in Corollary~\ref{c.move}, counting the number of ways to construct an appropriate 3-sized edge set~$E$. The edge~$e$, which is moved from~$\ell_1$, belongs \emph{a~priori} to~$E$. WLOG the second edge should form a~2‑bond with~$e$ in~$G$ (Lemma~\ref{l.move}\ref{ll.moveG}), so it can be any one of the remaining $\ell_1 - 1$ edges in the same chain. The third edge~$e_3$ should not cause $E$ to contain any other bond in~$G$ (also Lemma~\ref{l.move}\ref{ll.moveG}) which precludes $e_3 \in \ell_1$. Furthermore, Lemma~\ref{l.move}\ref{ll.moveG'} forbids $e_3 \in \ell_3$ (which would create a 2‑bond in~$G'$) as well as $e_3 \in \ell_5$ (which would create a~3‑bond). The remaining choices for~$e_3$ are the edges in $\ell_2$,~$\ell_4$ and~$\ell_6$. We obtain $x_{3,2} = (\ell_1 -1)(\ell_2+\ell_4 + \ell_6)$.
	
	Similarly, $x'_{3,2} = \ell_3(\ell_2+\ell_4 + \ell_6)$, $x_{3,3} = \ell_6\ell_4$ and $x'_{3,3} = \ell_6\ell_2$. Insertion into~\eqref{K4d3bl} yields
	\begin{equation}
		\label{K4ex}
		d_3(G) - d_3(G') = (\ell_1 - 1 - \ell_3)(\ell_2+\ell_4 + \ell_6) + (\ell_4-\ell_2)\ell_6 \,. 
	\end{equation}
	With the given chain lengths, this sums to~$-1$. From \eqref{K4ex} it is immediate that one could instead start with any $K_4$-subdivision for which $\ell_1 - \ell_3 = 2$ and the other chains have the same lengths as in~$G$, and obtain the exact same result.
\end{exmp}

\begin{prop}
	\label{p.K4balanced}
	Let \(G\in \G_{n,n+2}\), where \(n\geq 1\). If~\(G\) is an imbalanced weak subdivision of~\(K_4\), then there is another weak \(K_4\)-subdivision \(H\in \G_{n,n+2}\) such that \(d_2(H) < d_2(G)\) and \(d_3(H) < d_3(G)\). Hence, an imbalanced weak \(K_4\)-subdivision can never be \(p\)‑optimal.
\end{prop}

\begin{proof}
	Let $G$ have chain lenghts $\ell_1$ through~$\ell_6$, arranged as in Figure~\ref{f.K4}. Suppose that $\ell_1$~and~$\ell_3$ have the largest difference in length for all pairs of adjacent chains, and that $\ell_1$ is chosen as large as possible with this condition. We define $\delta = \ell_1 - \ell_3$ and note that $\delta\geq 1$, since $G$ is assumed to be imbalanced.
	
	\begin{figure}[htb!]
		\centering
		\addtocounter{figure}{-1}
		\renewcommand{\thefigure}{\arabic{figure}b}
		\includegraphics{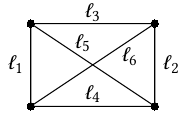}
		\caption{A general representation of a weak subdivision of~$K_4$, where the labels are weights representing chain lengths.}
		\label{f.K4}
	\end{figure}
	
	Let $G'$ be the graph obtained by moving one edge from~$\ell_1$ to~$\ell_3$, which gives another weak $K_4$-subdivision. Let $\Delta = d_3(G)-d_3(G')$, a formula for which was obtained in Example~\ref{ex.li} using Corollary~\ref{c.move}. Rearranging \eqref{K4ex} gives
	\begin{align}
		\label{K4}
		\Delta	&= (\delta - 1)(\ell_2+\ell_4 + \ell_6) - \ell_6(\ell_2-\ell_4) \nonumber\\
		&= (\delta-1)(\ell_2+\ell_4) + \ell_6((\delta-1)-(\ell_2-\ell_4))\,.
	\end{align}
	
	One would hope to prove that $\Delta$ is positive. However, in three particular cases, $\Delta$ will be zero, so that $G$~and~$G'$ have the same number of spanning trees. In these cases, we will repeat the same surgery on~$G'$ to obtain~$G''$ and a positive~$\Delta$.
	
	\emph{Case~1:} $\delta = 1$. In this case, we must have $\ell_1= a+2$, $\ell_2 = a$, and $\ell_3 = \ell_4 = \ell_5 = \ell_6 = a+1$, for some~$a\geq 0$. From~\eqref{K4}, we obtain
	\begin{equation*}
		\Delta = -\ell_6(\ell_2-\ell_4) = -(a+1)\left(a-(a+1)\right) = a+1 \geq 1 \,.
	\end{equation*}
	
	\emph{Case~2:} $\delta \geq 2$ and $\ell_2 - \ell_4 = \delta$. Insertion into~\eqref{K4} gives
	\begin{equation*}
		\begin{split}
			\Delta &= (\delta-1)(2\ell_4 + \delta) - \ell_6 \\
			&= (\delta-2)(\delta + 2\ell_4) + \ell_4 + (\delta - (\ell_6 - \ell_4)) \geq 0 \,,
		\end{split}
	\end{equation*}
	since all three terms are nonnegative. We consider the possibility of $\Delta$ being~$0$, which happens exactly when $\delta = 2$, $\ell_4 = 0$ and $\ell_6 - \ell_4 = \delta$, which in turn implies $\ell_6=2$. Since $\ell_2 - \ell_4 = \delta$ by assumption, we have $\ell_2 =2$. Furthermore, the assumption that $\ell_1-\ell_3 = \delta$ is as large as possible for adjacent chains forces $\ell_1$ to be~$2$ (since $\ell_1$ is also adjacent to the zero chain~$\ell_4$) and $\ell_3$ to be~$0$. Lastly, since $\ell_5$ is also adjacent to a zero chain, $\ell_5 \leq 2$. Thus, the chains of~$G$ are $(\ell_1, \ldots, \ell_6) = (2, 2, 0, \ab 0, \ab [0 \text{ or } \ab 1 \text{ or } 2], 2)$. For each of these three possibilities, the resulting $G'$ is imbalanced and distinct from any of the three possible graphs~$G$, and by starting from the beginning with~$G'$ instead of~$G$, we will obtain a graph~$G''$ with strictly lower $d_3$‑value. (In~fact, $G'$ will fall into Case~3.)
	
	\emph{Case~3:} $\delta \geq 2$ and $\ell_2 - \ell_4 \leq \delta-1$. We first assume that at least one of $\ell_2$~and~$\ell_4$ is positive. Then, insertion into~\eqref{K4} gives $\Delta \geq (\delta-1)(\ell_2 + \ell_4) > 0$. If on the other hand $\ell_2 = \ell_4 = 0$, insertion into~\eqref{K4} gives $\Delta = \ell_6(\delta-1)$. 
	Since $\ell_2$,~$\ell_4$ and~$\ell_6$ form a~cycle, and a \hyperlink{zerocycle}{cycle of zero chains} is not permitted, $\ell_6$ has to be positive, which implies $\Delta > 0$.
	
	Now, either let $H = G'$, or for any of the three exceptional cases, $H = G''$. By Theorem~\ref{t.d2}\ref{tt.d2decrease}, $H$ has a strictly lower $d_2$‑value than~$G$. With this, we have shown that $H$ is as required.
\end{proof}

\begin{prop}
	\label{p.onlyK4}
	Let \(G \in \G_{n,n+2}\), where \(n\geq 1\). If \(G\) is not a  weak subdivision of~\(K_4\), then there is a bridgeless \(H\in \G_{n,n+2}\) such that \(d_2(H) < d_2(G)\) and \(d_3(H) < d_3(G)\). Hence, any \(p\)‑optimal \((n,n+2)\)-graph is a weak \(K_4\)-distillation.
\end{prop}

\begin{proof}
	If $G$ has a bridge, the statement follows from Proposition~\ref{p.bridge}, so we can assume $G$ to be bridgeless. From Theorem~\ref{t.3ec}\ref{tt.3ec2} it follows that $G$ is equivalent to some weak subdivision~$G^\dag$ of~$K_4$ (the only distillation in $\D_2$), and from \ref{t.3ec}\ref{tt.3ec3} that $G^\dag$ is imbalanced with respect to $K_4$. Since $G$~and~$G^\dag$ have the same reliability function (Proposition~\ref{p.equiv}), the statement now follows from Proposition~\ref{p.K4balanced}.
\end{proof}


\begin{thm}
	\label{t.K4optimal}
	For each \(n\geq 1\) and \(m = n +2\) there is a uniquely optimal graph in \(\G_{n,m}\). This graph is a balanced weak subdivision of~\(K_4\) (``weak'' is not needed when \(n\geq 4\)), which is specified up to isomorphism by the following set of additional conditions, where
	\(m\equiv r \pmod{6}\) and \(r \in [-2 \twodots 3]\). See Figure~\ref{f.firstk2}.
	\begin{itemize}
		\item If \(r\in \{0,\pm1\}\), no further condition is needed.
		\item If \(r= 2\) (\(r= -2\)), the two longer (shorter) chains correspond to a matching in~\(K_4\).
		\item If \(r=3\), the three longer chains correspond to a simple 3‑path in~\(K_4\).
	\end{itemize}
\end{thm}

\begin{figure}[h!]
	\vspace{-\parsep}
	\centering%
	\includegraphics{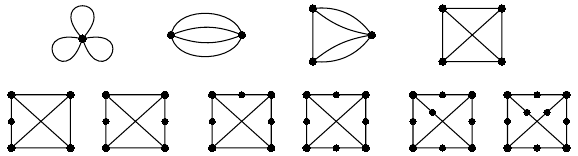}
	\caption{The first ten uniquely optimal $(n,n+2)$-graphs. The pattern continues and cycles every six graphs.}
	\label{f.firstk2}
\end{figure}

\begin{rmk}
	\label{r.k2generation}
	The following characterization due to~\cite{Li} generates the uniquely optimal $(n,n+2)$-graphs for $n\geq 4$: \emph{Cycle through the three perfect matchings of~$K_4$ and successively introduce a new vertex into each of the corresponding chains.} See the second row of Figure~\ref{f.firstk2}.
\end{rmk}

\begin{proof}
	Fix $n\geq 1$, so that $m \geq 3$, and recall that $\B_m(K_4)$ contains the balanced weak subdivisions of~$K_4$. If $r=0$, all chains have the same length, so $\B_m(K_4)$ contains a single perfectly balanced graph. There is also only one graph in the set when $r=\pm1$, since $K_4$ is edge-transitive.
	
	For the other cases, consider a graph $G\in\B_m(K_4)$ with a corresponding weighting of~$K_4$. By Proposition~\ref{p.d3min}, $G$~minimizes $d_3(\ph)$ in~$\B_m(K_4)$ if and only if $K_4$ has balanced $\piv$-values (except the \emph{only if} direction for $m=3$, but then there is only one graph). 
	
	Suppose that $r=2$ ($r=-2$). There are then two possible choices for~$G$: The two longer (shorter) chains can either be adjacent or nonadjacent, and this choice specifies $G$ up~to isomorphism. Only the latter choice, for which the edges of~$K_4$ with larger (smaller) weights form a perfect matching, gives balanced $\piv$-values. 
	
	Suppose that $r=3$. There are exactly three nonisomorphic graphs which can be obtained from different arrangements of the three longer chains, except when $m=3$, and the options are as follows. The longer chains can either 
	\begin{eninline}[label=\textit{\Alph*)}, ref=\textit{\Alph*}]
		\item emanate from one vertex, 
		\item form a~cycle, or
		\item\label{path} form a simple path. 
	\end{eninline}
	Only option~\ref*{path} gives $K_4$ a balanced $\piv$-multiset, namely $\mset{\piv_i}_{i=1}^4 = \mset{1,2,2,1}$. (Since the shorter chains do not form a~cycle, this construction is well-defined also for $m=3$.)
	
	We now use Proposition \ref{p.K4balanced} and~\ref{p.onlyK4} to conclude that the specified bridgeless graph uniquely minimizes $d_3(\ph)$ in~$\G_{n,m}$. It minimizes $d_2(\ph)$ by Theorem~\ref{t.d2}\ref{tt.d2iff}. Unique optimality follows.
\end{proof}

\section{The uniquely optimal \texorpdfstring{$(\MakeLowercase{k} = 3)$}{(k=3)}-graphs}
\label{s.k=3}

\subsection{Background}
\label{ss.k=3background}
If one considers and compares all possible distillations of exceedance~3, it can arguably seem quite intuitive that subdivisions of the complete bipartite graph~$K_{3,3}$ (Figure~\ref{f.K33}) should be most reliable.
As it happens, $K_{3,3}$ is similar to~$K_4$ in that its edge set can be partitioned into three perfect matchings. Boesch, Li and Suffel~\cite{Li} conjectured that UMRGs could be generated from~$K_{3,3}$ in analogy to the perfect matching method for~$K_4$ (see the remark after Theorem~\ref{t.K4optimal}) as follows.

\begin{falsethm}[\cite{Wang94}] 
\label{falsethm}
	Partition the nine edges of~\(K_{3,3}\) into three perfect matchings. A~sequence of UMRGs is obtained by cycling through these matchings, successively introducing a new vertex into each of the corresponding chains.
\end{falsethm}

However, the above is false, even though it has been considered a~theorem since~1994, when  a~proof of what was then called Boesch's Conjecture was claimed by Wang~\cite{Wang94}. (The result has been cited more than 50~times,
without any indication of a~mistake.) In addition to reaching the wrong conclusion, the same paper has an unrelated fatal flaw in its treatment of $\Pi_3$-subdivisions, as we demonstrate in Example~\ref{ex.WR}. A~corrected and exhaustive characterization of UMRGs with exceedance~3 is given in Section~\ref{ss.correction}. Our treatment is self-contained.

The two most common representations of~$K_{3,3}$ are shown in Figure~\ref{f.K33}. While representation~A most clearly displays the bipartite structure, representation~B indicates that $K_{3,3}$ belongs to the \emph{Möbius ladders}, defined below.

\begin{figure}[h!]
	\centering
	\newcommand{\wdot}{circle (0.4*\vxr-0.4*\et)}%
	\begin{tikzpicture}[x=0.8cm,y=0.8cm]
		\def\ey{1.5}
		\def\r{\ey/sqrt(3)}
		\draw (0,0) \co(a1)\vx -- (0,\ey) \co(b1)\vx -- (1,0) \co(a2)\vx -- ++(0,\ey) \co(b2)\vx -- ++(1,-\ey) \co(a3)\vx -- +(0,\ey) \co(b3)\vx (a1)--(b2) (a1)--(b3) (a2)--(b3) (a3)--(b1);
		\coordinate (C) at (5,\ey/2);
		\foreach \v in {0,60,...,300}{
			\draw ($(C)+({-\r*cos(\v)},{-\r*sin(\v)})$) \vx\co(v\v){} -- ($(C)+({-\r*cos(\v+60)},{-\r*sin(\v+60)})$);
		}
		\draw (v0) -- (v180) (v60) -- (v240) (v120) -- (v300);
		\filldraw[white] (v0) \wdot (v120) \wdot (v240) \wdot (b1) \wdot (b2) \wdot (b3) \wdot;
		\node[inner xsep=\inxsep, inner ysep=\inysep] (A) at ($(b1)+(-0.7,0.1)$) {\lrg{$\text{A}$}};
		\node[inner xsep=\inxsep, inner ysep=\inysep] at ({$(v0)-(0.4,0)$} |- A) {\lrg{$\text{B}$}};
	\end{tikzpicture}
	\caption{Two representations of the complete bipartite graph~$K_{3,3}$.}
	\label{f.K33}
\end{figure}

\begin{dfn}
	\label{d.mobius}
	For $k\geq 1$, the $k$-Möbius ladder, denoted by~$M_k$, is the graph obtained from~$C_{2k}$ by adding an edge between each pair of opposite vertices. The $2k$~edges in the cycle are called \emph{rails} and the $k$ additional edges are \emph{rungs} (but note that $M_k$ is edge transitive for $k\in\{1,2,3\}$).
\end{dfn}

We note that if Conjecture~\ref{cj.W}\ref{cjj.Wbasic} in Section~\ref{ss.k=4} holds, then every $p$‑optimal graph with $k\in \{1,2,3,4\}$ is a weak subdivision of the $k$‑Möbius ladder. However, the Petersen graph is uniquely optimal, so this pattern does not continue for $k=5$.

\subsection{Subdivisions of the triangular prism are not \texorpdfstring{$p$}{p}‑optimal}
\label{ss.Pi3}
Consider $\Pi_3$ with the edges labeled as in Figure~\ref{f.Pi3}. Regarding $\Pi_3$ as a ``circular ladder'', and in analogy with the Möbius ladders, we call the $c$‑edges \emph{rungs} and the $l$- and $r$‑edges \emph{rails}, and two rails with the same index are \emph{opposite} rails. Since $K_{3,3}$ is a Möbius ladder, it can be obtained from~$\Pi_3$ by introducing a ``half-twist'', or more precisely, by choosing a rung~$c_i$ in~$\Pi_3$ and then exchanging the incidences of two adjacent, opposite rails $l_i$ and~$r_i$, as shown in Figure~\ref{f.Pi3}. (Using $l_{i+1}$ and $r_{i+1}$ yields the same graph.)  We consider the corresponding surgery in an arbitrary weak $\Pi_3$-subdivision, and call it a \emph{reconnection} across~$c_i$.

\begin{figure}[t!]
	\centering
	\includegraphics{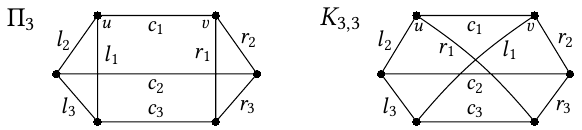}
	\caption{Reconnecting the rails $l_1$~and~$r_1$ across the rung~$c_1$ in the triangular prism~$\Pi_3$ yields~$K_{3,3}$. The edge labels also represent chain lengths of arbitrary weak subdivisions.}
	\label{f.Pi3}
\end{figure}

In this subsection, we will use the notation
\begin{equation}
	\label{ds1}
	\delta_i = l_i -r_i\quad \textrm{and}\quad \sigma_i = l_i + r_i\,,
\end{equation}
where $i \in \{1,2,3\}$, and $l_i$~and~$r_i$ refer to chain lengths of a weak subdivision of $\Pi_3$ or $K_{3,3}$ labeled according to Figure~\ref{f.Pi3}.

\begin{exmp}
	\label{ex.WR}
	We give an infinite family of ``strong'' counterexamples to the central claim of the proof of \cite[Theorem~9]{Wang94}. The proof idea was that given a $\Pi_3$-subdivision, then the $K_{3,3}$-subdivision obtained by reconnecting any two opposite rail chains, as in Figure~\ref{f.Pi3}, should have more spanning trees. Our counterexamples are strong in the sense that we provide a family of $\Pi_3$-subdivisions for which every possible rearrangement of chains into a $K_{3,3}$-subdivision has fewer spanning trees. We focus on a particular such example, and then indicate its generalization.
	
	Our particular example is the $\Pi_3$-subdivision and $(14,17)$-graph~$G$ in Figure~\ref{f.WR}, which has 50~more spanning trees than~$G'$ and 2~more than~$G''$. The latter two graphs are the only $K_{3,3}$-subdivisions with the same chain lengths as~$G$; to see this, consider that up to isomorphism there is only one way to choose two adjacent edges of~$K_{3,3}$ and only one way to choose two nonadjacent edges.
	(Equivalently, the line graph of $K_{3,3}$ is distance transitive with diameter~2.)
	
	\begin{figure}[hb!]
		\centering
		\includegraphics{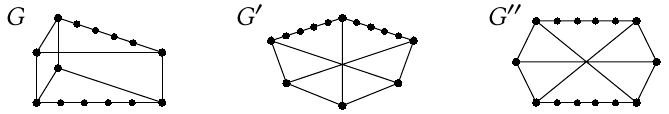}
		\caption{The $\Pi_3$-subdivision $G$ has more spanning threes than each of the two possible $K_{3,3}$-subdivisions with the same chain lengths.}
		\label{f.WR}
	\end{figure}
	
	We begin by writing
	\begin{equation}
		\label{d4bondsum}
		d_4(\ph) = d_{4,2}(\ph) + d_{4,3}(\ph) + d_{4,4}(\ph)\,, 
	\end{equation} 
	recalling that $d_{4,i}(\ph)$ is the number of 4‑disconnections in which the smallest bond has size~$i$. To see that $d_{4,2}(\ph)$ is the same for $G$, $G'$ and~$G''$, pair the chains of $G$ with those of~$G'$ (or~$G''$) by length and consider an induced pairing between the edges. 
	
	We now calculate $d_{4,3}(\ph)$ and $d_{4,4}(\ph)$ for the three graphs, starting with the latter. Every 4‑bond of~$\Pi_3$ and~$K_{3,3}$ is trivial (isolates one edge). Using Corollary~\ref{c.bonds}, we go over the chains of $G$,~$G'$ and~$G''$, multiply the lengths of their adjacent chains and sum the results to obtain
	\begin{alignat*}{2}
		d_{4,4}(G) &= &5^2 1^2 + 6\cdot5^1 1^3 + 2\cdot1^4 = 57\\
		d_{4,4}(G') &= &5^2 1^2 + 6\cdot5^1 1^3 + 2\cdot1^4 = 57\\
		d_{4,4}(G'') &= 2 \cdot{}&5^2 1^2 + 4\cdot5^1 1^3 + 3\cdot1^4 = 73
	\end{alignat*}
	
	Regarding $d_{4,3}(\ph)$, we note that $G'$ and~$G''$ have only trivial 3‑bonds, while $G$ in~addition has a nontrivial 3‑bond, separating the two 3‑cycles. Any $(4,3)$-disconnection is uniquely determined by the following two steps: 
	\begin{eninline}[label=(\arabic*)]
		\item Choose a~3‑bond in the proper distillation and pick one edge from each of the corresponding chains.
		\item Choose a fourth edge from any one of the other six chains; this does not create any additional bond for these graphs (as can be deduced from the 3‑edge-connectedness of their distillations).
	\end{eninline}
	This gives
	\begin{alignat*}{2}
		d_{4,3}(G) &=   & 4 \cdot 5^1 1^2 10 + 3\cdot 1^3 14 = 242\\
		d_{4,3}(G') &= 5^2 1^1 6 +{} &2\cdot5^1 1^2 10 + 3\cdot1^3 14 = 292\\
		d_{4,3}(G'') &=   & 4\cdot 5^1 1^2 10 + 2\cdot1^3 14 = 228
	\end{alignat*}
	
	Using \eqref{d4bondsum} and that $d_{4,2}(\ph)$ is unchanged gives $d_4(G')-d_4(G)=(57-57)+(292-242) = 50$, and similarly $d_4(G'')-d_4(G)=2$. It is straightforward to generalize the above equations and show that replacing the chains of length~5 by any two which are at least as long yields a similar counterexample.
\end{exmp}

\begin{lem}
	\label{l.Pibnf}
	For \(n\geq2\), let \(G\in \G_{n,n+3}\) be a weak \(\Pi_3\)-subdivision for which all rung chains~\(c_i\) are positive. Referring to~\eqref{ds1}, suppose that the following hold.
	\begin{enumenun}[label=(\arabic*)]
		\item\label{ll.bigdeltas} \( \delta_1 \geq 1\) and \(-\delta_3\geq 1\), with at least one strict inequality.
		\item\label{ll.span} \(\delta_2\in [0 \twodots -\delta_3-1]\).
	\end{enumenun}
	Then there is another weak \(\Pi_3\)-subdivision \(G''\in \G_{n,n+3}\), obtained by moving two edges between rail chains, such that \(d_i(G'') < d_i(G) \) for \(i\in \{2,3,4\}\).
\end{lem}

\begin{rmk}
	With a strategic chain labeling (e.g.\ as given in the proof of Proposition~\ref{p.noPi3}) the above lemma applies to many more $\Pi_3$-subdivisions than is immediately apparent.
\end{rmk}

\begin{proof}
	We obtain $G''$ from~$G$ by first moving one edge from the chain~$l_1$ to~$r_1$, calling the result~$G'$, and then moving one edge from~$r_3$ to~$l_3$. (See Figure~\ref{f.piEdgeMove}. Note that the labeling agrees with  $\Pi_3$ in Figure~\ref{f.Pi3}.) These surgeries are possible since
	$l_1$ and~$r_3$ are positive by assumption~\ref*{ll.bigdeltas}.
	By Theorem~\ref{t.d2}\ref{tt.d2decrease} we have $d_2(G'')<d_2(G)$, since we have twice moved an edge from a longer to a shorter chain, and at least one of the two chain pairs was initially imbalanced.
	
	\begin{figure}[hb!]
		\centering
		\includegraphics{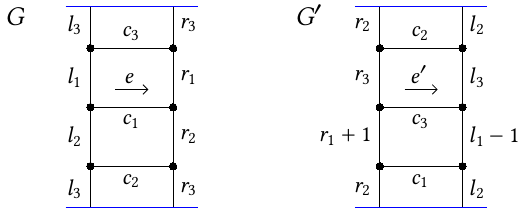}
		\caption{An edge $e$ of the weak $\Pi_3$-subdivision $G$ is moved, yielding~$G'$, in which in turn another edge~$e'$ is moved to obtain~$G''$ (not shown).}
		\label{f.piEdgeMove}
	\end{figure}
	
	We use the shorthand~$d_i$ for~$d_i(G)$, with primes added for~$G'$ and~$G''$. Let $C= c_1 + c_2 + c_3$. By~Corollary~\ref{c.move} (for a detailed example, see the derivation of \eqref{K4ex} in Example~\ref{ex.li}) we obtain
	\begin{align}
		\label{d3Pi3I}
		d_3 -{} d_3' &= x_{3,2} - x_{3,2}' + x_{3,3} - x_{3,3}' \nonumber\\
		&= (l_1-1)(C+\sigma_2 +\sigma_3) - r_1(C+\sigma_2 +\sigma_3) + (c_1l_2 + c_3l_3) - (c_1r_2 + c_3r_3) \nonumber \\
		&= (\delta_1 -1)(c_1 + c_2 + c_3 + \sigma_2 + \sigma_3) + c_1\delta_2 +c_3\delta_3 \,.
	\end{align}
	We can make substitutions in~\eqref{d3Pi3I}, with a little help from Figure~\ref{f.piEdgeMove}, to furthermore obtain
	\begin{equation*}
		d_3' - d_3'' = (-\delta_3 -1)(c_1 + c_2 + c_3 + \sigma_1 + \sigma_2) + c_3(-\delta_1+2) +c_2(-\delta_2) \,.
	\end{equation*}
	Adding the two equations above yields, with some rearranging and canceling,
	\begin{equation}
		\label{d3Pi3II}
		d_3 - d_3'' = (\delta_1 -1)(c_1 + c_2 + \sigma_2 + \sigma_3) + (-\delta_3-1)(c_1 + \sigma_1 + \sigma_2) + c_1\delta_2 + c_2(-\delta_3-1 - \delta_2)\,.
	\end{equation}
	By our assumptions, it is immediate that each of the four terms above is nonnegative, and that the first or second is positive. Hence, $d_3(G'') < d_3(G)$.
	
	Finally, we study how the number of 4‑disconnections changes. Corollary~\ref{c.move} gives
	\begin{equation}
		d_4 - d_4' = (x_{4,2} - x_{4,2}') + (x_{4,3} - x_{4,3}') + (x_{4,4} - x_{4,4}') \,.
	\end{equation}
	We calculate the three terms above. The first two can be compared to~\eqref{d3Pi3I}. Letting $C_\times = c_1c_2 + c_1c_3 + c_2c_3$, one can verify the following. (Except for in $(l_1-1)$, all minus signs come from the subtracted primed $x$‑terms, and there is no cancellation involved.)
	\begin{align}
		\label{d4move}
		x_{4,2} - x_{4,2}' &= \big((l_1-1)-r_1\big)\big(C_\times + c_1\sigma_3 + c_3\sigma_2 + c_2(\sigma_2 + \sigma_3) + \sigma_2\sigma_3\big) \nonumber \\
		x_{4,3} - x_{4,3}' &= c_1(l_2 - r_2)(c_2 + c_3 + \sigma_3) + c_3(l_3-r_3)(c_1 + c_2 + \sigma_2)\\
		x_{4,4} - x_{4,4}' &= c_1c_2(l_3 - r_3) + c_3c_2(l_2-r_2) \nonumber
	\end{align}
	Adding these three equations yields the following. (The five colors will soon be explained.)
	\begin{equation}
		\label{d4Pi3}
		\begin{multlined}[.9\displaywidth]
			d_4-d_4' = (\delta_1-1)(\textcolor{Red3}{C_\times} + c_1\sigma_3 + \textcolor{Green4}{c_3\sigma_2} + c_2(\sigma_2+\sigma_3) + \sigma_2\sigma_3) \\
			+ c_1\delta_2(\textcolor{Gold4}{c_2} + \textcolor{Orchid4}{c_3} + \sigma_3) + c_3\delta_3(\textcolor{Red3}{c_1+c_2}+\textcolor{Green4}{\sigma_2}) + \textcolor{Red3}{c_1c_2\delta_3} + \textcolor{Blue3}{c_3c_2\delta_2}\,.
		\end{multlined}
	\end{equation}
	We then obtain $d_4'-d_4''$ by making substitutions in~\eqref{d4Pi3} according to Figure~\ref{f.piEdgeMove}:
	\begin{equation*}
		\begin{multlined}[.9\displaywidth]
			d_4'-d_4'' = (-\delta_3-1)(\textcolor{Red3}{C_\times} + \textcolor{Green4}{c_3\sigma_2} + c_2\sigma_1 + c_1(\sigma_1+\sigma_2) + \sigma_1\sigma_2) \\
			+ c_3(-\delta_1+2)(\textcolor{Red3}{c_1 + c_2} + \textcolor{Green4}{\sigma_2}) + c_2(-\delta_2)(\textcolor{Blue3}{c_3}+\textcolor{Gold4}{c_1}+\sigma_1) + \textcolor{Orchid4}{c_3c_1(-\delta_2)} + \textcolor{Red3}{c_2c_1(-\delta_1+2)}\,.
		\end{multlined}
	\end{equation*}
	
	Adding the last two equations, all the same-colored terms cancel. With a slight rearrangement we obtain
	\begin{align}
		\label{d4Pi3II}
		\begin{split}
			d_4-d_4'' ={}& (\delta_1-1)( c_1\sigma_3 + c_2(\sigma_2+\sigma_3) + \sigma_2\sigma_3) + c_1\delta_2\sigma_3 \\
			&+ (-\delta_3-1)\big(c_1(\sigma_1+\sigma_2) + \sigma_1\sigma_2\big) + c_2\sigma_1(-\delta_3-1-\delta_2)\,.
		\end{split}
	\end{align}
	
	Just as for~\eqref{d3Pi3II}, it is immediate that each of the four terms of~\eqref{d4Pi3II} are nonnegative. The only additional observation needed to see that the first or third is positive is that it follows from $\delta_1\geq 1$ and $-\delta_3\geq1$ that $\sigma_1\geq 1$ and $\sigma_3\geq 1$.
	Hence, $d_4(G'') < d_4(G)$.
\end{proof}

\begin{prop}
	\label{p.noPi3}
	For  \(n\geq 2\), let \(G\in \G_{n,n+3}\) be a weak \(\Pi_3\)-subdivision which is not also a weak \(K_{3,3}\)-subdivision (equivalently, which has three positive rung chains~\(c_i\)). Then there is a weak \(K_{3,3}\)-subdivision \(G'\in \G_{n,n+3}\) such that \(d_2(G')\leq d_2(G)\), \(d_3(G')<d_3(G)\) and \(d_4(G') < d_4(G)\), which is therefore strictly more reliable than~\(G\). (In particular, a \(\Pi_3\)-subdivision cannot be \(p\)‑optimal for any~\(p\).)
\end{prop}

\begin{proof}
	Since the graphs of Figure~\ref{f.Pi3} become isomorphic by contracting a $c_i$‑edge in both graphs, any weak $\Pi_3$-subdivision with a zero-length rung chain is also a weak $K_{3,3}$-subdivision. Hence, $G$ has only positive rung chains $c_i$, and no such~$G$ can be a weak $K_{3,3}$-subdivision, since the positive rung chains creates nontrivial 3‑bonds. This proves ``equivalently''.
	
	Now, we label the chains of~$G$ as in Figure~\ref{f.Pi3} and according to the following restrictions. \begin{eninline}[label={\textbf{R\arabic*:}},ref={R\arabic*}]
		\item\label{R1} At least two of the values $\delta_1$, $\delta_2$ and~$\delta_3$, see~\eqref{ds1}, should be nonnegative.
		\item\label{R2} If one or two of the $\delta_i$‑values are zero, then the nonnegative numbers should include the largest of the three absolute values. \item\label{R3} $\delta_1\geq \delta_2 \geq \delta_3$.
	\end{eninline}
	(\ref*{R1}~and~\ref*{R2} can be satisfied by appropriately choosing the ``left'' and ``right'' sides of~$G$, while~\ref*{R3} is satisfied by an appropriate indexing.) Note that $\delta_1\geq \delta_2\geq 0$ by \ref*{R1}~and~\ref*{R3}. Consequently, if $\delta_3$ is negative, then \ref*{R2} forces $\delta_1$ to be positive. We consider three exhaustive cases.
	
	\emph{Case 1:} $\delta_3 < -\delta_2$ and $\max(\delta_1, -\delta_3) >1$. We first show that the conditions \ref{ll.bigdeltas} and~\ref{ll.span} of Lemma~\ref{l.Pibnf} are fulfilled. Since $\delta_2$ is always nonnegative, $\delta_3 < -\delta_2$ implies both $\delta_2\in [0 \twodots -\delta_3-1]$, which is~\ref{ll.span}, and that $\delta_3$ is negative. The latter, as noted above, implies that $\delta_1$ is positive. The strict inequality additionally required by~\ref{ll.bigdeltas} is given by $\max(\delta_1, -\delta_3) >1$.
	
	Applying Lemma~\ref{l.Pibnf}, we have a weak $\Pi_3$-subdivision $H\in \G_{n,n+3}$ with unchanged rung chains and strictly smaller $d_i$‑values than $G$. By relabeling the chains of $H$ according to \ref*{R1}, \ref*{R2} and \ref*{R3} and repeatedly applying Lemma~\ref{l.Pibnf} if necessary, we eventually obtain a graph which falls into Case~2 or Case~3.
	
	\emph{Case 2:}  $\delta_3 < -\delta_2$ and $\max(\delta_1, -\delta_3) \leq 1$. Just as in Case~1 it follows that $\delta_1$ is positive, $\delta_2$ is nonnegative and $\delta_3$ is negative. But then $\max(\delta_1, -\delta_3) \leq 1$ implies that $\delta_1 = -\delta_3 =1$, and since $-\delta_3 > \delta_2$, it follows that $\delta_2=0$. Thus $(\delta_1,\delta_2,\delta_3)=(1,0,-1)$. We finish Case~2 after Case~3, which is the main work.
	
	\emph{Case 3:} $\delta_3\geq -\delta_2$.
	Construct $G'$ from~$G$ by reconnecting the chains $l_1$ and~$r_1$ across~$c_1$, as shown in Figure~\ref{f.Pi3}. (Interpreting the labels of Figure~\ref{f.Pi3} as weights, this surgery is well-defined even if $l_1$, $r_1$ or both have zero length. Since the rung chains~$c_i$ are assumed positive, there are no \hyperlink{zerocycle}{``zero cycles''}, so $G$ and~$G'$ are well-defined weak subdivisions of $\Pi_3$ and $K_{3,3}$, respectively.)
	
	Since the chain lengths are unchanged, we have that $d_2(G') = d_2(G)$, and furthermore that the number of 3- and 4‑disconnections which contain a~2‑bond is the same. In particular, $d_3(G) - d_3(G')=b_3(G')-b_3(G)$. The 3‑bonds of the weak distillations $\Pi_3$ and~$K_{3,3}$ are exactly the same, except for the trivial bonds separating an endvertex of~$c_1$ (see Figure~\ref{f.Pi3}) and for the nontrivial bond $\{c_1,c_2,c_3\}$ of~$\Pi_3$. Hence, using Corollary~\ref{c.bonds},
	\begin{equation}
		\label{d3rec}
		b_3(G) - b_3(G') = c_1c_2c_3 + c_1l_1l_2 + c_1r_1r_2 - c_1r_1l_2 - c_1l_1r_2
		= c_1(c_2c_3 + \delta_1\delta_2)\,.
	\end{equation}
	This expression is positive since the $c_i$'s are positive and since $\delta_1$ and~$\delta_2$ are nonnegative. Hence, $d_3(G')<d_3(G)$.
	
	To obtain $d_4(G) - d_4(G')$, we use the same notation as in \eqref{d4bondsum} and the fact that $d_{4,2}(\ph)$ is unchanged, writing
	\begin{equation}
		\label{d4i}
		d_4(G) - d_4(G') = d_{4,3}(G) - d_{4,3}(G') + d_{4,4}(G) - d_{4,4}(G')\,.
	\end{equation}
	
	We start by considering $d_{4,3}(G) -d_{4,3}(G')$. The 3‑bonds of $G$~and~$G'$ which are not involved in~\eqref{d3rec} are in a natural one-one correspondence, which induces a one-one correspondence between the $(4,3)$-disconnections which contain any of these bonds. We deduce that any other $(3,4)$-disconnection is uniquely determined by choosing some 3‑bond counted in~\eqref{d3rec} and adding an edge from any of the other six chains. Thus, recalling the notation of~\eqref{ds1},
	\begin{align}
		\label{d43}
		\begin{split}
			d_{4,3} -d_{4,3}' ={}& c_1\big(c_2c_3(\sigma_1 + \sigma_2 + \sigma_3) + l_1l_2(\textcolor{Red3}{r_1 + r_2} + c_2 + c_3 + \sigma_3) + r_1r_2(\textcolor{Red3}{l_1 +l_2} + c_2 \\
			&+ c_3 + \sigma_3) - r_1l_2(\textcolor{Red3}{l_1 + r_2} + c_2 + c_3 +\sigma_3) - l_1r_2(\textcolor{Red3}{l_2 + r_1} + c_2 + c_3 + \sigma_3)\big)
		\end{split}\nonumber\\
		={}& c_1\big(c_2c_3(\sigma_1 + \sigma_2 + \sigma_3) + (l_1-r_1)(l_2-r_2)(c_2 + c_3 + \sigma_3)\big)\,,
	\end{align}
	where the terms in red cancel (because every edge set which contains one edge each from four out of the five chains in $\{c_1, l_1, l_2,\ab r_1, r_2\}$ disconnects both $G$ and~$G'$).
	
	To obtain $d_{4,4}(G) -d_{4,4}(G')$, we use that every 4‑bond of $\Pi_3$ and~$K_{3,3}$, and hence of $G$~and~$G'$, is trivial. Furthermore, every labeled edge in the two distillations has the same adjacencies, except for the four edges $l_1$, $l_2$, $r_1$ and~$r_2$; see Figure~\ref{f.Pi3}. Thus, using Corollary~\ref{c.bonds}, we multiply the four adjacent weights for each of these four edges, summing the results for~$\Pi_3$ and subtracting the results for~$K_{3,3}$. This yields
	\begin{align}
		\label{d44}
		d_{4,4} -d_{4,4}' ={}& c_1c_3l_2l_3 + c_1c_2l_1l_3 + c_1c_3r_2r_3 + c_1c_2r_1r_3
		- c_1c_3l_3r_2 - c_1c_2l_3r_1 - c_1c_3l_2r_3 - c_1c_2l_1r_3 \nonumber \\
		={}& c_1c_2(l_1-r_1)(l_3-r_3) + c_1c_3(l_2-r_2)(l_3-r_3)
	\end{align}
	
	By~\eqref{d4i}, we add \eqref{d43}~and~\eqref{d44} to obtain
	\begin{align}
		\label{d4rec}
		d_4-d_4' &= c_1\big(c_2c_3(\sigma_1 + \sigma_2 + \sigma_3) + \delta_1\delta_2(c_2 + c_3 + \sigma_3) + c_2\delta_1\delta_3 + c_3\delta_2\delta_3 \big) \nonumber\\
		&= c_1\big(c_2c_3(\sigma_1 + \sigma_2 + \sigma_3) + \delta_1\delta_2\sigma_3 + c_2\delta_1(\delta_2 + \delta_3) + c_3\delta_2(\delta_1 + \delta_3) \big) \,. 
	\end{align}
	We know that $c_1$ is positive, and that the four terms of the second factor are at least nonnegative: $\delta_2 + \delta_3 \geq 0$ since $\delta_3 \geq -\delta_2$ by assumption, and then it follows from~\ref{R3} that $\delta_1 + \delta_3 \geq 0$. Furthermore, since $c_2$ and~$c_3$ are positive, and at least some $l_i$ or~$r_i$ is positive, the first term is positive. Hence, $d_4(G')<d_4(G)$.
	
	\emph{Case 2, finished:} Let $G'$ be obtained as in Case~3.
	Inserting $(\delta_1,\delta_2,\delta_3)=(1,0,-1)$ into \eqref{d3rec} implies $d_3(G)-d_3(G')= c_1c_2c_3 > 0$, and into~\eqref{d4rec} yields $d_4(G) - d_4(G') = c_1c_2\big(c_3(\sigma_1 + \sigma_2 + \sigma_3)-1\big)$. Since $\sigma_1 + \sigma_2 + \sigma_3\geq 4$, we conclude that $d_4(G) - d_4(G') > 0$, which finishes the~proof.
\end{proof}

\subsection{Zeroing in on balanced weak \texorpdfstring{$K_{3,3}$}{K33}-subdivisions}
\label{ss.K33}
In Figure~\ref{f.K33labels} we introduce a new edge labeling for~$K_{3,3}$ (as usual the labels will also represent chain lengths), in which three edges share a vertex if and only if they all share the same letter or index. Furthermore, the 4‑cycles are the edge sets that combine two out of three letters and two out of three indices; for example, $a_3a_1c_1c_3$ specifies a~4‑cycle in~$K_{3,3}$.

\begin{figure}[hb!]
	\centering
	\includegraphics{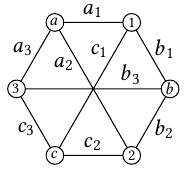}
	%
	%
	\caption{The complete bipartite graph $K_{3,3}$ with an edge labeling used extensively in Section~\ref{ss.K33}. The edge labels can represent chain lengths of weak $K_{3,3}$-subdivisions.}
	\label{f.K33labels}
\end{figure}

With the labels of Figure~\ref{f.K33labels} representing chain lengths, we define
\begin{equation}
	\label{ds2}
	\delta_i = a_i -b_i\quad \textrm{and}\quad \sigma_i = a_i + b_i + c_i\,, \quad\mathrlap{\quad i\in\{1,2,3\}\,.}
\end{equation}

We will need formulas for how $d_3(\ph)$ and~$d_4(\ph)$ change when an edge is moved between adjacent chains. Suppose that $a_1$ is positive in~$G$, and let $G'$ be obtained by moving one edge from~$a_1$ to~$b_1$. By applying Corollary~\ref{c.move} we obtain
\begin{equation}
	\label{d3K33}
	d_3(G)-d_3(G') = (\delta_1 - 1)(\sigma_2 + \sigma_3) + a_2a_3 - b_2b_3 \,,
\end{equation}
where the first term originates from $x_{3,2}-x_{3,2}'$, the second term is~$x_{3,3}$ and the third is~$x_{3,3}'$ (cf.~\eqref{d3Pi3I}). The change in~$d_4(\ph)$ will be derived when proving Proposition~\ref{p.K33balanced}.

\begin{lem}
	\label{l.d3KBB}
	Let \(G\) be a weak \(K_{3,3}\)-subdivision in which \(a_1\) and~\(b_1\) are adjacent chains with the largest possible length difference among all adjacent pairs (see Figure~\ref{f.K33labels}). Move an edge from~\(a_1\) to~\(b_1\) to obtain~\(G'\). If \(a_1-b_1\geq 2\), then \(d_3(G') < d_3(G)\).
\end{lem}

\begin{proof}
	With $G$ as in the lemma, suppose that $a_1-b_1\geq 2$. Let the other chains be labeled as in Figure~\ref{f.K33labels} and~so that $a_2-b_2\geq a_3 - b_3$. (Take a labeling where the latter does not hold. Swap the places of vertices $2$ and~$3$ and then relabel to agree with the figure.)
	Rearrange~\eqref{d3K33}, with $d_3(G)-d_3(G') = \Delta$, to obtain the following. Our assumptions $\delta_1\geq 2$, $\delta_1 \geq \max(\abs{\delta_2},\abs{\delta_3})$ and $\delta_2\geq \delta_3$ are used for the inequalities.
	\begin{align}
		\Delta &= (\delta_1 - 1)(\sigma_2 + \sigma_3) + (a_2-b_2)a_3 + (a_3 - b_3)b_2 \nonumber\\
		&=(\delta_1 -1)(c_2 + c_3) + (\delta_1 -2)(a_2+b_3) + (\delta_1 + \delta_2)a_3 + (\delta_1 + \delta_3)b_2 + (\delta_2-\delta_3) \nonumber\\
		\label{d3K33_2}
		&\geq (c_2 + c_3) + (\delta_1 + \delta_2)a_3 + (\delta_1 + \delta_3)b_2 + (\delta_2 - \delta_3) \geq 0
	\end{align}
	
	It remains to show that $\Delta = 0$ is impossible. Suppose for contradiction that $\Delta = 0$, so that all four terms of~\eqref{d3K33_2} are zero. In particular, $c_2=c_3=0$ and $\delta_2 = \delta_3$, which we denote by~$\delta_{2,3}$.
	
	If on the one hand, $\delta_1 + \delta_{2,3} > 0$, then, considering the second and third terms of~\eqref{d3K33_2}, $a_3=b_2=0$. Then, $\delta_2 = \delta_3$ simplifies to $a_2 = -b_3$, so $a_2$ and~$b_3$ are also zero chains. However, this would give us several \hyperlink{zerocycle}{``zero cycles''}, such as $c_2c_3a_3a_2$, which is impossible.
	
	If on the other hand, $\delta_1 + \delta_{2,3} = 0$, then, in particular, $\delta_1 = -\delta_2 = b_2-a_2$,  so $b_2 = \delta_1 + a_2$. Since $b_2$~is adjacent to the zero chain~$c_2$, we also have $b_2 = b_2-c_2 \leq \delta_1$. Put together, this gives $a_2\leq 0$, so $a_2$ is a zero chain. Analogously, we see that $a_3$ is a zero chain, which leads to the ``zero cycle'' $c_2c_3a_3a_2$. We conclude that $\Delta > 0$.
\end{proof}

\enlargethispage{-2mm}
\begin{prop}
	\label{p.K33balanced}
	If \(G\) is an imbalanced weak \(K_{3,3}\)-subdivision, then there is a graph~\(G'\), obtained by moving an edge between two adjacent chains of~\(G\), such that \(d_2(G')\leq d_2(G)\), \(d_3(G') < d_3(G)\) and \(d_4(G')<d_4(G)\).
\end{prop}

\begin{proof}
	Since there is only one way to choose two adjacent edges of~$K_{3,3}$ up to isomorphism (as noted in Example~\ref{ex.WR}), we can assume that we want to move an edge from the chain~$a_1$ to~$b_1$ in Figure~\ref{f.K33labels}, but without yet specifying how these chains are to be chosen. We first derive an expression for $d_4(G)-d_4(G')$, which we obtain through Corollary~\ref{c.move}, starting with the $x_{i,j}$‑terms. (Like for~\eqref{d4move}, expanding the expressions while keeping $(a_1-1)$ and then grouping terms by sign recovers the $x$‑terms.)
	\begin{align}
		\label{K33xterms}
		x_{4,2} - x_{4,2}' &= \big((a_1-1)-b_1\big)\big((a_2 + a_3)(b_2 + b_3) + (c_2 + c_3)(a_2 + a_3 + b_2 + b_3)  \big) \nonumber \\
		x_{4,3} - x_{4,3}' &= a_2a_3(b_2 + b_3 + c_2 + c_3) - b_2b_3(a_2 + a_3 + c_2 + c_3) \\
		x_{4,4} - x_{4,4}' &= a_2b_3c_3 + a_3b_2c_2 - a_2b_3c_2 - a_3b_2c_3 \nonumber
	\end{align}
	Adding the above equations, with $\Delta = d_4(G) - d_4(G')$, and rearranging the right-hand side yields the following. (An equation corresponding to~\eqref{d4K33_1} was claimed to be nonnegative in~\cite{Wang94}, without any justification and under different conditions on the variables than we will use.)
	\begin{align}
		\begin{split}
			\label{d4K33_1}
			\Delta ={}& \big(\delta_1 - 1\big)\big((a_2 + a_3)(b_2 + b_3) + (c_2 + c_3)(a_2 + a_3 + b_2 + b_3)  \big)\\
			& + \delta_2\big(a_3b_3 + c_3(a_3 + b_3)\big) + \delta_3\big(a_2b_2 + c_2(a_2 + b_2)\big)
		\end{split}\\
		\begin{split}
			\label{d4K33}
			={}& \big(\delta_1 - 1\big)\big(a_2b_3 + a_3b_2 + c_2(a_3 + b_3) + c_3(a_2 + b_2)\big) \\
			&+ \big(\delta_1 + \delta_2 - 1\big)\big(a_3b_3 + c_3(a_3 + b_3)\big) + \big(\delta_1 + \delta_3 - 1 \big) \big( a_2b_2 + c_2(a_2 + b_2)\big) \,.
		\end{split}
	\end{align}
	
	We now fix a chain labeling of~$G$ according to Figure~\ref{f.K33labels}. The labeling should be such that
	\begin{eninline}[label={\textbf{C\arabic*:}},ref={C\arabic*}]
		\item\label{C1} $\delta_1 = a_1-b_1$ is as large as possible,
		\item\label{C2} $a_1$~is as~large as possible under the preceding condition,
		\item\label{C3} $\delta_2$~is as large as possible under the preceding two conditions, and
		\item\label{C4} $b_2$ is as small as possible under the preceding three conditions.
	\end{eninline}
	
	$G'$ is obtained by moving an edge from~$a_1$ to~$b_1$. We need to show that $d_2(G') \leq d_2(G)$, that $d_3(G') < d_3(G)$ and that $\Delta$ above is positive. We will consider four cases. That they are 
	exhaustive will be clear from the following three observations:
	\begin{eninline}[label=(\arabic*)]
		\item $\delta_1\geq 1$,
		\item $\delta_1 + \delta_2 \geq 0$ and $\delta_1 + \delta_3 \geq 0$, by~\ref{C1}, and
		\item $\delta_2\geq \delta_3$, because of~\ref{C3}. 
	\end{eninline}
	Regarding Case~3 and Case~4 below, it follows from Theorem~\ref{t.d2}\ref{tt.d2decrease} that $d_2(G') < d_2(G)$ and from Lemma~\ref{l.d3KBB} that $d_3(G') < d_3(G)$, so for these cases we only need to show that $\Delta > 0$.
	
	\emph{Case 1:} $\delta_1 = 1$. This implies that all pairs of adjacent chains of~$G$ are balanced. We note that $G$ and~$G'$ have the same chain lengths (with the lengths of $a_1$~and~$b_1$ switched) and that $d_2(G) = d_2(G')$.
	Since $G$~itself is imbalanced and no two chains are separated by more than one chain, we deduce that the chain lengths of~$G$ are $\ell$, $\ell +1$ and $\ell +2$, for some $\ell\geq 0$, and that $G$ has a cycle of chains $\ell_1\ell_2\ell_3\ell_4$ such that $(\ell_1,\ell_2,\ell_3,\ell_4) = (\ell +2, \ell+1, \ell, \ell+1)$. The conditions~\ref{C1} through~\ref{C4} then forces $a_1b_1b_2a_2$ to be such a~cycle.
	Inserting these values into~\eqref{d3K33} yields the following equation. For the inequality, we use that by~\ref{C1}, $a_3\geq \ell+1$ (since $a_3$~is adjacent to~$a_1$) and $b_3\leq \ell+1$ (since $b_3$~is adjacent to~$b_2$).
	\begin{equation*}
		d_3(G)-d_3(G') = (\ell +1)a_3 - \ell b_3 \geq (\ell+1)^2 - \ell(\ell+1) = \ell + 1 > 0 \,,
	\end{equation*}
	Inserting the same values into~\eqref{d4K33} yields
	\begin{equation*}
		\begin{split}
			\Delta &= \big(a_3b_3 + c_3(a_3 + b_3)\big) + \delta_3\big((\ell+1)\ell + c_2(2\ell +1)\big)\\
			&\geq \big((\ell+1)b_3 + c_3(\ell + 1 + b_3)\big) + \big((\ell+1) - b_3\big)\big((\ell+1)\ell + c_2(2\ell +1)\big)\,,
		\end{split}
	\end{equation*}
	where the possible values for~$b_3$ are $\ell + 1$ and~$\ell$. If on the one hand $b_3 = \ell +1$, then $\Delta \geq (\ell+1)^2 > 0$. If on the other hand $b_3 = \ell$, then $\Delta \geq c_2 + c_3$. Now, $c_2$ and~$c_3$ cannot both be zero chains, since a zero chain in~$G$ would presuppose $\ell=0$ and we would then have the  \hyperlink{zerocycle}{``zero cycle''} $c_2c_3b_3b_2$. We conclude that $\Delta > 0$.
	
	\emph{Case 2:} $\delta_1 \geq 2$ while $\delta_1 + \delta_2 = 0 $ and $\delta_1 + \delta_3 = 0$. We will show that this case is inconsistent with~\ref{C3}. Since $\delta_1 =  -\delta_2 =  b_2 - a_2$, the chain~$b_2$ has length $a_2 + \delta_1$, so the cycle of chains $(a_1,a_2, b_2,b_1)$ has lengths $(b_1 +\delta_1, a_2, a_2 + \delta_1, b_1)$. Furthermore, since $\delta_1$ is the largest difference between adjacent chains, we have 
	\begin{equation*}
		\left\{
		\begin{array}{@{}l@{}}
			(b_1+\delta_1) - a_2 \leq \delta_1\\
			(a_2+\delta_1) - b_1 \leq \delta_1
		\end{array}\right.
		\implies
		\left\{\begin{array}{@{}l@{}}
			b_1 - a_2 \leq 0\\
			a_2 - b_1 \leq 0
		\end{array}\right.
		\implies\:\, a_2 = b_1 \,.
	\end{equation*}
	Hence, with $\ell$ denoting the length of~$b_1$, the aforementioned cycle $(a_1,a_2,\ab b_2,b_1)$ has lengths $(\ell +\delta_1,\ab \ell,\ab \ell + \delta_1, \ell)$. It is analogously shown that $(a_1,a_3,\ab b_3,b_1) = (\ell +\delta_1, \ell,\ab \ell + \delta_1, \ell)$. This gives the cycle $(b_2,a_2,\ab a_3,b_3)$ lengths $(\ell +\delta_1, \ell,\ab \ell,\ab \ell + \delta_1)$.
	
	Since any 4‑cycle of $K_{3,3}$ can be mapped to any other 4‑cycle, $G$ can be relabeled so that the chains $(a_1, b_1,\ab b_2, a_2)$ have lengths $(\ell +\delta_1, \ell,\ab \ell,\ab \ell + \delta_1)$. Compared to the first labeling, $\delta_1$ and the length of $a_1$ is unchanged, but $\delta_2$ has strictly increased from $-\delta_1$ to $\delta_1$. Case~2 is thus inconsistent with~\ref{C3}. 
	
	\emph{Case 3:} $\delta_1 \geq 2$ while $\delta_1 + \delta_2 > 0$ and $\delta_1 + \delta_3 = 0$.  Just as in Case~2, $(a_1, b_1)$ and $(b_3, a_3)$ have the same lengths. Eliminate $a_3$ and~$b_3$ from~\eqref{d4K33} to obtain
	\begin{align}
		\begin{split}
			\Delta ={}& \big(\delta_1-1\big)\big(a_1a_2 + b_1b_2 + c_2(a_1 + b_1) + c_3(a_2 + b_2)\big)\\
			& +   \big(\delta_1 + \delta_2 - 1\big)\big(a_1b_1 + c_3(a_1 + b_1)\big) - \big(a_2b_2 + c_2(a_2 + b_2)\big)
		\end{split} \nonumber\\
		\begin{split}\label{d4K33case3}
			={}&\big(\delta_1 -1 \big) \big(b_1b_2 + c_3(a_2 + b_2)\big) + \big(\delta_1 + \delta_2 - 1\big)\big(a_1b_1 + c_3(a_1 + b_1)\big)\\
			&+ a_2\big((\delta_1-1)a_1 - b_2\big) + c_2\big((\delta_1-3)(a_1+b_1) +(a_1-a_2) + (a_1-b_2) + 2b_1\big) \,.
		\end{split}
	\end{align}
	We will make repeated use of the inequalities $a_1\geq a_2\geq b_2$. The former follows from~\ref{C1}, since $a_2$~is adjacent to~$a_3$ which has the same length as~$b_1$. For the latter, consider that the isomorphism which swaps the vertices $a$ with~$b$ and $1$~with~$3$ makes $a_2$~and~$b_2$ exchange places, while the lengths of the chains with labels $a_1$~and~$b_1$ remain the same.  Hence, \ref{C3}~implies $a_2\geq b_2$. Consider furthermore the isomorphism of~$K_{3,3}$ which mirrors Figure~\ref{f.K33labels} vertically. It exchanges $(a_2,b_2)$ with $(c_1,c_3)$, while no other chain lengths change, since $a_3 = b_1$. Hence, \ref{C3}~implies $\delta_2\geq c_1-c_3$. Combining the two isomorphisms above, \ref{C3}~implies that $\delta_2 \geq c_3-c_1$. It follows that if $a_2 = b_2$, then $c_1 = c_3$, in which case \ref{C4}~implies that $b_2\leq c_3$. This last fact will be needed for Subcase~3b.
	
	\emph{Subcase 3a:} $\delta_1 \geq 3$. By our assumptions and since $a_1\geq a_2\geq b_2$, all four terms of~\eqref{d4K33case3} are nonnegative. To show that $\Delta>0$, suppose for contradiction that $\Delta=0$. In particular, the third term equals zero, and since it is also bounded below by $a_2(2a_1-b_2) = a_2a_1 + a_2(a_1-b_2)$, where both terms are nonnegative, we have $a_1a_2=0$. This implies $a_2=0$, which in turn implies $b_2=0$, so that $\delta_2=0$. From~\eqref{d4K33case3}, we obtain $\Delta \geq a_1(b_1 + c_3 + c_2)$. This forces~$b_1$ (hence also~$a_3$), $c_2$ and~$c_3$ to be zero chains. This results in the ``zero cycle'' $a_2c_2c_3a_3$. Hence, $\Delta > 0$.
	
	\emph{Subcase 3b:} $\delta_1 = 2$.  Let $\ell$ denote the length of~$b_1$, so that the cycle $(a_1,a_3,\ab b_3,b_1)$ has lengths $(\ell +2, \ell,\ab \ell + 2, \ell)$. Since $b_2$~is adjacent to~$b_3$ and~$b_1$, \ref{C1}~implies that $b_2 = \ell + s$ where $s\in\{0,1,2\}$. Now insert $\delta_1=2$ into~\eqref{d4K33case3} and use $a_2\geq b_2$ (in the form $\delta_2\geq 0$) to obtain
	\begin{align}
		\Delta &\geq \big(b_1b_2 + c_3(a_2+b_2)\big) + \big(a_1b_1 + c_3(a_1 + b_1)\big) + a_2(a_1-b_2) + c_2\big(a_1-a_2-(b_2-b_1)\big) \nonumber\\
		& \geq a_2c_3 + a_1(b_1 + c_3) + a_2(2-s) + c_2(a_1-a_2-s)\,.
		\label{subcase3b} 
	\end{align}
	We consider, in turn, the three possible values of~$s$. Suppose that $s=0$. Then \eqref{subcase3b} implies $\Delta \geq a_1(b_1 + c_3) + a_2 + c_2(a_1-a_2)$. Each of these terms is nonnegative; suppose for contradiction that they are all zero. Then $b_1=0$ (and hence $a_3=0$), $c_3=0$, $a_2=0$ and $c_2=0$, which gives the ``zero cycle'' $a_3c_3c_2a_2$. Hence, $\Delta>0$.
	
	Suppose that $s=1$. Insertion into~\eqref{subcase3b} implies 
	\begin{equation}
		\Delta \geq a_1(b_1 + c_3) + a_2 + c_2((a_1-1)-a_2)\,. \label{s=1}
	\end{equation}
	Since $a_2$ and~$c_2$ are both adjacent to~$b_2$ with length $\ell+1$, they are both bounded above by $a_1=\ell + 2$, in observance of~\ref{C2}. If on the one hand $a_2<a_1$, all three terms of~\eqref{s=1} are nonnegative, and one shows that $\Delta>0$ exactly as for $s=0$ (recalling that $a_1\geq 2$). If on the other hand $a_2=a_1$, then \eqref{s=1}~simplifies to $\Delta \geq a_1(b_1 + c_3) + (a_1 - c_2)$, in which both terms are nonnegative; suppose for contradiction that both are zero. Then $b_1=0$ and $c_3=0$. With $b_1 = 0$, we have a cycle of known lengths, so $a_3=0$, $a_1=2$ (so also $a_2=2$, and $c_2=2$ since the second term is zero) and additionally $b_2= \ell+1 = 1$, so $\delta_2=1$. Now the cycle $(a_2,a_3,c_3,c_2)$ has lengths $(2,0,0,2)$. However, remapping $(a_1,b_1,b_2,a_2)$ to this cycle would give us the same value for~$\delta_1$, the same length for~$a_1$, but a strictly larger $\delta_2$-value, which contradicts~\ref{C3}. Hence, $\Delta>0$.
	
	Suppose that $s=2$. Since $a_2$ and~$c_3$ are adjacent to~$a_3$ with length~$\ell$, they are bounded above by $a_1 = \ell + 2$ by~\ref{C1}. Since also $a_2\geq b_2 = \ell +2$, we have $a_2 = \ell + 2$. Recall that $a_2=b_2$ implies $c_3\geq b_2$, so we also have $c_3=\ell +2$. Lastly, $c_2$ is bounded above by $\ell + 3$ by~\ref{C2}, since adjacent to $b_2=\ell+2$. Hence, by \eqref{subcase3b}, $\Delta \geq a_2c_3 + a_1c_3 -2c_2 \geq 2(\ell+2)^2 - 2(\ell+3) > 0$.
	
	\emph{Case 4:} $\delta_1 \geq 2$ while $\delta_1 + \delta_2 > 0$ and $\delta_1 + \delta_3 > 0$. All terms of~\eqref{d4K33} are nonnegative, so $\Delta\geq 0$. To show that $\Delta > 0$, consider that $\delta_2$ and~$\delta_3$ can be either 
	\begin{eninline}[label=(\arabic*)]
		\item\label{nonneg} both nonnegative,
		\item\label{neg} both negative, or 
		\item\label{nonneg.neg} $\delta_2$~nonnegative but $\delta_3$ negative, recalling that $\delta_2\geq \delta_3$ because of~\ref{C3}.
	\end{eninline}
	
	Suppose~\ref*{nonneg}. From~\eqref{d4K33_1} we obtain that $\Delta \geq (a_2 + a_3)(b_2 + b_3) + (c_2 + c_3)(a_2 + a_3 + b_2 + b_3)$. Since  $a_2a_3b_3b_2$, $a_2a_3c_3c_2$ and $b_2b_3c_3c_2$ are cycles, each of which must contain at least one edge to avoid a ``zero cycle'', at least one of these two terms is positive, and $\Delta>0$.
	Suppose~\ref{neg}. By keeping only part of the first term of~\eqref{d4K33} we have $\Delta \geq a_2b_3 + a_3b_2 + c_2b_3 + c_3b_2$. Since $a_2a_3c_3c_2$ is a~cycle, at least one of these chains is positive, and since furthermore $\delta_2<0$ and $\delta_3<0$ implies that $b_2$ and~$b_3$ are positive, at least one of the four terms is positive, and $\Delta>0$. Finally, suppose~\ref*{nonneg.neg}. Since $\delta_3$ is negative, $b_3$ is~positive. From~\eqref{d4K33}, $\Delta \geq a_2b_3 + c_2b_3 + a_3b_3 + c_3b_3$. Since $a_2c_2c_3a_3$ forms a~cycle, at least one of the four terms is positive. Hence, $\Delta>0$.
\end{proof}

\begin{prop}
	\label{p.onlyK33}
	For \(n\geq 1\), if \(G \in \G_{n,n+3}\) is not a weak \(K_{3,3}\)-subdivision, then there is a bridgeless \(H\in \G_{n,n+3}\) such that \(d_2(H) \leq d_2(G)\), \(d_3(H) < d_3(G)\) and \(d_4(H) < d_4(G)\). Hence, any \(p\)‑optimal \((n,n+3)\)-graph is a weak \(K_{3,3}\)-subdivision.
\end{prop}

\begin{proof}
	If $G$ has a bridge, the statement follows from Proposition~\ref{p.bridge}. Suppose that $G$ is bridgeless. If $G$ is a weak $\Pi_3$-subdivision, then the conclusion follows from Proposition~\ref{p.noPi3}. Suppose that $G$ is not a weak $\Pi_3$-subdivision. By Theorem~\ref{t.3ec}\ref{tt.3ec2}\ref{tt.3ec3}, there is an equivalent graph~$G^\dag$ which is an imbalanced weak $K_{3,3}$-subdivision. The conclusion then follows from Proposition~\ref{p.K33balanced}.
\end{proof}

\subsection{Correcting the \texorpdfstring{$K_{3,3}$}{K33}-theorem}
\label{ss.correction}

Before Theorem~\ref{t.k3}, which identifies the uniquely optimal $(k=3)$-graphs, we point out three ways in which our result gives a different picture than the hitherto accepted result by Wang~\cite{Wang94}.
\begin{itemize}[partopsep=0pt]
	\item Wang's theorem is wrong regarding the graphs described by the case $r=-4$ below. Hence, every ninth graph of the sequence, starting from the eleventh graph (see Figure~\ref{f.firstk3}), is a UMRG which has not been described before.
	
	\item The description of the sequence by perfect matchings of~$K_{3,3}$ fails (see Section~\ref{ss.k=3background}). A~simple characterization of the graph sequence now seems elusive---we need particular rules for different values of~$r$.
	
	\item Moreover, it is not even possible to generate the graphs by successive vertex insertion in~$K_{3,3}$ (consider the 11th and 12th graphs in Figure~\ref{f.firstk3}). 
	Illustrations such as~\cite[Fig.~2]{GilMyr}, \cite[Fig.~1]{Llagostera} and \cite[Fig.~2]{RomeroGSconj}, which specify an ordering of the edges of $K_{3,3}$, therefore have to be abandoned.
\end{itemize}

\begin{figure}[b!]
	\vspace{2mm}
	\centering
	\includegraphics{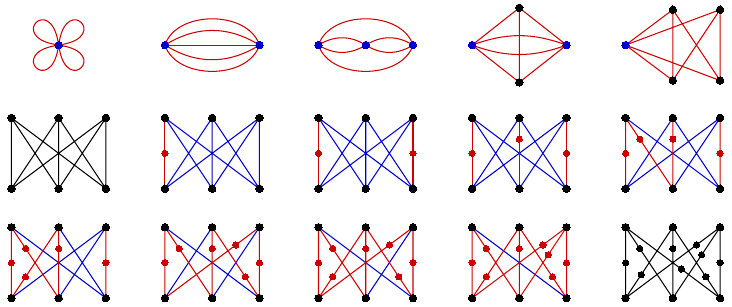}
	\caption{The first fifteen uniquely optimal $(n,n+3)$-graphs. The pattern continues and cycles every nine graphs. Red chains are one edge longer than blue ones; a blue vertex indicates the presence of one or more zero chains relative to~$K_{3,3}$.}
	\label{f.firstk3}
\end{figure}

\begin{thm}
	\label{t.k3}
	For each \(n\geq 1\) and \(m = n + 3\), there is a uniquely optimal graph in~\(\G_{n,m}\), which also uniquely maximizes the number of spanning trees. This graph is a balanced weak subdivision of~\(K_{3,3}\) (``weak'' is not needed when \(n\geq 6\)) which is specified up to isomorphism by the following set of additional conditions, where \(m \equiv r \pmod{9}\) and \(r \in [-4 \twodots 4]\).
	\begin{itemize}
		\item If \(r\in \{0,\pm1\}\), no further condition is needed.
		\item If \(r\in\{\pm2,\pm3\}\), the two or three longer (for positive~\(r\)) or shorter (for negative~\(r\)) chains correspond to a matching in $K_{3,3}$.
		\item If \(r=4\), the four longer chains correspond to a~3‑path and a~1‑path, disjoint, in~\(K_{3,3}\).
		\item If \(r=-4\), the four shorter chains correspond to two disjoint 2‑paths in~\(K_{3,3}\).
	\end{itemize}
\end{thm}

\begin{proof}
	We fix $n\geq 2$ and hence $m\geq 5$; the degenerate first case can be separately verified. $\B_m(K_{3,3})$ contains the balanced weak $K_{3,3}$-subdivisions.
	Recall that $\abs{r}$ is the number of chains which are longer or shorter than the standard length~$q$ as defined by \eqref{euclid}. If $r=0$, all chains have the same length, so $\B_m(K_{3,3})$ contains only a single, perfectly balanced graph. There is also only one graph when $r=\pm1$, since $K_{3,3}$ is edge-transitive.
	
	Consider a balanced weighting of~$K_{3,3}$ with the implied $G\in\B_m(K_{3,3})$. Suppose that $r\in\{\pm2,\ab \pm3,\ab \pm4\}$. In general, different placements of the heavier (lighter) weights in~$K_{3,3}$ will yield nonisomorphic weak subdivisions.
	A little thought reveals that the following alternatives, illustrated in Figure~\ref{f.options}, exhaust~$\B_m(K_{3,3})$. (If $m=5$, then option~\ref*{m5undef} is not well-defined because of a \hyperlink{zerocycle}{zero cycle}.) 
	
	\begin{figure}[hb!]
		\centering
		\includegraphics{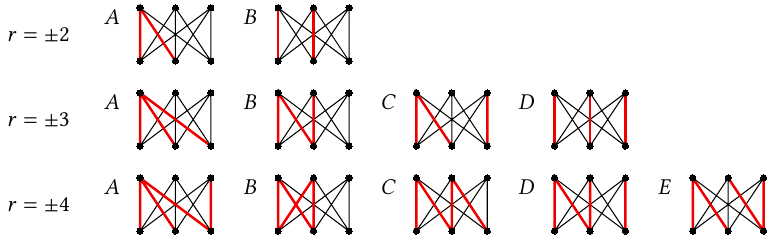}
		\caption{Thick red edges signify chains which are all one edge longer or shorter than the others. The different arrangements are nonisomorphic balanced $K_{3,3}$-subdivisions (when all chains are positive).}
		\label{f.options}
	\end{figure}
	
	\begin{enumerate}[label={If $r=\pm\arabic*$:}]
		\setlist[eninline,1]{label=$(\Alph*)$, ref=$\Alph*$}
		\setcounter{enumi}{1}
		\item The two heavier (lighter) edges are either
		\begin{eninline}
			\item  adjacent, or
			\item  nonadjacent.
		\end{eninline}
		\item The three heavier (lighter) edges form either
		\begin{eninline}
			\item  a~3‑star,
			\item  a~3‑path,
			\item  a 2-path and a~1‑path, or
			\item  a~perfect matching.
		\end{eninline}
		\item The four edges as above form either
		\begin{eninline}
			\item  a~3‑star with an extra edge attached,
			\item  a~4-cycle,\label{m5undef}
			\item  a~4‑path,
			\item  a~3‑path and a~1‑path, or
			\item  two 2‑paths.
		\end{eninline}
	\end{enumerate}
	
	The main work is to compare the above alternatives with regard to the number of 4‑dis\-connections, which we, as usual, group by the size of the smallest bond. Since the number of $(4,2)$-disconnections is the same for all $G\in\B_m(K_{3,3})$, we have
	\begin{equation}
		\label{d4k3}
		d_4(G) = d_{4,2} + d_{4,3}(G) + b_4(G)\,, \qquad\mathrlap{G\in \B_m(K_{3,3})\,,}
	\end{equation}
	where $d_{4,2}$ is a constant (and $b_4(G)$ equals $d_{4,4}(G)$ by definition).
	
	We now expand upon the notation of Definition~\ref{d.isolation}.
	\begin{itemize}
		\item Recall that $\piv(u)$ counts the number of edges incident to $u\in K_{3,3}$ which are either heavier or lighter than~$q$, with sign.
		\item Let $\dv_{4,3}(u)$ denote the number of $(4,3)$-disconnections of~$G$ which isolate~$u$ in~$K_{3,3}$. 
		\item Let $\pie(e)$ denote the number of edges adjacent to $e\in K_{3,3}$ which are either heavier or lighter than $q$, with sign.
		\item Let $b_4^\e(e)$ denote the number of 4‑bonds of~$G$ which isolate~$e$ in~$K_{3,3}$.
	\end{itemize}

	Label the vertices of~$K_{3,3}$ by $(u_i)_{i=1}^6$. Since $G$ has only trivial 3‑bonds, we have
	\begin{equation}
		\label{dv43}
		d_{4,3}(G) = \sum_{\mathclap{i\in [1\twodots 6]}} \dv_{4,3}(u_i) \,.
	\end{equation} 
	Given a vertex $u\in K_{3,3}$, a disconnection of~$G$ counted by $\dv_{4,3}(u)$ contains a 3‑bond isolating~$u$ (counted by~$b_3^\rv(u)$, see Table~\ref{tb.3bonds}) and a fourth edge from any chain not incident to~$u$. There are $9q+r$ edges in~$G$, and the total length of the chains incident to~$u$ is $3q + \piv(u)$. Hence, the number of possible choices for the fourth edge is $6q +r -\piv(u)$. It follows that the possible values of $\dv_{4,3}(u)$ are as given by Table~\ref{tb.dftv}.
	
	\begin{table}[b!]
		\vspace{-1.5mm}
		\renewcommand{\arraystretch}{1.05}
		\caption{Functions for counting and comparing $d_{4,3}(G)$, the number of 4‑disconnections in which the smallest bond is a~3‑bond, in balanced weak subdivisions of~$K_{3,3}$.}
		\label{tb.dftv}
		\begin{tabularx}{\textwidth}{L \Rr \LlX L}
			\toprule
			\bm{\piv(u)} && \bm{\mathrlap{\dv_{4,3}(u)}} & \bm{\varphi_{4,3}\big(\piv(u)\big)}\\
			\midrule
			\phantom{\pm}0 & \phantom{(q\pm1)^0}q^3(6q+r) & {}= 6q^4 + rq^3 & 0\\
			\cline{2-3}
			\pm1 & (q\pm1)^1q^2(6q+r\mp1) & {}= 6q^4 + (r\pm5)q^3 \phantom{0} + (\abs{r}-1)q^2 & 0 \Tstrutt\\
			\cline{2-3}
			\pm2 & (q\pm1)^2q^1(6q+r\mp2) & {}= 6q^4 + (r\pm10)q^3 + \bigl(2(\abs{r}-1)\:{+} & \vspace{-1.5pt}\Tstrutt\\
			&& \phantom{{}={}} 4\bigr)q^2 + (r\mp2)q & 4q^2 + (r\mp2)q \\
			\cline{2-3}
			\pm3 & \phantom{q^0}(q\pm1)^3(6q+r\mp3) &{}= 6q^4 + (r\pm15)q^3 + \bigl(3(\abs{r}-1)\:{+} & \vspace{-1.5pt}\Tstrutt\\
			&& \phantom{{}={}} 12 \bigr)q^2 + (3r\mp3)q + (\abs{r}-3) & 12q^2 + (3r\mp3)q + (\abs{r}-3) \\
			\bottomrule
		\end{tabularx}
	\end{table}
	
	By considering~\eqref{dv43} and Table~\ref{tb.dftv}, we note that $d_{4,3}(G)$ is a quartic polynomial in~$q$ with leading coefficient~36. Furthermore, we note that the cubic coefficient of each term in~\eqref{dv43} has a fixed part equaling~$r$ and a dependent part which equals $5\piv(u_i)$. Summing the six terms and using that the $\piv$-values sum to~$2r$ by~\eqref{pi2r} gives us a cubic coefficient of $d_{4,3}(G)$ which equals~$16r$.
	
	Regarding the quadratic coefficient of $d_{4,3}(G)$, consider that every incidence of a chain of length $q+1$ or $q-1$ to some vertex~$u$ is associated with a unit increase in the absolute value of~$\piv(u)$, and hence, as can be read from Table~\ref{tb.dftv}, adds at least $(\abs{r}-1)q^2$ to the sum of~\eqref{dv43}. Recalling that the total number of such incidences is fixed and equals~$\abs{2r}$, we define $\varphi_{4,3}(\piv(u))$ to be the quadratic polynomial in~$q$ with leading coefficient equal to the part of the quadratic coefficient of $\dv_{4,3}(u)$ which exceeds $\abs{\piv(u)}(\abs{r}-1)$ (see Table~\ref{tb.dftv}), and with linear and constant terms equal to the linear and constant terms of $\dv_{4,3}(u)$.
	
	Applying the above considerations to~\eqref{dv43} yields
	\begin{equation}
		\label{dv43final}
		d_{4,3}(G) = 36q^4 + 16rq^3 + \abs{2r}(\abs{r}-1)q^2 + \sum_{\mathclap{i\in[1 \twodots 6]}}\varphi_{4,3}\big(\piv(u_i)\big)\,, \quad G\in\B_m(K_{3,3}\mathrlap{)\,,}
	\end{equation}
	where the terms preceding the ``$\varphi$-sum''' are constant, but the $\piv$-values depend upon~$G$.
	
	We now turn our attention to~$b_4(G)$. (The reasoning is analogous to the proof of Proposition~\ref{p.d3min} from \eqref{b3G} to~\eqref{phisum}.) Label the edges of~$K_{3,3}$ by $(e_i)_{i=1}^9$, while separately keeping their weights.  Since $G$ only has trivial 4‑bonds, we have
	\begin{equation}
		\label{b4e}
		b_4(G) = \sum_{\mathclap{i\in [1\twodots 9]}} b_4^\e(e_i) \,.
	\end{equation} 

	Consider the possible values of~$b_4^\e(e_i)$, given in Table~\ref{tb.bfe}. Clearly, \eqref{b4e} is a quartic polynomial in~$q$ with the leading coefficient~9.
	Since every edge in~$K_{3,3}$ is adjacent to four other edges, and there are $\abs{r}$~edges which can be counted by~$\pie(\ph)$, we have
	\begin{equation}
		\label{pi4r}
		\sum_{\mathclap{i\in[1 \twodots 9]}} \pie(e_i) = 4r\,. 
	\end{equation}
	And since the cubic coefficient of~$b_4^\e(e_i)$ always equals~$\pie(e_i)$,  the cubic coefficient of~\eqref{b4e} is~$4r$.
	We define $\varphi_{4,4}(\pie)$ to be the quadratic, linear and constant terms of~$b_4^\e(\ph)$, regarded as a function of~$\pie$. The possible values are given in Table~\ref{tb.bfe}. Applying our considerations to~\eqref{b4e} yields
	\begin{equation}
		\label{b4}
		b_4(G) = 9q^4 + 4rq^3 + \sum_{i\in[1\twodots 9]} \varphi_{4,4}(\pie_i) \quad G\mathrlap{{}\in \B_m(K_{3,3})\,,}
	\end{equation}
	where only $\pie_i= \pie(e_i)$ depends upon~$G$.
	
	\begin{table}[t!]
		\caption{Functions for counting and comparing the number of trivial 4‑bonds in balanced weak subdivisions of a given cubic, 3‑edge-connected graph.}
		\label{tb.bfe}
		\renewcommand{\arraystretch}{1}
		\begin{tabularx}{\textwidth}{C \LX X \Lr \Ll X L}
			\toprule
			\quad &\bm{\pie(e)} &&& \bm{\mathrlap{b_4^\e(e)}} && \bm{\varphi_{4,4}(\pie)} \\
			\midrule
			&\phantom{\pm}0 && \phantom{(q\pm1)^2}q^4 &{}= q^4 && 0 \\
			&\pm1 && (q\pm1)^1q^3 &{}= q^4 \pm q^3 && 0 \\
			&\pm2 && (q\pm1)^2q^2 &{}= q^4 \pm 2q^3 + q^2 && q^2\\
			&\pm3 && (q\pm1)^3q &{}= q^4 \pm 3q^3 + 3q^2 \pm q && 3q^2 \pm q \\
			&\pm4 && (q\pm1)^4 &{}= q^4 \pm 4q^3 + 6q^2 \pm 4q + 1 && 6q^2 \pm 4q + 1 \\
			\bottomrule
		\end{tabularx}
	\end{table}
	
	Substitute \eqref{dv43final} and~\eqref{b4} into~\eqref{d4k3} to obtain
	\begin{equation}
		d_4(G)  =  K  +  \smashoperator{\sum_{i\in[1 \twodots 6]}}\varphi_{4,3}(\piv_i)  +  \smashoperator{\sum_{j\in[1\twodots 9]}} \varphi_{4,4}(\pie_j) = K + \Phi_{4,3}(G) + \Phi_{4,4}(G) \,, \quad G\in \B_m(K_{3,3})\,,
	\end{equation}
	where $K$ is a constant and $\Phi_{4,3}(\ph)$ and~$\Phi_{4,4}(\ph)$ are defined by the preceding respective sums. Hence, $d_4(\ph)$ is minimized in~$\B_m(K_{3,3})$ if and only if the sum $\Phi_{4,3} + \Phi_{4,4}$, which we denote by~$\Phi_4$, is minimized. In Table~\ref{tb.options}, this sum is calculated from the multisets $\mset{\piv_i}$ and~$\mset{\pie_j}$ by summing the values of $\varphi_{4,3}$ and~$\varphi_{4,4}$ according to Table~\ref{tb.dftv} and~\ref{tb.bfe}. (The multisets are obtained by inspecting the vertices and edges of the graphs in Figure~\ref{f.options}.)
	
	\begin{table}[b!]
		\caption{For each possible graph $G\in \B_m(K_{3,3})$: The $\piv$- and $\pie$‑multisets (superscripts for multiplicity), the resulting $\Phi$-values, and whether $G$ minimizes~$d_4$ (by minimizing~$\Phi_4$) and minimizes~$d_3$ (by balanced $\piv$-values).}
		\label{tb.options}
		\setlength{\tabcolsep}{4pt}
		\begin{tabularx}{\textwidth}{L C L L \LX L L c c}			
			\toprule
			\bm{r} & \bm{G} & \bm{{\phantom{\pm}\mset{\piv_i}_{i=1}^6}} &  \bm{{\phantom{\pm}\mset{\pie_i}_{i=1}^9}} & \bm{\Phi_{4,3}}  & \bm{\Phi_{4,4}} & \bm{\Phi_4} & $\bm{d_4}$ & $\bm{d_3\!}$ \\
			\midrule
			\pm 2 
			& A  & \pm\mset{2^1 1^2 0^3} & \pm\mset{2^1 1^6 0^2} 
			& 4q^2 & q^2 & 5q^2 & \xmark & \xmark \Tstrut\\
			& B & \pm\mset{1^4 0^2} & \pm\mset{2^2 1^4 0^3} 
			& 0 & 2q^2 & 2q^2 & \cmark & \cmark \\
			\midrule[0pt]
			\pm 3 
			& A & \pm\mset{3^1 1^3 0^2} & \pm \mset{2^3 1^6}
			& 12q^2 \pm 6q & 3q^2 
			& 15q^2 \pm 6q & \xmark & \xmark \\
			& B & \pm\mset{2^2 1^2 0^2} & \pm\mset{2^4 1^4 0^1}
			& 2(4q^2 \pm q) &  4q^2 
			& 12q^2\pm2q & \xmark & \xmark \\
			& C & \pm\mset{2^1 1^4 0^1} & \pm\mset{3^1 2^2 1^5 0^1}
			& 4q^2 \pm q & (3q^2 \pm q) + 2q^2 
			& 9q^2 \pm 2q & \xmark & \xmark \\
			& D &  \pm\mset{1^6} &  \pm\mset{2^6 0^3}
			& 0 & 6q^2 & 6q^2 & \cmark & \cmark \\
			\midrule[0pt]
			\pm 4 
			& A & \pm\mset{3^1 2^1 1^3 0^1} 
			& \pm\mset{3^1 2^5 1^3} & 
			\begin{tabular}{L}
				(12q^2 {\pm} 9q {+} 1) \vspace{-4pt}\\
				+(4q^2 {\pm} 2q)
			\end{tabular} 
			&  (3q^2\pm q)+ 5q^2
			& 24q^2 {\pm} 12q {+} 1 & \xmark & \xmark \\
			& B & \pm\mset{2^4 0^2} & \pm\mset{2^8 0^1} 
			& 4(4q^2 {\pm} 2q) & 8q^2
			& 24q^2\pm 8q & \xmark & \xmark \\
			& C & \pm\mset{2^3 1^2 0^1} & \pm\mset{3^2 2^3 1^4} 
			& 3(4q^2{\pm}2q) & 2(3q^2\pm q)+3q^2
			& 21q^2 \pm 8q & \xmark & \xmark \\
			& D & \pm\mset{2^2 1^4} & \pm\mset{3^2 2^4 1^2 0^1} 
			& 2(4q^2{\pm}2q) &  2(3q^2\pm q)+4q^2
			& 18q^2 \pm 6q & \clap{\cmark\kern-0.1em/\xmark} & \cmark \\
			& E & \pm\mset{2^2 1^4} & \pm\mset{4^1 2^4 1^4} 
			& 2(4q^2{\pm}2q) & (6q^2 {\pm} 4q{+}1){+}4q^2
			& 18q^2 \pm 8q {+} 1 & \clap{\xmark/\cmark} & \cmark \\
			\bottomrule
		\end{tabularx}
	\end{table}
	
	Furthermore, $d_3(\ph)$ is minimized over $\B_m(K_{3,3})$ if and only if $K_{3,3}$ has balanced $\piv$-values. This is by Proposition~\ref{p.d3min}, except for the \emph{only~if} direction when $m=5$, which is not strictly needed, but can easily by verified for the sake of Table~\ref{tb.options}.
	
	We find in Table~\ref{tb.options} that, with $r$~given, there is exactly one graph $G\in\B_m(K_{3,3})$ which minimizes $d_4(\ph)$, and that this graph also minimizes $d_3(\ph)$. By Proposition~\ref{p.K33balanced} and \ref{p.onlyK33} we conclude that $G$ minimizes $d_2(\ph)$, minimizes $d_3(\ph)$ and uniquely minimizes $d_4(\ph)$ in the entire set~$\G_{n,m}$. Hence, this graph, as described by Theorem~\ref{t.k3}, is uniquely optimal.
\end{proof}

\section{Results and conjectures for \texorpdfstring{$\MakeLowercase{k}=4$}{k=4} and \texorpdfstring{$\MakeLowercase{k}=5$}{k=5}}
\label{s.k=4,5}
The methods of this paper can be used to obtain partial solutions to the reliability \hyperlink{prb}{problem} for $k=4$ and $k=5$, and hopefully for even larger $k$. We do~not aim to exhaust the possibilities here, and it will be clear that already the cases $k=4,5$ present new challenges.

\subsection{Intriguing \texorpdfstring{$p$}{p}-dependencies for Wagner subdivisions (\texorpdfstring{$k=4$}{k=4})}
\label{ss.k=4}
In a long-standing conjecture by Ath and Sobel~\cite{AthSobel}, a particular sequence of successive vertex insertions into the Wagner graph (Figure~\ref{f.wagner}) was ventured to yield uniformly optimal graphs (as is the case for~$K_4$ and as was believed to be the case for~$K_{3,3}$). While Romero~\cite{Wagner} demonstrated the Wagner graph itself to be uniquely optimal (this is clear from the~proof) among simple graphs, Romero and Safe~\cite{RomeroSafe} subsequently revealed that Ath and Sobel's conjecture fails for every twelfth subdivision, where there in~fact do~not exist uniquely optimal graphs. By the end of this section, we will be ready to state a new conjecture which gives a rather different picture of the $(k=4)$-case.

\begin{figure}[htb!]
	\centering
	\begin{tikzpicture}
		\def\r{0.8}
		\foreach \a [count=\j, remember=\j as \jj (initially 1)] in {22.5,67.5,...,382.5}{
			\draw({\r*cos(\a)},{\r*sin(\a)}) node[vx](\j){} node[spv]{} -- (\jj.center);}
		\draw[Blue3, very thick] (1)--(5) (2)--(6) (3)--(7) (4)--(8);
		\node[inner xsep=\inxsep, inner ysep=\inysep] at (-1.7*\r,0.75*\r) {\lrg{$W$}};
	\end{tikzpicture}
	\caption{All known $p$‑optimal $(n,n+4)$-graphs are weak subdivisions of the Wagner graph~$W$, which is the 4-Möbius ladder. The thick blue edges are \emph{rungs}, the others are \emph{rails}.}
	\label{f.wagner}
\end{figure}
	
The main result of~\cite{RomeroSafe} is reformulated and extended to multigraphs in Theorem~\ref{t.RomSafe} below, with a supplementary proof. Theorem~\ref{t.Wlargem} then gives a new infinite set of $m$‑values for which there is no uniquely optimal $(m-4,m)$-graph. While the two theorems have the same structure, there are intriguing differences between them which indicate different mechanisms at play. While the former theorem gives comparatively small values of~$m$ for which there is no uniquely optimal graph, the latter theorem strongly suggests that there can be only finitely many uniquely optimal $(m-4,m)$-graphs.

\begin{thm}[Romero and Safe~\cite{RomeroSafe}]
	\label{t.RomSafe}
	Consider \(\G_{12q-8,12q-4}\), where \(q\geq 2\).
	\begin{enumenun}
		\item\label{tt.RSlargep} There is a unique graph~\(G\) which is \(p\)‑optimal for sufficiently large~\(p\). This graph is the balanced \(W\)‑subdivision in which the four shorter chains correspond to a perfect matching of rails.
		
		\item\label{tt.RSsmallp} A rearrangement of the rail chains of~\(G\) gives a different balanced subdivision of~\(W\) with more spanning trees, which is therefore strictly more reliable than~\(G\) for sufficiently small~\(p\). Hence, there is no uniquely optimal graph in \(\G_{12q-8,12q-4}\).
	\end{enumenun}
\end{thm}

\begin{rmks}
	\label{r.RomSafe}
	\begin{eninline}[label={(\arabic*)}, before={}]
		\item Originally about simple graphs, the supplementary proof below ensures that the theorem holds for multigraphs.
		\item The proof in~\cite{RomeroSafe} uses Theorem~10 of Wang and Zhang~\cite{Wang97}, which we consider unreliable (see discussion in Section~\ref{ss.3dis}). With the supplement below, there is no dependency upon~\cite{Wang97} (nor upon any other paper).
	\end{eninline}
\end{rmks}

\begin{proof}[Supplementary proof of Theorem \ref*{t.RomSafe} (see remarks above)]
	To identify the graphs which are $p$-optimal for sufficiently large~$p$ (henceforth: \emph{large-$p$-optimal}), we want to successively minimize the coefficients $d_i(\ph)$ of~\eqref{dcoeff} for $i$~up to~5 (the first is trivial). By Theorem~\ref{t.d2}\ref{tt.d2iff}, the graphs which minimize $d_2(\ph)$ are the balanced subdivisions denoted by~$\B_m(\D_4)$, where $m=12q-4$.
	By Theorem~\ref{t.d3min2}, subdivisions of $\D_4$-graphs which contain a~3‑cycle (see Figure~\ref{f.Dgraphs}) cannot minimize $d_3(\ph)$, since the cycles imply nontrivial 3‑bonds. The remaining distillations are the Wagner graph~$W$ and the Cube~$C$.
	
	By the same proposition, $d_3(\ph)$ is minimized in $\B_m({W,C})$ exactly when the correspondingly weighted $W$~or~$C$ has balanced $\piv$-values. We have four edges of weight~$q-1$, since $r=-4$, and eight vertices, so the lighter edges should form a perfect matching in $W$~or~$C$. There are three ``nonisomorphic'' perfect matchings in~$W$ and two in~$C$, as shown in~\cite{RomeroSafe}, so we have five graphs remaining. Among them, it was also shown that $d_4(\ph)$ is minimized exclusively by the $W$‑subdivision for which the perfect matching is contained in the 8‑cycle. 
	This proves part~\ref*{tt.RSlargep}.
	
	Now, let $G'\in \G_{12q-8,12q-4}$ be the balanced subdivision of $W$ such that, going around the cycle of rails, the corresponding chains have lengths $(\tilde{q}, q, \tilde{q}, q,\ab q, \tilde{q}, q, \tilde{q})$, where $\tilde{q} = q-1$. By showing that $d_5(G')<d_5(G)$, one obtains~\ref*{tt.RSsmallp}.
\end{proof}

The heuristics behind Theorem~\ref{t.Wlargem} below, which we now explain, strongly suggests that for sufficiently large~$m$, there is no set $\G_{m-4,m}$ with a uniquely optimal graph. The Wagner graph, which is the 4‑Möbius ladder, is not edge-transitive; by Lemma~\ref{l.Wtrees}, the rails belong to slightly more spanning trees than the rungs. For $W$‑subdivisions, this should mean that the rails are ``more important'' for connectedness when $p$ is small, and hence that the rail chains should be slightly shorter than the rung chains---presumably by a constant factor---to maximize $R(G,p)$. In Proposition~\ref{p.Wtreepoly}\ref{pp.Wmax}, a continuous version of the problem will suggest $\sqrt{10}-2$ as the proportionality constant.

\begin{thm}
	\label{t.Wlargem}
	Consider \(\G_{12q-4,12q}\), where \(q\geq 1\).
	\begin{enumenun}
		\item\label{tt.Wlargem.largep}
			For sufficiently large~\(p\), there is a unique \(p\)‑optimal graph, namely the perfectly balanced \(W\)‑subdivision \(G\).
		\item\label{tt.Wlargem.smallp}
			For \(q\geq8\), there is an imbalanced subdivision of~\(W\) with more spanning trees than~\(G\), which is therefore strictly more reliable than~\(G\) for sufficiently small~\(p\). Hence, there is no uniquely optimal graph in~\(\G_{12q-4,12q}\).
	\end{enumenun}
\end{thm}

\begin{proof}[Proof of part~\ref*{tt.Wlargem.largep}]
	As in the proof of Theorem~\ref{t.RomSafe}, we will successively minimize $d_2(\ph)$, $d_3(\ph)$ and~$d_4(\ph)$. Since $m=12q$, there is only one balanced graph for each of the distillations in~$\D_4$, and by Theorem~\ref{t.d2}\ref{tt.d2iff}, these are exactly the graphs which minimize $d_2(\ph)$. Among these four graphs, $d_3(\ph)$ is minimized exclusively by the subdivision of~$W$ and that of the Cube~$C$, by Theorem~\ref{t.d3min2} (the $\piv$-condition is vacuous). Since the chains have length~$q$, we denote the former graph by~$W_q$ and the latter by~$C_q$.
	
	We now wish to show $d_4(C_q)-d_4(W_q) > 0$. A~pairing of the chains of $W_q$~and~$C_q$ induces a pairing of the $(4,2)$-disconnections, so $d_{4,2}(W_q)=d_{4,2}(C_q)$. All 3‑bonds of $W_q$~and~$C_q$ are trivial, and it is easy to see that $d_{4,3}(W_q)=d_{4,3}(C_q)$. We also note that the graphs have the same number of trivial 4‑bonds. Denoting the number of nontrivial 4‑bonds by~$b_4^\n(\ph)$, we therefore have $d_4(C_q)-d_4(W_q) = b_4^\n(C_q) - b_4^\n(W_q)$. By inspection, one can see that both $W$ and~$C$ has exactly one nontrivial 4‑bond for each pair of disjoint 4‑cycles, so $b_4^\n(W)=2$ and $b_4^\n(C) = 3$, and hence $b_4^\n(W_q) = 2q^4$ and $b_4^\n(C_q) = 3q^4$.  Therefore, $d_4(C_q)-d_4(W_q) = q^4 > 0$, and we have shown that $W_q$ is the unique large-$p$-optimal $(12q-4, 12q)$-graph.
\end{proof}
	
\begin{proof}[Proof of part \ref*{tt.Wlargem.smallp}] 
	In $W_q$, we now fix a rail chain which we denote by~$\ell$ and a rung chain adjacent to~$\ell$ which we denote by~$g$. Let $W'_q$ denote the graph obtained by moving an edge from~$\ell$ to~$g$. We show that $W'_q$ has more spanning trees than~$W_q$, i.e.\ that $t(W'_q)-t(W_q)>0$. Considering that each spanning tree of $W_q$ and~$W'_q$ naturally maps to a spanning tree of~$W$, we can obtain $t(W_q)$ and~$t(W'_q)$ from the spanning trees of~$W$ and the chain lengths of $W_q$ and~$W'_q$, respectively. (For each spanning tree of~$W$, the relevant chains are those which do \emph{not} map to edges in the tree.)
	
	As is well known, $t(W)=392$; this also follows from Lemma~\ref{l.Wtrees}\ref{ll.Wrungs} below. Hence, $t(W_q) = 392q^5$. Using Lemma~\ref{l.Wtrees}\ref{ll.Wparticular}, we similarly obtain that $t(W'_q) = 114q^5 + 117(q+1)q^4 + 110(q-1)q^4 + 51(q-1)(q+1)q^3$. It follows that 
	\begin{equation}
		t(W'_q)-t(W_q) = (117-110)q^4-51q^3 = (7q-51)q^3 \,,
	\end{equation}
	which is positive since $q\geq 8$. By Proposition~\ref{p.largesmallp}\ref{pp.spanningtrees}, $W'_q$ is strictly more reliable than~$W_q$ for sufficiently small~$p$.
\end{proof}

\begin{lem}
	\label{l.Wtrees}
	\begin{eninline}[label={\normalfont(\alph*)}, ref=(\alph*)]
		\item~Let \(t_i\) be the number of spanning trees of \(W\) which contain exactly \(i\)~rungs. Then \((t_i)_0^4 = (8,64,\ab 160,\ab 128,32)\). \label{ll.Wrungs}
		\item~Let \(g\) be a particular rung of~\(W\) and let \(\ell\) be a rail adjacent to~\(g\). Out of the 392~spanning trees of~\(W\), 114~contain both \(\ell\) and~\(g\), 117 contain \(\ell\) but not~\(g\), 110 contain \(g\) but not~\(\ell\) and 51 contain neither \(\ell\) nor~\(g\). \label{ll.Wparticular}
	\end{eninline}
\end{lem}

\begin{proof}
	The above lemma can of~course be verified more or less manually. We will give one way to systematically count the spanning trees of~$W$ and obtain the desired numbers. Out of the spanning trees (henceforth, just ``trees'') with exactly $i$~rungs, counted by~$t_i$, let $t_i^g$ count those which contain~$g$, let $t_{i}^{\ell}$ count those which contain~$\ell$ and let $t_{i}^{g\ell}$ count those which contain both. By symmetry we have, $t_{i}^{g} = \frac{i}{4}t_i$, and since a spanning tree with $i$~rungs contains $7-i$~rails, we likewise have $t_{i}^{\ell} = \frac{7-i}{8}t_i$. We refer to the representation of~$W$ in Figure~\ref{f.wagner} and consider in turn $i=0,1,2,3,4$.
	
	$i=0$: Deleting all rungs gives an 8‑cycle of rails, and subsequently deleting any rail gives a~tree. Hence, $t_0=8$, $t_0^\ell = 7$, $t_0^g = 0$ and $t_0^{g\ell} = 0$.
	
	$i=1$: The rung~$g$ partitions the rail cycle into two parts. Combine $g$ with three out of the four rails on each side to obtain a~tree. Hence, $t_1^g = 16$. One-fourth of these trees do~not include~$\ell$, so $t_1^{g\ell}=12$.
	
	$i=2$: Let $t_2^g = t^{g}_{2'} + t_{2''}^{g}$, where $t^{g}_{2'}$ counts the trees in which the two rungs belong to the same 4‑cycle and $t^{g}_{2''}$ those in which they do~not (and~so, in our picture, form a~cross). For~$t^{g}_{2''}$, the two rungs partition the rails into four sets of~two. Choose one set to connect the rungs and then choose one rail from each of the three remaining sets. This gives $t^{g}_{2''}=4\cdot2^3=32$. Five-eighths of these contain~$\ell$, so $t_{2''}^{g\ell}=20$. For~$t^{g}_{2'}$, the two rungs partition the rails into two sets of three and two singletons, and the second rung determines whether $\ell$ is in a~3‑set or single. Suppose that $\ell$ is in a~3‑set. The rungs must be connected by either a singleton or by a~3‑set. In the former case, choose a singleton and then two out of three edges from each 3‑set. This gives 18~trees, 12~of which contain~$\ell$. In the latter case, choose a~3‑set and then two edges from the other 3‑set. This gives 6~trees, 5~of which contain~$\ell$. Now suppose that $\ell$ is a singleton. This likewise gives 24~trees, but only~9 (half of the trees in ``the former case'') contain~$\ell$. The cases sum to $t^{g}_{2'} = 48$ and $t_{2'}^{g\ell}=26$, and finally to $t^{g}_{2} = 80$ and $t_{2}^{g\ell}=46$.
	
	$i=3$: Three rungs partition the rails into four singeltons and two sets of two. Starting with~$g$, there are two ways to choose two more rungs so that $\ell$ is a singleton, and one way so that $\ell$ is in a~2‑set. Suppose that $\ell$ is a singleton. Either there is a 2‑set with both edges included in the tree, or there is not. In the former case, choose such a~2‑set, then choose one edge from the other 2‑set, and finally choose one of the remaining four rails. This gives $2^3\cdot 4 = 32$ trees, 8~of which contain~$\ell$. In the latter case, choose one edge from each of the 2‑sets, and then choose two out of the remaining four rails without creating a~4‑cycle. This gives $2^3(6-2)=32$ trees, 16~of which contain~$\ell$. Now suppose that $\ell$ is in a~2‑set. In ``the former case'' above, there are 16~trees, 12~of which contain~$\ell$. In ``the latter case'', there are 16~trees, 8~of which contain~$\ell$. All in all, $t^{g}_{3} = 96$ and $t_{3}^{g\ell}=44$.
	
	$i=4$: The trees are obtained by choosing three rails without creating a~cycle, which can only be created by including two opposite rails. Hence, $t_4^g = 32$, and three-eighths include~$\ell$, so $t_4^{g\ell}=12$.
	
	Putting our numbers together and using the relations between $t_i$, $t_i^g$ and~$t_i^\ell$, we obtain $(t_i)_0^4 = (8,64,\ab 160,\ab 128,32)$, $(t_i^g)_0^4 = (0,16,\ab 80,\ab 96,32)$, $(t_i^\ell)_0^4 = (7,48,\ab 100,\ab 64,12)$ and $(t_i^{g\ell})_0^4 = (0,12,\ab 46,\ab 44, 12)$. By summing these sequences, $W$~has 392~trees, out of which 224 contain $g$, 231 contain $\ell$ and 114 contain both. Hence, 110~trees contain $g$ but not~$\ell$, 117 contain $\ell$ but not~$g$, and $392-(114+110+117)= 51$ contain neither $\ell$ nor~$g$.
\end{proof}

The next proposition relates the number of spanning trees to how the edges are distributed between the rails and rungs, and motivates Conjecture~\ref{cj.W}\ref{cjj.Wratio}.

\begin{prop}\ 
	\label{p.Wtreepoly}
	\begin{enumenun}
		\item\label{pp.Wnumber} Consider a subdivision~\(G\) of the Wagner graph, in which the rail chains have length~\(\ell\) and the rung chains have length~\(g\). The number of spanning trees of~\(G\) is
		\begin{equation}
			\label{Wsubtrees}
			t(G) = 8g^4\ell + 64g^3\ell^2 + 160g^2\ell^3 + 128 g\ell^4 + 32\ell^5 \,.
		\end{equation}
		\item\label{pp.Wmax} Consider a weighted Wagner graph~\(G\), where the rail weights~\(\ell\) and the rung weights~\(g\) are positive real numbers. Fixing \(m=8\ell+4g\) and then letting \(\ell\)~vary, the above expression \(t(G)\) is maximized if and only if
		\begin{equation}
			\frac{g}{\ell} = \sqrt{10}-2 \,.
		\end{equation}
	\end{enumenun}
\end{prop}

\begin{proof}[Proof of part~\ref*{pp.Wnumber}]
	Given a spanning tree of the Wagner graph with exactly $i$~rungs, there are $4-i$~rungs and $i+1$~rails which are not contained in the tree. Consider that the spanning trees of~$W$ induce a partitioning of the spanning trees of~$G$
	and use Lemma~\ref{l.Wtrees}\ref{ll.Wrungs} to obtain~\eqref{Wsubtrees}.
\end{proof}

\begin{proof}[Proof of part~\ref*{pp.Wmax}]
	Using $m=8\ell + 4g$, we substitute~$g$ with $m/4 - 2\ell$ in~\eqref{Wsubtrees} and simplify to obtain 
	\begin{equation}
		\label{Wrailtrees}
		t(G) = \frac{1}{32}\ell\left(m^2-32\ell^2\right)^2 .
	\end{equation}
	Since $\ell$ and~$g$ are positive, $\ell\in(0,m/8)$. By elementary calculus, the maximum of~\eqref{Wrailtrees} is obtained for $\ell=m/(4\sqrt{10})$ which yields $g = (\!\sqrt{10}-2)m/(4\sqrt{10})$ and hence $g/\ell = \sqrt{10}-2$.
\end{proof}

Considering the above results, we would like to propose the following progression of specifications, which still do~not paint a complete picture, of the $p$‑optimal $(k=4)$-graphs.

\begin{conj}
	\label{cj.W}
	Consider the sets \(\G_{m-4,m}\), for \(m\geq 4\).
	\begin{enumenun}
		\item\label{cjj.Wbasic} Every \(p\)‑optimal graph is a weak subdivision of the Wagner graph.
		\item For sufficiently large~\(p\), there is a unique \(p\)-optimal graph (necessarily balanced).
		\item In every \(p\)-optimal graph, the set of rung chains and the set of rail chains are each balanced.
		\item\label{cjj.Wratio} As \(p\) decreases and \(m\)~increases, the optimal ratio between the ``average'' rung chain length and the ``average'' rail chain length should approach \(\sqrt{10}-2\).
		\item As a consequence, a uniquely optimal graph exists for only finitely many~\(m\).
	\end{enumenun}
\end{conj}

\subsection{A weakened conjecture about Petersen subdivisions (\texorpdfstring{$k=5$}{k=5})}
\label{ss.k=5}
The famous Petersen graph (see Figure~\ref{f.petersen}), denoted by~$P$, is well known to be the smallest cubic graph of girth at least~5, and it has only trivial 3‑ and 4‑bonds. Since cycles of length 3~or~4 give rise to nontrivial bonds in cubic graphs (with exceptions when $n\leq6$), the Petersen graph seems especially promising as a distillation for $p$‑optimal graphs. Furthermore, $P$~itself has been shown to be uniformly optimal among simple $(10,15)$‑graphs~\cite{Petersen}.

\begin{figure}[hb!]
	\centering
	\includegraphics{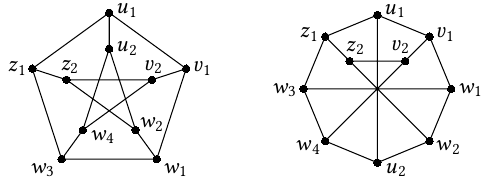}
	\caption{Two representations of the Petersen graph~$P$.}
	\label{f.petersen}
\end{figure}

Ath and Sobel also conjectured a particular sequence of successive subdivisions of~$P$ to yield uniformly optimal $(m-5,m)$-graphs~\cite{AthSobel}. However, Proposition~\ref{p.d3min} together with Theorem~\ref{t.d2}\ref{tt.d2iff} easily disproves the conjecture for all $m$ such that $m\equiv 5$ or $m\equiv 10$ modulo~15.
For such~$m$, the specified subdivisions (Figure~\ref{f.AthSobel})  do~not give balanced $\piv$-values, which implies that they do~not minimize $d_3(\ph)$ among the necessarily bridgeless graphs minimizing $d_2(\ph)$. Hence, they fail to be large-$p$-optimal.

\begin{figure}[b!]
	\centering
	\def\r{1.0}
	\begin{tikzpicture}
		\foreach \a [count=\j, remember=\j as \jj (initially 1)] in {0,45,...,360}{
			\draw({\r*sin(\a)},{\r*cos(\a)}) node[vx] (\j){} node[spv]{} -- (\jj.center);
			}
		\draw (2)--(6) node[vx,pos=0.18](r){} (8)--(4) node[vx,pos=0.18](l){} (r)--(l) (1)--(5) (3)--(7);
		\draw[Red1,  very thick] (1)--(2) (3)--(7) (1)--(5) (6)--(r) (4)--(l);
		\draw[line width=0.5] (1) circle(4pt) (8) circle(4pt);
		\path ($(1)+(0,4pt)$) node[sp]{};

		\foreach \a [count=\j, remember=\j as \jj (initially 1)] in {0,45,...,360}{
			\draw({3.7+\r*sin(\a)},{\r*cos(\a)}) node[vx](\j){} node[spv]{} -- (\jj.center);
			}
		\draw (2)--(6) node[vx,pos=0.18](r){} (8)--(4) node[vx,pos=0.18](l){} (r)--(l) (1)--(5) (3)--(7);
		\draw[Red1,  very thick] (l)--(r)--(6)--(5)--(1)--(2)--(3)--(7)--(8)--(l)--(4);
		\draw[line width=0.5] (l) circle(4pt) (4) circle(4pt);
	\end{tikzpicture}
	\caption{Representations of two infinite sets of balanced Petersen-subdivisions which fail to be uniformly optimal, contrary to a conjecture. Edges represent chains, where thick red chains are one edge longer than black. The different numbers of red chains at the encircled vertices shows that none of the graphs minimize the number of disconnecting sets of size~3.}
	\label{f.AthSobel}
\end{figure}

We also note that the conjectures given in~\cite{AthSobel} for uniformly most reliable $(n, n+6)$-graphs and $(n,n+7)$-graphs both fail in an infinite number of cases for the exact same reason. These conjectures concern subdivisions of two graphs known as the \emph{Yutsis 18j-symbol~F} (named as one of Adolfas Jucys's angular momentum diagrams), which is uniquely optimal among simple $(12,18)$-graphs~\cite{Rsurg}, and the \emph{Heawood graph} in~$\G_{14,21}$.

That the particulars fail for the conjecture in~\cite{AthSobel} does not guarantee that there are no other uniquely optimal subdivisions of~$P$ (or the Yutsis or the Heawood graph) for the given values of~$m$. However, as shown in Proposition~\ref{p.myrvold} below, there is a weak $P$‑subdivision in~$\G_{6,11}$ which is large-$p$-optimal but not small-$p$-optimal. In light of this, we would like to suggest the following weaker conjecture regarding the case $k=5$. (The second statement below follows from the first since $P$~is edge transitive.)

\begin{conj} \label{con.petersen}
	For fixed \(n\geq 1\) and \(m=n+5\), and for each~\(p\), every \(p\)‑optimal \((n,m)\)-graph is a balanced weak subdivision of the Petersen graph. In particular, a uniquely optimal graph exists for each $m=15q + r$ where $q\geq1$ and $r\in\{-1,0,1\}$.
\end{conj}

The following proposition mainly extends a~result by Myrvold~\cite{MyrvoldShort} to multigraphs. The paper describes a sequence of graph pairs, obtained by edge deletion in~$K_n$ for $n\geq 6$, which demonstrates that UMRGs do~not always exist among simple graphs. (While only the particular pair with exceedance~5 is explicitly covered below, our argument can be adapted to the subsequent graphs.) A half-overlapping similar sequence had already been found by Kelmans~\cite{Kelmans}, with a~focus on spanning trees and proving small-$p$-optimality rather than large-$p$-optimality. The crucial claim is in the proof of~\cite[Theorem~3.3]{Kelmans} and can be traced to~\cite{KelmansTrees}, where it seems to~us that the proof does not apply to multigraphs. 

\begin{prop}
	\label{p.myrvold} There is a unique large-\(p\)‑optimal \((6,11)\)-graph~\(G\) (shown in Figure~\ref{f.myrvold}) which is a weak subdivision of~\(P\), the Petersen graph. A~different weak $P$‑subdivision \(G'\in \G_{6,11}\) has a larger number of spanning trees, and is therefore strictly more reliable than~\(G\) for sufficiently small~\(p\). In particular, there is no uniquely optimal graph in~\(\G_{6,11}\).
\end{prop}

\begin{figure}[htb!]
	\centering
	\includegraphics{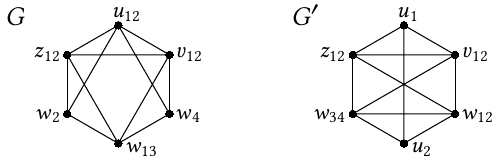}
	\caption{For the $(6,11)$-graphs, $p$-optimality depends upon $p$. The large-$p$-optimal graph is~$G$, while $G'$ is small-$p$-optimal among simple graphs and likely among multigraphs. Both are weak subdivisions of the Petersen graph (cf.\ Figure~\ref{f.petersen}).}
	\label{f.myrvold}
\end{figure}

\begin{proof}
	$G$ and $G'$, shown in Figure~\ref{f.myrvold}, are 3‑edge-connected $(6,11)$-graphs. The vertex labels indicate how the two graphs can be obtained from the Petersen graph in Figure~\ref{f.petersen} by edge contraction, which shows that they are balanced weak subdivisions of~$P$.
	
	By considering their complements, it is easy to see that $G$ and~$G'$ are the only two simple graphs with degree sequence $(4,4,\ab 4,\ab 4,\ab 3,3)$. $G$~has no 2‑bonds and two 3‑bonds (the trivial ones), and clearly, the degree sequence of~$G$ and~$G'$ is necessary to have no 2‑bonds and at most two 3‑bonds in~$\G_{6,11}$. Since $G$ furthermore has no 4‑bonds, while $G'$ has one (isolating its two adjacent 3‑vertices), $G$~alone minimizes $d_2(\ph)$, $d_3(\ph)$ and $d_4(\ph)$ among simple $(6,11)$-graphs.
	
	We now need to show that $\G_{6,11}$ has no large-$p$-optimal graph with a multiple edge. Since the degree sequence of~$G$ is necessary to be large-$p$-optimal, suppose that $G''\in \G_{6,11}$ has the same degree sequence and a multiple edge between $x_1$ and~$x_2$. Then, $x_1$~and~$x_2$ can be isolated from the other vertices by a nontrivial 2‑, 3‑ or 4‑bond. This implies that at least one of $d_2(\ph)$, $d_3(\ph)$ and $d_4(\ph)$ is not minimized by~$G''$, so that $R(G,p)>R(G'',p)$ for $p$ sufficiently large.
	
	The proof is finished by verifying that $t(G')>t(G)$ (see~\cite{MyrvoldShort} or use e.g.\ Kirchhoffs matrix tree theorem to obtain $t(G')=225$ and $t(G)=224$).
\end{proof}

\begin{rmk}
	MacAssey and Samaniego~\cite{McAsseySamaniego} studied the reliability of simple $(6,11)$-graphs under a different model: All edges have \emph{lifetimes} which are independent and identically distributed random variables. It was shown that, with $G$ and~$G'$ as in Figure~\ref{f.myrvold}, the probability that $G'$~fails before~$G$ exceeds~$1/2$. In this weaker sense, called \emph{stochastic precedence}, $G$~was shown to be the single most reliable simple $(6,11)$-graph. 
\end{rmk}

\section*{Acknowledgments}
The authors would like to thank Nathan Kahl and Kristi Luttrell for pointing out the connection between edge shifts and Whitney twists. The second author acknowledges the support of the Swedish Research Council, grant no.~2020-03763.

\setlength{\bibsep}{4pt plus 1pt minus 1pt}


\end{document}